\documentclass[10pt, letterpaper,reqno]{amsart}
\usepackage[left=1in,right=1in,bottom=0.5in,top=0.8in]{geometry}
\usepackage{amsfonts}
\usepackage{amsmath, amssymb}
\usepackage{graphicx}
\usepackage[font=small,labelfont=bf]{caption}
\usepackage{epstopdf}
\usepackage{xcolor}
\usepackage{amsthm}
\usepackage{float}
\usepackage{enumitem}
\usepackage{pgfplots}
\usepackage{listings}
\usepackage{longtable}
\usepackage{mathrsfs}
\usepackage{dsfont}
\usepackage{mathtools}
\usepackage{dsfont}
\usepackage{stmaryrd}
\usepackage[pdfpagelabels,hyperindex]{hyperref}
\hypersetup{linkbordercolor=green}

\newcommand*{\rom}[1]{\expandafter\@slowromancap\romannumeral #1@}
 \colorlet{lgray}{white!80!black}
\colorlet{lred}{white!85!red}
\colorlet{lgreen}{white!60!green}
\colorlet{dgreen}{black!30!green}
\colorlet{lpurple}{white!60!purple}
\colorlet{lblue}{white!60!blue}
\definecolor{green}{rgb}{0.1,0.8,0.1}
\definecolor{yellow}{rgb}{1.0,0.85,0.25}
\definecolor{purple}{rgb}{1.0, 0, 1.0}
\definecolor{blue}{rgb}{0, 0, 1.0}
\tikzstyle{unfused}=[lgray, line width=1.5pt, ->]
\tikzstyle{fused}=[lgray, line width=4pt, ->]
\tikzstyle{dual}=[black, line width=1pt, dashed]
\tikzstyle{lightdual}=[black, line width=0.5pt, dashed]
\tikzstyle{cut}=[black, line width=1.0pt]

\theoremstyle{plain}
\newtheorem{theorem}{Theorem}[section]
\newtheorem{lemma}[theorem]{Lemma}
\newtheorem{proposition}[theorem]{Proposition}

\theoremstyle{remark}
\newtheorem{definition}[theorem]{Definition}
\newtheorem{remark}[theorem]{Remark}

\colorlet{shadecolor}{gray!20}
\pgfplotsset{compat=1.9}
\usetikzlibrary{shapes.multipart}
\usetikzlibrary{patterns}
\usetikzlibrary{shapes.multipart}
\usetikzlibrary{arrows}
\usetikzlibrary{decorations.markings}
\usepgflibrary{decorations.shapes}
\usetikzlibrary{decorations.shapes}
\usepgflibrary{shapes.symbols}
\usetikzlibrary{shapes.symbols}
\usetikzlibrary{decorations.pathreplacing}
\tikzstyle{fleche}=[>=stealth', postaction={decorate}, thick]
\tikzstyle{axis}=[->, >=stealth', thick, gray]
\tikzstyle{paths}=[>->, >=stealth', thick]
\tikzstyle{path}=[->, >=stealth', thick]
\tikzstyle{grille}=[dotted, gray]
\DeclareFontFamily{U}{mathx}{}
\DeclareFontShape{U}{mathx}{m}{n}{<-> mathx10}{}
\DeclareSymbolFont{mathx}{U}{mathx}{m}{n}
\DeclareMathAccent{\widehat}{0}{mathx}{"70}
\DeclareMathAccent{\widecheck}{0}{mathx}{"71}

\usepackage{romannum}
\usepackage{upgreek}

\usepackage[scr=boondoxo]{mathalfa}

\newcommand{\ber}{\mathrm{Ber}}

\newcommand{\bg}{\boldsymbol{\gamma}}
 
\newcommand{\exc}{\mathsf{exc}}
\newcommand{\mea}{\mathsf{mea}}
\newcommand{\revmea}{\mathsf{revmea}}

\renewcommand{\a}{\mathsf{a}}
\renewcommand{\b}{\mathsf{b}}
\renewcommand{\c}{\mathsf{c}}
\renewcommand{\d}{\mathsf{d}}

\numberwithin{equation}{section}

\newcommand{\TV}[1]{{\left\lVert #1 \right\rVert}_{\normalfont
\text{TV}}}

\usepgflibrary{fpu}
\makeatletter
\let\NAT@parse\undefined
\makeatother

\author{Dominik Schmid}
\address
{
Dominik Schmid, 
University of Augsburg, Germany 
}
\email{d.schmid@uni-a.de} 
\author{Zongrui Yang}
\address
{
Zongrui Yang, 
University of California, Berkeley
}
\email{zongrui@berkeley.edu}

\title[Approximation for the stationary measure of the open ASEP]{Sharp approximations for the stationary measures of the open ASEP III: The fan region}  
\begin{document}
\pagenumbering{arabic}

\begin{abstract} 

We study the stationary measure of the open asymmetric simple exclusion process (ASEP) in the fan region with general boundary rates. 
The stationary height fluctuations in the maximal current phase and on its boundaries within the fan region are known to have limits involving Brownian excursions or meanders. We establish discrete counterparts of these limits: on intervals away from the boundaries, the joint law of the occupation variables is approximated in total variation distance by particle configurations constructed from lazy simple random walk excursions and meanders. In the high- and low-density phases of the fan region, we establish Bernoulli product approximations with errors that decay exponentially in the distance from the relevant boundary. Our methods combine refined combinatorial estimates for Motzkin paths and the matrix product ansatz in the totally asymmetric case with a Fourier-type decomposition (based on Askey--Wilson polynomials) of the change of measure for general asymmetry. In addition, in the maximal current phase and on its boundaries within the fan region, we prove moderate deviation estimates for particle counts on arbitrary sets of sites. These estimates rely on a covariance identity for an auxiliary path measure, through which total nonnegativity (of order two) of the matrices yields negative correlations.
\end{abstract}
\maketitle

%
\section{Introduction and main results}

\subsection{Preface}

The asymmetric simple exclusion process (ASEP) is a central model in probability and non-equilibrium statistical physics. Introduced in the biology literature in 1968 \cite{macdonald1968kinetics} and in the mathematical literature in 1970 \cite{SPITZER1970246}, it has become a prototypical interacting particle system, exhibiting phase transitions and shock formation \cite{liggett1975ergodic,derrida1993exact,FF:ShockFluctuations}. We refer the reader to \cite{liggett1999stochastic,wood2020combinatorial,blythe2007nonequilibrium,corwin2022some} for a selection of expository works from probabilistic, combinatorial, and statistical mechanical perspectives.

The present paper is the third in a series of three on sharp approximations of the stationary measure of the open ASEP. In this model, particles perform biased random walks on a finite segment, subject to the exclusion rule, and may enter or exit at the endpoints at prescribed rates. The boundary rates govern a rich phase diagram, reflected in the structure of the stationary measure; see, for example,  \cite{liggett1975ergodic,derrida1993exact,uchiyama2004asymmetric,Bryc2017,wang2023askey}. The parameter space of the open ASEP is divided into the fan and shock regions, separated by a common boundary. We focus on the fan region, where the effective density of the left reservoir exceeds that of the right reservoir; see Section~\ref{subsec:model and results} for a precise definition. This complements our companion papers \cite{SchmidYangShock,SchmidYangCoexistence}, which treat the shock region. Here, our results allow general boundary rates in the fan region and do not require Liggett's condition.

The stationary height fluctuations of the open ASEP have different limits involving Brownian motion, excursions, and meanders, depending on the phase. In the maximal current phase, the limiting height fluctuations involve the sum of a Brownian motion and an independent Brownian excursion; on the boundaries of the maximal current phase within the fan region, the excursion is replaced by a Brownian meander or its time reversal. These limits were established in \cite{bryc2019limit} using the connection with Askey--Wilson processes from \cite{bryc2010askey,Bryc2017}; see also the earlier work \cite{derrida2004asymmetric} for the excursion limit in the totally asymmetric case. Our goal is to explain these structures at the discrete level through quantitative total variation approximations of the stationary measure itself.

Our main results approximate the joint law of the occupation variables on intervals away from the relevant boundaries in total variation distance by configurations constructed from conditioned lazy simple random walks. In the maximal current phase, the reference path is an excursion: it starts at the origin, stays nonnegative, and returns to the origin after traversing the segment. On one boundary of the maximal current phase within the fan region, the reference path is a meander, conditioned only to stay nonnegative; the construction for the other boundary is obtained by spatial reversal and particle–hole exchange. Upward and downward steps specify occupied and vacant sites, respectively, while independent fair coin flips determine the occupations corresponding to horizontal steps. These constructions provide discrete pre-limiting counterparts of the Brownian fluctuation limits. We also establish Bernoulli product approximations in the high- and low-density phases of the fan region, with errors exponentially small in the distance from the left and right boundary, respectively.
The approximation errors are controlled by the distance of the observed interval from the relevant boundaries. Consequently, the intervals may occupy almost the entire segment, provided these distances tend to infinity. These results complement and, in the overlapping cases, strengthen earlier bulk and boundary approximation results \cite{nestoridi2024approximating,yang2025limits}.

Our proofs combine probabilistic and analytic methods. Starting from the matrix product ansatz, we represent stationary configurations through weighted bicolored Motzkin paths. In the totally asymmetric case, refined combinatorial estimates on these paths control the stationary measure. For general asymmetry, we introduce a Fourier-type decomposition based on Askey--Wilson orthogonal polynomials to analyze the change of measure and compare the weighted path laws with their totally asymmetric counterparts. We give precise statements of the main results in Section~\ref{sec:MainResults} and a more detailed account of the proof strategy in Section~\ref{sec:Overview}.

Finally, we prove moderate deviation estimates for stationary particle counts on arbitrary sets of sites in the maximal current phase and on its boundaries within the fan region. These estimates require an additional argument beyond the total variation approximations. We derive a determinantal formula for covariances of the occupation variables under suitable tilts of an auxiliary path measure. Using log-concavity, we establish total nonnegativity of order two for the matrices involved in this formula. This leads to negative correlations and hence the required concentration bounds.

\subsection{The open ASEP}\label{subsec:model and results}

We define the open ASEP as an interacting particle systems on a finite state space. More precisely, the open ASEP is a continuous-time Markov chain $(\tau(t))_{t \geq 0}$ on the state space $\{0,1\}^{N}$  with respect to parameters
\begin{equation}\label{eq:conditions open ASEP}
\alpha,\beta>0,\quad \gamma,\delta\geq 0,\quad\text{and}\quad  0\leq q<1 . 
\end{equation}
We write $\tau=(\tau_1,\tau_2,\dots,\tau_N) \in \{0,1\}^N$ for the configurations of the open ASEP, where $\tau_i$ denotes the indicator whether the $i^{\text{th}}$ site is occupied in $\tau$. Write $\llbracket a,b\rrbracket := [a,b] \cap \mathbb Z$ for all $a<b$.
In the dynamics, on each site $\llbracket 2,N\rrbracket$, particles attempt jumps to the left according to rate $q$ Poisson clocks, and on the sites $\llbracket 1,N-1\rrbracket$ to the right according to rate $1$ Poisson clocks, independently of all other particles. However, a move is performed if and only if the target is a vacant site.  At the two boundaries, particles can enter and exit the segment. More precisely, at the left boundary, particles attempt to enter the system from the
reservoir at rate $\alpha$ and attempt to exit at rate $\gamma$. Similarly, at the right boundary, particles
attempt to enter at rate $\delta$ and attempt to exit at rate $\beta$; see Figure~\ref{fig:openASEP} for an illustration. The assumption $q<1$ ensures that particles have a  higher rate of moving to the right.  The special case $q=\gamma=\delta=0$, where particles can only move to the right, is called the open totally asymmetric simple exclusion process, or \textit{open TASEP}. 
It will be convenient to consider the boundary parameters
\begin{equation}\label{eq:defining ABCD}
\a=\phi_+(\alpha,\gamma),\quad\b=\phi_+(\beta,\delta),\quad \mathsf{c}=\phi_-(\alpha,\gamma), \quad \mathsf{d}=\phi_-(\beta,\delta),
\end{equation}
where we set
\begin{equation} \label{eq:phi}
\phi_{\pm}(x,y):=\frac{1}{2x}\left(1-q-x+y\pm\sqrt{(1-q-x+y)^2+4xy}\right), \quad \mbox{for }\; x>0\mbox{ and }y\geq 0.
\end{equation}
In particular, note that under our assumptions \eqref{eq:conditions open ASEP}, 
\begin{equation}\label{eq:conditions qABCD}
\a,\b\geq0,\quad -1<\mathsf{c},\mathsf{d}\leq 0 . 
\end{equation}
We define $\rho_{\textup{L}}:=(1+\a)^{-1}$, respectively $\rho_{\textup{R}}:=\b(1+\b)^{-1}$ as the left side, respectively right side \textit{effective densities}. The parameters $\a$ and $\b$ allow for a simple partition of the open ASEP into different regimes. When $\a > \max(\b,1)$, we say that the open ASEP belongs to the \textit{low density phase}, for $\b > \max(\a,1)$ the open ASEP belongs to the \textit{high density phase}, and for $\a,\b \leq 1$ to the \textit{maximal current phase}. We refer to the special case $\a=\b>1$ as the \textit{coexistence line}. Moreover, we say that the open ASEP belongs to the \textit{fan region} if $\a \b < 1$, and to the \textit{shock region} if $\a \b >1$; see Figure~\ref{fig:Phase} for a visualization.  We provide a more thorough discussion of the role of the different phases in the subsequent subsection.

\begin{figure}
\centering
\begin{tikzpicture}[scale=0.85]
\def\spiral[#1](#2)(#3:#4:#5){
\pgfmathsetmacro{\domain}{pi*#3/180+#4*2*pi}
\draw [#1,
       shift={(#2)},
       domain=0:\domain,
       variable=\t,
       smooth,
       samples=int(\domain/0.08)] plot ({\t r}: {#5*\t/\domain})
}
\def\particles(#1)(#2){
  \draw[black,thick](-3.9+#1,0.55-0.075+#2) -- (-4.9+#1,0.55-0.075+#2) -- (-4.9+#1,-0.4-0.075+#2) -- (-3.9+#1,-0.4-0.075+#2) -- (-3.9+#1,0.55-0.075+#2);
\node[shape=circle,scale=0.6,fill=red] (Y1) at (-4.15+#1,0.2-0.075+#2) {};
  	\node[shape=circle,scale=0.6,fill=red] (Y2) at (-4.6+#1,0.35-0.075+#2) {};
  	\node[shape=circle,scale=0.6,fill=red] (Y3) at (-4.2+#1,-0.2-0.075+#2) {};
   	\node[shape=circle,scale=0.6,fill=red] (Y4) at (-4.45+#1,0.05-0.075+#2) {};
  	\node[shape=circle,scale=0.6,fill=red] (Y5) at (-4.65+#1,-0.15-0.075+#2) {}; }
\def\annhil(#1)(#2){	  \spiral[black,thick](9.0+#1,0.09+#2)(0:3:0.42);
\draw[black,thick](8.5+#1,0.55+#2) -- (9.5+#1,0.55+#2) -- (9.5+#1,-0.4+#2) -- (8.5+#1,-0.4+#2) -- (8.5+#1,0.55+#2); }
\node[shape=circle,scale=1.5,draw] (B) at (2.3,0){} ;
	\node[shape=circle,scale=1.5,draw] (C) at (4.6,0) {};
	\node[shape=circle,scale=1.2,fill=red] (CB) at (2.3,0) {};
    \node[shape=circle,scale=1.5,draw] (A) at (0,0){} ;
 	\node[shape=circle,scale=1.5,draw] (D) at (6.9,0){} ; 
 	 	\node[shape=circle,scale=1.5,draw] (Z) at (-2.3,0){} ;
	\node[shape=circle,scale=1.2,fill=red] (YZ) at (-2.3,0) {};
   \node[line width=0pt,shape=circle,scale=1.6] (B2) at (2.3,0){};
	\node[line width=0pt,shape=circle,scale=2.5] (D2) at (6.9,0){};
		\node[line width=0pt,shape=circle,scale=2.5] (Z2) at (-2.3,0){};
		
	\node[line width=0pt,shape=circle,scale=2.5] (X10) at (6.8,0){};
		\node[line width=0pt,shape=circle,scale=2.5] (X11) at (-2.2,0){};	
			\node[line width=0pt,shape=circle,scale=2.5] (X12) at (8.4,0){};
		\node[line width=0pt,shape=circle,scale=2.5] (X13) at (-3.8,0){};		
\draw[thick] (Z) to (A);	
	\draw[thick] (A) to (B);	
		\draw[thick] (B) to (C);	
  \draw[thick] (C) to (D);

\particles(0)(0);
\particles(6.9+4.9+1.6)(0);


\draw [->,line width=1pt]  (B2) to [bend right,in=135,out=45] (C);
 \draw [->,line width=1pt] (6.9+1.6,-0.475) to [bend right,in=135,out=45] (D);	
   \draw [->,line width=1pt] (B2) to [bend right,in=-135,out=-45] (A);
     \draw [->,line width=1pt] (Z) to [bend right,in=135,out=45] (-3.9,-0.475);
    \node (text1) at (3.5,1){$1$} ;    
	\node (text2) at (1.1,1){$q$} ;   
	\node (text3) at (-2.3-1,1){$\alpha$}; 
	\node (text4) at (6.9+1,1){$\beta$};   
    \node (text3) at (-2.3-1,-1){$\gamma$}; 
	\node (text4) at (6.9+1,-1){$\delta$};   

\node[scale=0.9] (textL) at (-5.2,0){$\rho_{L}$} ;
\node[scale=0.9] (textR) at (9.9,0){$\rho_{R}$} ;

    \node[scale=0.9] (text1) at (-2.2,-0.7){$1$} ;    
    \node[scale=0.9] (text1) at (6.8,-0.7){$N$} ;   
  	
  \draw [->,line width=1pt] (-3.9,0.475) to [bend right,in=135,out=45] (Z);
   \draw [->,line width=1pt] (D) to [bend right,in=135,out=45] (6.9+1.6,0.475);

	\end{tikzpicture}	
\caption{\label{fig:openASEP} Visualization of the open ASEP for parameters $(q,\alpha,\beta,\gamma,\delta)$  on a segment of length $N$.}
\end{figure}
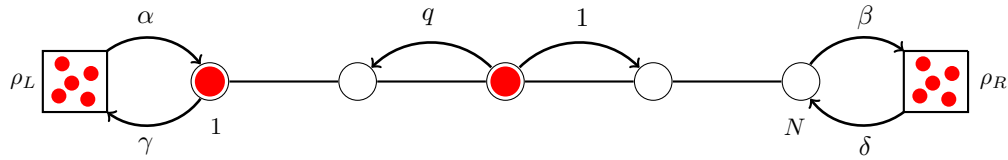

Note that for all $\alpha,\beta>0$, the open ASEP on a segment of length $N$ has a unique stationary measure $\mu_N=\mu^{(q,\a,\b,\c,\d)}_N$ for all $n\in \mathbb N$, where $\a,\b,\c,\d$ are given by \eqref{eq:defining ABCD}. It is a straightforward computation to check that for $\a \b =1$, we have that $\rho_L=\rho_R$ and that the invariant measure $\mu_N$ is a homogeneous Bernoulli product measure with parameter $\rho_L$ on $\{0,1\}^{N}$. Our goal is to show that for general boundary parameters, the stationary measure $\mu_N$ can be well-approximated by simpler measures, which we now define. 
\begin{definition}\label{def:exc and mea}
Fix $N\in\mathbb N$. We define probability measures $\mathbb{P}_N^{\exc}$, $\mathbb{P}_N^{\mea}$, and $\mathbb{P}_N^{\revmea}$ on $\{0,1\}^N$ as follows. Suppose $\nu$ is the probability measure on $\{-1,0,1\}$ given by $\nu(-1)=\nu(1)=\frac{1}{4}$ and $\nu(0)=\frac{1}{2}$. Consider the $\nu$-random walk $S_i$, $i\in\llbracket0,N\rrbracket$ started at $S_0=0$.  We define the process $S^{\exc}_i$, $i\in\llbracket0,N\rrbracket$ as $S_i$ conditioned on  $\{S_i\geq0\mbox{ for }i\in\llbracket0,N\rrbracket, S_N=0\}$; and the process $S^{\mea}_i$, $i\in\llbracket0,N\rrbracket$ as $S_i$ conditioned on  $\{S_i\geq0\mbox{ for }i\in\llbracket0,N\rrbracket\}$. Suppose $\kappa_i$, $i\in\llbracket1,N\rrbracket$ are i.i.d. $\ber(\frac{1}{2})$ random variables independent from $S^{\exc}_i$ and $S^{\mea}_i$. Denote  
\[
\tau_i^{\star}={\bf1}_{\{S^{\star}_i-S^{\star}_{i-1}=1\}}+\kappa_i{\bf1}_{\{S^{\star}_i-S^{\star}_{i-1}=0\}}
\qquad\mbox{ for } \star\in\{\exc,\mea\},\,\mbox{ }i\in\llbracket1,N\rrbracket,
\] 
and denote $\tau^{\revmea}_i=1-\tau_{N+1-i}^{\mea}$ for $i\in\llbracket1,N\rrbracket$.
For each $\star\in\{\exc,\mea,\revmea\}$, we write $\mathbb{P}_N^\star$ to denote the measure on $\{0,1\}^N$ of the random vector $(\tau_i^\star)_{i\in\llbracket1,N\rrbracket}$. 
\end{definition}
In words, we obtain the measures $\mathbb{P}_N^\mea$ on $\{0,1\}^N$ by considering a lazy simple random walk conditioned to stay non-negative, and then assigning a particle for every up step, an empty site for every down step, and placing a particle according to the outcome of an independent fair coin flip for every horizontal step. The construction of the measure  $\mathbb{P}_N^\exc$ is analogous, considering a a lazy simple random walk conditioned to stay non-negative and to return to the origin after $N$ steps. 
Intuitively, the measures $\mathbb{P}_N^\star$ for $\star\in\{\exc,\mea,\revmea\}$ can be seen as the discrete analogues of the (renormalized) sum of a standard Brownian motion and  
a Brownian excursion of length $N$, a Brownian meander, and a time-reversed Brownian meander, respectively.

\subsection{Statement of the main results}\label{sec:MainResults}

Let us now state the main results on approximating the invariant measure of the open ASEP using the above measures. 
We start with the stationary measure $\mu_N^{(0,\a,\b,\c,\d)}$ of the open TASEP with parameters $\a,\b,\c,\d$ (and $q=\gamma=\delta=0)$. 
We have the following special case, which was already observed in \cite{derrida2004asymmetric}, and for which we give a short proof at the end of Section~\ref{sec:MPAs}.
\begin{proposition}\label{prop:SpecialExcursion} In the special case of the open TASEP with $\alpha=\beta=1$ and $\gamma=\delta=q=0$, we have that $\mu_N^{(0,0,0,0,0)}=\mathbb{P}_N^{\exc}$ for all $N\in \mathbb N$. 
\end{proposition} The exact representation in Proposition~\ref{prop:SpecialExcursion} motivates the following approximation result for the entire maximal current phase of the open TASEP. As a measure of distance, we consider for two probability measures $\nu,\nu'$ on a common probability space $\Omega$ the total variation distance
\begin{equation}\label{def:TVdistance}
    \TV{\nu - \nu'} := \frac{1}{2}\sum_{\omega \in \Omega} | \nu(\omega) - \nu'(\omega)| .  
\end{equation}

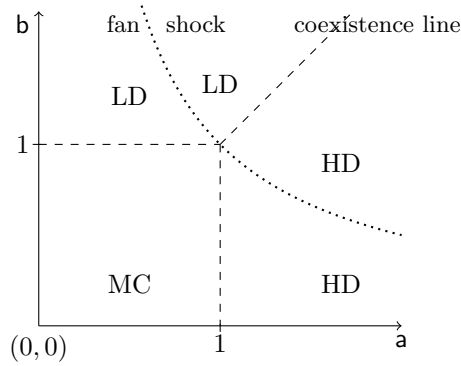
\begin{figure}
    \centering
    \begin{tikzpicture}[scale=0.8]
 \draw[scale = 1,domain=6.7:11,smooth,variable=\x,dotted,thick] plot ({\x},{1/((\x-7)*1/3+2/3)*3+5});
 \draw[->] (5,5) to (5,10.2);
 \draw[->] (5.,5) to (11,5);
   \draw[dashed] (5,8) to (8,8);
  \draw[dashed] (8,8) to (8,5);
   \draw[dashed] (8,8) to (10.2,10.2);
   \node [left] at (5,8) {$1$};
   \node[below] at (8,5) {$1$};
     \node [below] at (11,5) {$\a$};
   \node [left] at (5,10) {$\b$};
 \draw[dashed] (8,4.9) to (8,5.1);
  \draw[dashed] (4.9,8) to (5.1,8);
 \node [below] at (5,5) {$(0,0)$};
    \node [above] at (6.5,8.5) {LD};\node [above] at (8,8.7) {LD};
    \node at (10.6,10) {\small coexistence line }; 
    \node [below] at (10,6) {HD};  \node [below] at (10,8) {HD}; 
 \node [below] at (6.5,6) {MC}; 
 \node at (6.4,10) {\small{fan}};
 \node at (7.6,10) {\small{shock}};
 \end{tikzpicture} 
        \caption{    \label{fig:Phase} Visualization of the different regimes: Low density (LD), high density (HD), and maximal current (MC) phase. Moreover, the curve $\a \b =1$ separates the fan and shock region, while the case $\a=\b>1$ is called the coexistence line. }
\end{figure}

\begin{theorem}\label{thm:main open TASEP}
    Assume $\a,\b\geq0$ and $\mathsf{c},\mathsf{d}\in(-1,0]$. There exists a constant $C<\infty$ depending only on $\a,\b,\c$ and $\d$, such that the following hold.
    \begin{enumerate}
        \item [(1)] In the maximal current phase $\a,\b\in[0,1)$, for every $N,m,k\in\mathbb N$ satisfying 
        $1\le m\le N-k+1\le N$,  
\[
\TV{
\mu_N^{(0,\a,\b,\c,\d)}
\left(
(\tau_i)_{i\in\llbracket m,N-k+1\rrbracket}\in\cdot
\right)
-
\mathbb{P}_N^{\mathsf{exc}}
\left(
(\tau_i)_{i\in\llbracket m,N-k+1\rrbracket}\in\cdot
\right)}
\le
C/\min(m,k) .
\]
In two special cases in the maximal current phase, slightly stronger approximation results hold.  
Specifically, in the case when $\a\in[0,1)$ and $\b=\d=0$, for every $N\in\mathbb N$ and $m\in\llbracket1,N\rrbracket$, 
\[
\TV{
\mu_N^{(0,\a,0,\c,0)}
\left(
(\tau_i)_{i\in\llbracket m,N \rrbracket}\in\cdot
\right)
-
\mathbb{P}_N^{\mathsf{exc}}
\left(
(\tau_i)_{i\in\llbracket m,N \rrbracket}\in\cdot
\right)}
\le
C/m ;
\]
 and in the case when  $\a=\c=0$ and $\b\in[0,1)$, for every $N\in\mathbb N$ and $k\in\llbracket1,N\rrbracket$, 
\[
\TV{
\mu_N^{(0,0,\b,0,\d)}
\left(
(\tau_i)_{i\in\llbracket 1,N-k+1\rrbracket}\in\cdot
\right)
-
\mathbb{P}_N^{\mathsf{exc}}
\left(
(\tau_i)_{i\in\llbracket 1,N-k+1\rrbracket}\in\cdot
\right)}
\le
C/k .
\]
        \item [(2)] On the boundary of the maximal current phase $\a\in[0,1)$ and $\b=1$, for every $N\in\mathbb N$ and $m\in\llbracket1,N\rrbracket$, 
        \[
\TV{
\mu_N^{(0,\a,1,\c,\d)}
\left(
(\tau_i)_{i\in\llbracket m,N \rrbracket}\in\cdot
\right)
-
\mathbb P_N^{\mathsf{mea}}
\left(
(\tau_i)_{i\in\llbracket m,N \rrbracket}\in\cdot
\right)}
\le
C/m .
\]
        \item [(3)] On the boundary of the maximal current phase $\a=1$ and $\b\in[0,1)$, for every $N\in\mathbb N$ and $k\in\llbracket1,N\rrbracket$, 
        \[
\TV{
\mu_N^{(0,1,\b,\c,\d)}
\left(
(\tau_i)_{i\in\llbracket 1,N-k+1 \rrbracket}\in\cdot
\right)
-
\mathbb P_N^{\mathsf{revmea}}
\left(
(\tau_i)_{i\in\llbracket 1,N-k+1 \rrbracket}\in\cdot
\right)}
\le
C/k .
\]
    \end{enumerate} 
\end{theorem}

As the next theorem shows, the approximation of the stationary measure in Theorem~\ref{thm:main open ASEP} is sharp.

\begin{theorem}\label{thm:converse TASEP} 
For each $x,y\in\mathbb R\setminus\{1\}$, we denote $\eta(x,y)=x(1-x)^{-2}+y(1-y)^{-2}$. Assume $\a,\b\in[0,1)$ and $\c,\d\in(-1,0]$ satisfying  $\eta(\a,\c)\eta(\b,\d)\neq0$. Then there exists a constant $C>0$ depending only on $\a,\b,\c,\d$, such that for every $m,k\in\mathbb N$ there exists an integer $N_0\geq m+k-1$ such that for each $N\geq N_0$,
            \[
\TV{
\mu_{N}^{(0,\a,\b,\c,\d)}
\left(
(\tau_i)_{i\in\llbracket m,N-k+1\rrbracket}\in\cdot
\right)
-
\mathbb{P}_{N}^{\mathsf{exc}}
\left(
(\tau_i)_{i\in\llbracket m,N-k+1\rrbracket}\in\cdot
\right)}
\ge
C/\min(m,k) .
\] 
\end{theorem}

Let us stress that $\eta(\a,\c)\neq0$ is not a purely technical assumption, as it can be seen from Proposition~\ref{prop:SpecialExcursion} where $a=c=0$ and so the measure $\mathbb{P}_N^{\exc}$ exactly describes the stationary measure of the open TASEP. However, we leave a lower bound in the case $\eta(\a,\c)=0$ for $\a,\c \neq 0$ to future investigations. 

Next, we provide a similar quantitative approximation result on the stationary measure of the open ASEP in the maximal current phase for general boundary parameters.

\begin{theorem}\label{thm:main open ASEP}
    Assume $q\in[0,1)$, $\a,\b\geq0$ and $\mathsf{c},\mathsf{d}\in(-1,0]$. For $N\in\mathbb N$, consider the stationary measure $\mu_N^{(q,\a,\b,\c,\d)}$ of open ASEP. There exists $C<\infty$ depending only on $\a,\b,\c,\d$ and $q$, such that the following hold.
    \begin{enumerate}
        \item [(1)] In the maximal current phase $\a,\b\in[0,1)$, for every $N,m,k\in\mathbb N$ satisfying 
        $1\le m\le N-k+1\le N$,  
\[
\TV{
\mu_N^{(q,\a,\b,\c,\d)}
\left(
(\tau_i)_{i\in\llbracket m,N-k+1\rrbracket}\in\cdot
\right)
-
\mathbb{P}_N^{\mathsf{exc}}
\left(
(\tau_i)_{i\in\llbracket m,N-k+1\rrbracket}\in\cdot
\right)}
\le
C/\min(m,k)^{1/2} .
\]
        \item [(2)] On the boundary of the maximal current phase $\a\in[0,1)$ and $\b=1$, for every $N\in\mathbb N$ and $m\in\llbracket1,N\rrbracket$, 
        \[
\TV{
\mu_N^{(q,\a,1,\c,\d)}
\left(
(\tau_i)_{i\in\llbracket m,N \rrbracket}\in\cdot
\right)
-
\mathbb P_N^{\mathsf{mea}}
\left(
(\tau_i)_{i\in\llbracket m,N \rrbracket}\in\cdot
\right)}
\le
C/m^{1/2} .
\]
        \item [(3)] On the boundary of the maximal current phase $\a=1$ and $\b\in[0,1)$, for every $N\in\mathbb N$ and $k\in\llbracket1,N\rrbracket$, 
        \[
\TV{
\mu_N^{(q,1,\b,\c,\d)}
\left(
(\tau_i)_{i\in\llbracket 1,N-k+1 \rrbracket}\in\cdot
\right)
-
\mathbb P_N^{\mathsf{revmea}}
\left(
(\tau_i)_{i\in\llbracket 1,N-k+1 \rrbracket}\in\cdot
\right)}
\le
C/k^{1/2} .
\]
    \end{enumerate} 
\end{theorem}

While the lower bounds of order $\min(m,k)^{-1}$ from Theorem~\ref{thm:converse TASEP} can be extended to the open ASEP in the maximal current phase, it is an open question to determine whether the upper bounds in Theorem~\ref{thm:main open ASEP} of order $\min(m,k)^{-1/2}$ are sharp. Let us mention that the above theorem extends the results in \cite{nestoridi2024approximating}, where a (non-quantitative) total variation approximation by homogeneous Bernoulli product measures was achieved sufficiently far away from the boundaries for the three-parameter open ASEP, i.e. when $\gamma=\delta=0$.

While our focus lies in the following on establishing the above approximation results for the open ASEP in the maximal current phase, the same techniques allow also to approximate the stationary measure of the open ASEP in the high and low density phase of the fan region. To this end, for a finite set \(I\subseteq\mathbb N\) and \(p\in[0,1]\), we write $\ber_I(p)$ to denote Bernoulli product measure on $\{0,1\}^I$ with density $p$. Specifically, set
\begin{equation*}
    \ber_I(p)=\otimes_{i\in I} ((1-p)\delta_0+p\delta_1 ) , 
\end{equation*}
where \(\delta_x\) denotes the unit point mass at \(x\).
\begin{theorem}\label{thm:main open ASEP HLD}     Assume $q\in[0,1)$, $\a,\b\geq0$ and $\mathsf{c},\mathsf{d}\in(-1,0]$. For $N\in\mathbb N$, recall the stationary measure $\mu_N^{(q,\a,\b,\c,\d)}$. There exist $C,c\in(0,\infty)$ depending only on $\a,\b,\c,\d,q$, such that the following hold.
    \begin{enumerate}
        \item [(1)] In the high density phase of the fan region, i.e., $\b>1$ and $\a\b<1$,  for every $N\in\mathbb N$ and $m\in\llbracket1,N\rrbracket$, 
\begin{equation}\label{eq:HDbound1}
    \TV{ \mu_N^{(q,\a,\b,\c,\d)}
    \left( (\tau_i)_{i\in\llbracket m,N\rrbracket}\in\cdot\,\right)
    -\ber_{\llbracket m,N\rrbracket}\left( \b/(1+\b)\right)}
    \le C e^{-cm}.
\end{equation}
                \item [(2)] In the low density phase of the fan region, i.e., $\a>1$ and $\a\b<1$,  for every $N\in\mathbb N$ and $k\in\llbracket1,N\rrbracket$, 
               \begin{equation}\label{eq:HDbound2}
                    \TV{ \mu_N^{(q,\a,\b,\c,\d)}
    \left( (\tau_i)_{i\in\llbracket 1,N-k+1\rrbracket}\in\cdot\,\right)
    -\ber_{\llbracket 1,N-k+1\rrbracket}\left( 1/(1+\a)\right)}
    \le C e^{-ck}.
               \end{equation} 
\end{enumerate}
In \eqref{eq:HDbound1}, the constant \(c\) depends only on \(q,\b\), while in \eqref{eq:HDbound2}, it may be chosen to depend only on \(q,\a\).
\end{theorem}

Again, let us mention that Theorem~\ref{thm:main open ASEP HLD} extends the results in \cite{nestoridi2024approximating} for general boundary parameters.

\subsection{Moderate deviation estimates } 

Next, we have the following moderate deviation estimate on the number of particles under the stationary measure for the open ASEP in the maximal current phase. 

\begin{theorem}\label{thm:moderate deviation MC}
Assume $q\in[0,1)$, $\a,\b\in[0,1]$, $\a\b<1$, and $\c,\d\in(-1,0]$. 
For \(N\in\mathbb N\), let \(\mu_N^{(q,\a,\b,\c,\d)}\) be the stationary measure of the open ASEP on the lattice $\llbracket1,N\rrbracket$, and let \(\tau_i\in\{0,1\}\) denote the occupation variable at
site \(i\). There exists a constant \(C\in(0,\infty)\), depending only on \(q,\a,\b,\c,\d\), such that, for every \(N\in\mathbb N\), every nonempty set
\(I\subseteq\llbracket1,N\rrbracket\), and every \(x\geq0\),
\[ 
 \mu_N^{(q,\a,\b,\c,\d)}
 \left(
 \left|\sum_{i\in I}\tau_i-\frac{|I|}{2}\right|
 \geq x\sqrt{|I|}
 \right)
 \leq C e^{-x^2/2}.
\]
\end{theorem}

Let us stress that this is not a simple consequence of the above approximation results as we allow for general $x>0$, as well as general sets $I$. Instead, we require a novel type of negative correlation estimates for the stationary measure of a path measure closely related to the open TASEP; see Section~\ref{sec:LogConcave}.

\subsection{Related work}

The stationary measure of the open ASEP is intensively investigated over the past decades. A first study dates back to Liggett in \cite{liggett1975ergodic}, who identified the different phases depending on the boundary parameters. In \cite{derrida1993exact}, Derrida, Evans, Hakim and Pasquier use the celebrated matrix product ansatz to express the stationary measure of the open TASEP; see also \cite{blythe2007nonequilibrium} for an introduction to this technique. A solution to matrix product ansatz for the open ASEP with general boundary parameters was found by Uchiyama, Sasamoto and Wadati in \cite{uchiyama2004asymmetric}, establishing a connection between the open ASEP and Askey--Wilson orthogonal polynomials from \cite{askey1985some}. This led to a series of works \cite{Bryc2017,bryc2019limit,wang2023askey} studying the fluctuations, and more generally the Laplace transform of the height function representation, under the stationary measures of the open ASEP using Askey--Wilson processes and Askey--Wilson (signed) measures, which are defined using Askey--Wilson orthogonal polynomials. 
Let us mention that the matrix product ansatz representation is in general not unique and closely related to combinatorial objects; see for example  \cite{corteel2011tableaux} for a staircase tableaux representation and \cite{duchi2005combinatorial} for a connection between the stationary open TASEP and Catalan numbers. 

In \cite{enaud2004large}, Enaud and Derrida provide a solution to the matrix product ansatz, where the stationary measure of the open ASEP can be represented as the sum of two random walks, and which can in return be interpreted as a measure on the space of weighted bi-colored Motzkin paths; see also \cite{hegde2024large}.  
A similar combinatorial description of the stationary measure using bi-colored Motzkin paths was given in \cite{brak2006combinatorial} by Brak et al. for the three-parameter open ASEP where $\gamma=\delta=0$. More generally, the matrix product ansatz can be used to study the fluctuations of Motzkin paths with weights; see  \cite{BrycWang2019Motzkin,BrycWang2024MotzkinII,BrycKuznetsovWesolowski2025Motzkin}. Let us mention that the matrix product ansatz can also be used to study the stationary measure of the ASEP on the semi-infinite lattice \cite{grosskinsky2004phase,SasamotoWilliams2014,yang2025limits}, and for the open ASEP with the presence of light particles \cite{Uchiyama2008,SY:openLight}.

Recently, in a series of papers, Bryc et al. investigate the two-layer representation of the open TASEP and the open ASEP \cite{bryc2024two,bryc2024stationary,bryc2024limitCoex}, introduced first for integrable models on the strip in \cite{barraquand2024stationary}. This allows for a simple description of the stationary measure of the open ASEP (assuming Liggett's condition when $q\neq 0$), and hence a way to access for example the limiting density fluctuations or large deviations for the number of particles under the stationary measure; see e.g.~\cite{hegde2024large,SchmidYangShock,SchmidYangCoexistence}.  

In a broader context, \cite{gabrielli2026mixtures} shows that the measures representable by the matrix product ansatz (under certain non-negativity assumptions) can be written as a mixture of inhomogeneous product measures where law of the mixture is described by a Markov bridge. This is reminiscent of the construction of the two-layer representation from \cite{bryc2024two} (and the authors of \cite{gabrielli2026mixtures} verify the correspondence for the open TASEP). Let us mention that their results can also be interpreted as a type of Hall-Littlewood representation for the stationary measure of the open (T)ASEP. It would be interesting to see whether the results in the present article, relying on different path representations, can be generalized to their setup. 

Complementing local approximation results to the aforementioned results on the macroscopic properties of the stationary measure of the open ASEP were achieved in \cite{nestoridi2024approximating,yang2025limits}.
An approximation of the stationary measure of the open ASEP in the bulk of the segment using homogeneous Bernoulli-product measures first appeared in \cite{nestoridi2024approximating} for the special case of the three-parameter ASEP, relying on the representation from \cite{brak2006combinatorial}. In \cite{yang2025limits}, the law of the open ASEP is near the boundaries approximated using Askey--Wilson half-line measures. Moreover, at a  microscopic level, the decay of the particle density away from the boundary was determined by Sasamoto in \cite{sasamoto2000density} while correlation were studied for example in \cite{uchiyama2005correlation}. Let us stress that while it is possible to study correlations, and hence approximations, of the stationary measure of the open ASEP directly using contour integrals formulations in the Uchiyama--Sasamoto--Wadati representation from \cite{uchiyama2004asymmetric}, these exact representations are in general hard to evaluate, hence requiring us to take a different approach. 


\subsection{Results in the other parts of the phase diagram}

The present manuscript is the third in a series of three articles on sharp approximations for the stationary measure of the open ASEP. For convenience of the reader, and to distinguish the results in the present paper from the complementary approximation results in the shock region, let us summaries the main results from \cite{SchmidYangShock} and \cite{SchmidYangCoexistence}.

In \cite{SchmidYangShock}, we characterize the stationary measure of the open ASEP on a segment of length $N$ using its two-layer representation. The main results are two-folded. 
First, we introduce a general framework in order to compare the stationary measure of the open TASEP and the stationary measure of the open ASEP. More precisely, we consider the polynomial factors from \cite{bryc2024stationary} in the Radon-Nikodym derivative with respect to the two-layer representation of the open TASEP, to which we refer as Bryc polynomials, indexed by compositions of $N$. We provide a probabilistic characterization of the Bryc polynomials, i.e.~we show that they can be expressed by the moment generating function for a certain inhomogeneous birth-death chain run for $N$ steps. In return, we study this birth-death chain in depth and control its exponential moments, and hence the coefficients in the Bryc polynomials, for typical compositions. 

Second, we show that the aforementioned characterization of the Bryc polynomials can be used to obtain a sharp quantitative approximation of the stationary measure in terms of homogeneous Bernoulli-product measures for the open ASEP in the high and low density phase in the shock region. 

The work \cite{SchmidYangCoexistence} follows overall a similar strategy for sharp approximation results of the stationary measure for the open ASEP on the coexistence line, where we see the presence of a shock. Besides an increase in technical difficulty to achieve similarly sharp approximation results as in the high and low density phase, the key difference is to control precisely the shock location. In particular, we provide a sharp approximation of the shock location in total variation distance with respect to the uniform measure on the segment. 
Let us mention that in both \cite{SchmidYangShock} and \cite{SchmidYangCoexistence}, as well as for the applications in \cite{SchmidYangOpenKPZ} on constructing stationary measures for the open KPZ equation, we work under the parameter constraints
\begin{equation} \label{eq:Liggett's condition}
\gamma=q(1-\alpha),\quad\delta=q(1-\beta),\quad \alpha,\beta\in(0,1] 
\end{equation}
for the open ASEP, which is known as Liggett's condition \cite{liggett1975ergodic}. This is an essential input for the two-layer representation from \cite{bryc2024stationary}. In the present article, we allow for general boundary parameters $\alpha,\beta>0$ and  $\gamma,\delta \geq 0$, requiring a different approach, which we now outline.

\subsection{Overview of the proof ideas}\label{sec:Overview}

In order to approximate the stationary measure of the open ASEP, we rely on the matrix product ansatz by Derrida, Evans, Hakim and Pasquier in \cite{derrida1993exact}. We consider two different solutions for the matrix product ansatz: the first one taken from \cite{hegde2024large}, where the solution of Uchiyama, Sasamoto, and Wadati from \cite{uchiyama2004asymmetric} is interpreted in a two-layer representation; the second one corresponding to the Enaud-Derrida representation from \cite{enaud2004large}, which we specify for the open TASEP, i.e.~when $q=\gamma=\delta=0$. In both representations, we express the height function for a stationary open ASEP configuration using reweighted bicolored Motzkin paths, with fair coin flips for all horizontal steps. 
We will keep the perspective of representing the stationary measure of the open ASEP on the enlarged path space. 
Our main contribution is to control the reweighting in the above bicolored Motzkin representations explicitly, allowing for a precise comparison between the open TASEP, for which the stationary measure can be controlled directly using combinatorial bounds, and the open ASEP. 

We first establish a sharp approximation result for the stationary measure of the open TASEP where the weights for the aforementioned bicolored Motzkin paths from the Enaud-Derrida representation take a simple form. More precisely, we reduce the question of approximating the stationary measure of the open TASEP to effective bounds on the number of bicolored Motzkin path with certain starting and ending point. 

The approximation of the stationary measure for the open ASEP with $q \in (0,1)$ is more delicate. Intuitively, the local weights will away from the $x$-axis be $1$ up to an exponentially small error. This shows that for typical bicolored Motzkin paths in the maximal current phase, the total weight is determined by the law of the path close to the endpoints. To formalize this heuristics, we introduce a Fourier type decomposition to express the weights with respect to Askey-Wilson orthogonal polynomials; see also \cite{askey1985some}. In particular, this enables us to approximate the ratio between the partition function between weighted and unweighted bicolored paths, where we fix one endpoint and weight the starting point according to a suitable endpoint measure. Including now a suitable choice for the starting and endpoint measures, we convert this into an approximation for the stationary measure of the open ASEP. 
To our best knowledge, this is the first time this perspective on the weights of the Uchiyama--Sasamoto--Wadati representation is used to study the stationary measure of the open ASEP. 

For the approximation of the stationary measure for the open ASEP in the high and low density phase, we proceed similarly. In contrast to the maximal current phase, the underlying bicolored Motzkin path will be well-approximated by a biased random walk path, ensuring that the total reweighting from the Uchiyama--Sasamoto--Wadati representation will be away from the starting location of the path only have an exponentially small impact on the law of the path. 

To show that the approximation results are sharp for the open TASEP in maximal current phase, we study the expectation of the occupancy variables for a stationary configuration. To this end, we compare the stationary measure of the open TASEP to an Askey-Wilson type measure on the semi-infinite lattice, which we obtain from the Enaud--Derrida representation of the matrix product ansatz; see also \cite{yang2025limits} for a similar construction in the high and low density phase. Computing the leading order of the expectation of the occupancy variables for different choices of boundary parameters in this measure on the semi-infinite lattice gives the desired result.   

In order to establish moderate deviation estimates for the number of particles in the stationary measure of the open ASEP in the maximal current phase, we use a novel covariance identity for a path measure on space of bicolored Motzkin paths, closely related to the Uchiyama--Sasamoto--Wadati representation without bulk weights. Heuristically, fluctuations under this path measure will dominate the fluctuation of the number of particles under the stationary measure of the open ASEP. 
At the same time, we show that this measure has negative correlations, allowing for sharp estimates on the fluctuations by using Chernoff bounds. 
In order to prove negative correlations under this path measure, we first express the covariance between different steps of the path by the matrix product ansatz as a determinant of a certain $2\times2$ matrix, containing only the partition functions of the aforementioned measure. We then prove that the involved matrices are totally non-negative using log-concavity as well as, crucially, the exact choice and the properties of the matrices in the matrix product ansatz for expressing the correlations. 

\subsection{Outline of the paper}

In Section~\ref{sec:Preliminaries}, we recall the matrix product ansatz for the open ASEP as well as basic combinatorial estimates on bi-colored Motzkin paths. In Section~\ref{sec:TASEP}, we establish Theorem~\ref{thm:main open TASEP} on approximating the stationary measure of the open TASEP in the maximal current phase. In Section~\ref{sec:ASEP}, we prove the main approximation result from Theorem~\ref{thm:main open ASEP} on the open ASEP in the maximal current phase. Corresponding results from Theorem~\ref{thm:main open ASEP HLD} for the open ASEP in the high and low density phase of the fan region are provided in Section~\ref{sec:HighLow}. In Section~\ref{sec:Converse}, we show the lower bound in Theorem~\ref{thm:converse TASEP} on the approximation of the stationary measure the open TASEP. In Section~\ref{sec:Moderate}, we establish moderate deviation bounds for the total number of particles under the stationary measure of the open ASEP in the maximal current phase,  proving Theorem~\ref{thm:moderate deviation MC}.

\subsection*{Acknowledgments}
We thank Davide Gabrielli and Federica Iacovissi for helpful discussions. 
D.S. was partially supported by the Packard Foundation via Amol Aggarwal's Packard Fellowships for Science and Engineering.
Z.Y. was partially supported by the Miller Institute for Basic Research in Science at the University of California, Berkeley, Ivan Corwin's NSF grants DMS-1811143, DMS-2246576, and the Fernholz Foundation's ``Summer Minerva Fellows'' program. Both authors were supported by Ivan Corwin's Simons Foundation Grant 929852.

\subsection*{AI declaration}  

This project was initiated in fall 2024, and the proofs were written down by the end of August 2026. The authors used AI to assist with editing the manuscript, performing some computations, and checking arguments. During this process, AI suggested using log-concavity and total nonnegativity in Section~\ref{sec:LogConcave}. Apart from this suggestion, all proof ideas in this article were developed by the authors. 
\section{Preliminaries on the stationary measures of the open ASEP}\label{sec:Preliminaries}

In the following, we collect preliminary results on the matrix product ansatz for the open ASEP using bicolored Motzkin paths, as well as several useful combinatorial estimates. 

\subsection{Bicolored Motzkin paths} 
\begin{definition} \label{def:bicolored Motzkin path}
Fix $N\in\mathbb{Z}_{\geq0}$. A bicolored Motzkin path  of length $N$ is a sequence $\bg=(\gamma_0,\dots,\gamma_N)$ such that $\gamma_i\in\mathbb{Z}_{\ge 0}$ for $i\in\llbracket0,N\rrbracket$ with increments $\gamma_i-\gamma_{i-1}\in\{-1,0,1\}$ for $i\in\llbracket1,N\rrbracket$. Each horizontal step of the path, i.e., $\gamma_i-\gamma_{i-1}=0$, is colored either black or white; the $\pm1$ steps are uncolored. 
For every $x,y\in\mathbb{Z}_{\ge0}$, let $\mathrm{BCMot}^{x,y}_N$ denote the set of bicolored Motzkin paths with $\gamma_0=x$ and $\gamma_N=y$, and set $Z_N(x,y)$  the number of  elements in $\mathrm{BCMot}^{x,y}_N$. 
For every $x\in\mathbb Z_{\geq0}$, let $\mathrm{BCMot}^{x,\mathrm{free}}_N =\cup_{y\ge0}\mathrm{BCMot}^{x,y}_N$ denote the set of bicolored Motzkin paths with $\gamma_0=x$, and set $Z_N(x,\mathrm{free})=\sum_{y\ge0}Z_N(x,y)$ the number of  elements in $\mathrm{BCMot}^{x,\mathrm{free}}_N$.
We note that $Z_N(x,y)$ and $Z_N(x,\mathrm{free})$ are finite.
\end{definition}

\begin{definition}\label{def:bimot project to particles}
    Let $N\in\mathbb{N}$ and $\bg=(\gamma_0,\dots,\gamma_N)$ be a bicolored Motzkin path of length $N$. Define $\Phi(\bg)=(\tau_1,\dots,\tau_N)\in\{0,1\}^N$ by setting $\tau_i=1$ if either $\gamma_i-\gamma_{i-1}=1$ or $\gamma_i-\gamma_{i-1}=0$ and this horizontal step is colored black; and $\tau_i=0$ if either $\gamma_i-\gamma_{i-1}=-1$ or $\gamma_i-\gamma_{i-1}=0$ and this horizontal step is colored white.
\end{definition}

We refer to Figure~\ref{fig:MotzkinPaths} for visualization of a path  $\bg$ in $\mathrm{BCMot}^{0,0}_N$, and its projection to a configuration $\Phi(\bg)$ of the open ASEP. 

\begin{lemma}\label{lem:count biMot}
Let $N,x,y\in\mathbb{Z}_{\geq0}$.  
Then $Z_N(x,y)\geq 1$ if and only if $|x-y|\leq N$. 
Furthermore,
\begin{equation}\label{eq:ZNxy}
    Z_N(x,y)=
    \binom{2N}{N+x-y}
    -
    \binom{2N}{N+x+y+2}.
\end{equation}
Here, we use the convention that, for $n\in\mathbb{Z}_{\geq0}$ and $k\in\mathbb{Z}$, $\binom{n}{k}=0$ if $k\notin\llbracket0,n\rrbracket$. 
Moreover, for $N,x,y\in\mathbb{Z}_{\geq0}$,
\[
Z_N(x,y)=\frac{2^{N+1}}{\pi}\int_0^\pi
 (1+ \cos\theta)^N
\sin((x+1)\theta)\sin((y+1)\theta)\,d\theta.
\] \end{lemma}
Note that the quantity $4^{-N} Z_N(x,y)$ can also be interpreted as the probability that a lazy simple random walk started at $x$ ends at $y$ after $N$ steps and stays non-negative.
\begin{proof}
The first assertion follows from the fact that each step of a bicolored Motzkin path is in $\{-1,0,1\}$.
For each $N,x,y\in\mathbb{Z}_{\geq0}$, let $\mathrm{Dyck}_N^{x,y}$ be the set of paths $\boldsymbol{\eta}=(\eta_0,\ldots,\eta_{2N})$ in $\mathbb Z_{\geq0}$ satisfying $\eta_0=2x+1$, $\eta_{2N}=2y+1$, and $\eta_j-\eta_{j-1}\in\{-1,1\}$ for $j\in\llbracket1,2N\rrbracket$. Define the map 
    $\mathrm{BCMot}^{x,y}_N\to\mathrm{Dyck}_N^{x,y}$, $\boldsymbol{\gamma}\mapsto\boldsymbol{\eta}$ by 
    \[
    (\eta_{2i-1}-\eta_{2i-2},\eta_{2i}-\eta_{2i-1})=\begin{cases}
        (1,1) &\mbox{ if }\gamma_i-\gamma_{i-1}=1,\\
        (-1,-1) &\mbox{ if }\gamma_i-\gamma_{i-1}=-1,\\
        (1,-1) &\mbox{ if }\gamma_i-\gamma_{i-1}=0\mbox{ is colored black},\\
        (-1,1) &\mbox{ if }\gamma_i-\gamma_{i-1}=0\mbox{ is colored white},\\
    \end{cases}
    \]
for each $i\in\llbracket1,N\rrbracket$.
One can check that this map is a bijection from $\mathrm{BCMot}^{x,y}_N$ to $\mathrm{Dyck}_N^{x,y}$. 
For each $N\in\mathbb Z_{\geq0}$ and $x,y\in\mathbb Z$, let $\mathrm{Path}_{N}^{x,y}$ be the set of paths $\boldsymbol{\eta}=(\eta_0,\ldots,\eta_{2N})$ in $\mathbb Z$ satisfying $\eta_0=2x+1$, $\eta_{2N}=2y+1$, and $\eta_j-\eta_{j-1}\in\{-1,1\}$ for $j\in\llbracket1,2N\rrbracket$. 
When $x,y\in\mathbb Z_{\geq0}$, a path $\boldsymbol{\eta}\in\mathrm{Path}_{N}^{x,y}$ is not in $\mathrm{Dyck}_{N}^{x,y}$ if and only if it visits $-1$. For such a path, let 
$T=\min\{j\in\llbracket0,2N\rrbracket:\eta_j=-1\}$ and define a new path by
$\widetilde\eta_j=-2-\eta_j$ for $j\in\llbracket0,T\rrbracket$ and $\widetilde\eta_j=\eta_j$ for $\llbracket T+1,2N\rrbracket$. This reflection is a bijection from $\mathrm{Path}_{N}^{x,y}\setminus\mathrm{Dyck}_{N}^{x,y}$ onto $\mathrm{Path}_{N}^{-x-2,y}$. 
Note  that $\lvert\mathrm{Path}_{N}^{x,y}\rvert
=
\binom{2N}{N+x-y}$ for each $N\in\mathbb Z_{\geq0}$, $x,y\in\mathbb Z$.
Subtraction proves the first formula for $ Z_N(x,y)$. 

For every $k\in\mathbb Z_{\geq0}$, the evenness of $(1+\cos\theta)^N$, the binomial theorem, and binomial symmetry gives
\begin{align*}
\int_0^\pi (1+\cos\theta)^N\cos(k\theta)\,d\theta
 =\frac12\int_{-\pi}^{\pi}
(1+\cos\theta)^N e^{\mathtt{i}k\theta}\,d\theta =\frac{1}{2^{N+1}}\int_{-\pi}^{\pi}
e^{-\mathtt{i}N\theta}(1+e^{\mathtt{i}\theta})^{2N}
e^{\mathtt{i}k\theta}\,d\theta =\frac{\pi}{2^N}\binom{2N}{N+k}.
\end{align*}
 Using
$\sin((x+1)\theta)\sin((y+1)\theta)
=\frac12\left(
\cos((x-y)\theta)-\cos((x+y+2)\theta)
\right)$,
  we obtain
\begin{align*}
 \frac{2^{N+1}}{\pi}\int_0^\pi
(1+\cos\theta)^N
\sin((x+1)\theta)\sin((y+1)\theta)\,d\theta  =
\binom{2N}{N+|x-y|}
-\binom{2N}{N+x+y+2}=Z_N(x,y),
\end{align*} 
which concludes the proof of the second formula for $Z_N(x,y)$.
\end{proof}

Using the above measures on the bicolored Motzkin paths, we have a simple way to obtain the respective measures on particle configurations.

\begin{lemma}\label{lem:exc and mea projection}
Fix $N\in\mathbb N$. Consider the probability measures $\mathbb{P}^{\exc}_N$ and $\mathbb{P}^{\mea}_N$ from Definition \ref{def:exc and mea}, the map \(\Phi\) from Definition \ref{def:bimot project to particles}, and the sets $\mathrm{BCMot}^{x,y}_N$ and $\mathrm{BCMot}^{x,\mathrm{free}}_N$ from Definition \ref{def:bicolored Motzkin path}.
Then the measure $\mathbb{P}^{\exc}_N$ is the pushforward of the uniform measure on \(\mathrm{BCMot}^{0,0}_N\) under the map $\Phi$, and the measure $\mathbb{P}^{\mea}_N$ is the pushforward of the uniform measure on $\mathrm{BCMot}^{0,\mathrm{free}}_N$ under the map $\Phi$.
\end{lemma}
\begin{proof}
Let $\boldsymbol{S}=(S_i)_{i\in\llbracket0,N\rrbracket}$ be the $\nu$-random walk from Definition \ref{def:exc and mea}, and consider the events
$E_{\exc}
=
\left\{
S_i\geq0\ \text{for }i\in\llbracket0,N\rrbracket,
\ S_N=0
\right\}$ and  
$E_{\mea}
=
\left\{
S_i\geq0\ \text{for }i\in\llbracket0,N\rrbracket
\right\}$. 
Consider a probability space carrying the random walk $\boldsymbol{S}$ together with i.i.d. $\ber(1/2)$ random variables $\kappa_i$, $i\in\llbracket1,N\rrbracket$ independent of $\boldsymbol{S}$. Color the $i$th horizontal step of $\boldsymbol{S}$ black if $\kappa_i=1$ and white if $\kappa_i=0$. 
Observe that every fixed bicolored step sequence of length $N$ has probability $4^{-N}$. Hence, conditional on $E_{\exc}$, every path in $\mathrm{BCMot}^{0,0}_N$ has probability $Z_N(0,0)^{-1}$, and, conditional on $E_{\mea}$, every path in $\mathrm{BCMot}^{0,\mathrm{free}}_N$ has probability
$Z_N(0,\mathrm{free})^{-1}$. Note also that conditioning does not change the joint law of the variables $(\kappa_i)$ or their independence from the
conditioned walk.  Finally, for $\star\in\{\exc,\mea\}$ and $i\in\llbracket1,N\rrbracket$, the definition of the coloring and of $\Phi$ gives, pathwise,
 $ [\Phi(\bg^\star) ]_i
=
{\bf1}_{\{S_i^\star-S_{i-1}^\star=1\}}
+
\kappa_i
{\bf1}_{\{S_i^\star-S_{i-1}^\star=0\}}
=
\tau_i^\star$, 
where $\bg^\star$ denotes the corresponding conditioned bicolored path. Hence applying $\Phi$ to the two uniform path measures gives $\mathbb{P}^{\exc}_N$ and
$\mathbb{P}^{\mea}_N$, respectively.
\end{proof}

\subsection{Representation of the stationary measures of open ASEP}\label{sec:MPAs}

A key tool to access the stationary measure of the open ASEP is the celebrated matrix product ansatz, introduced in \cite{derrida1993exact}  by Derrida, Evan, Hakim and Pasquier, which we state as the following lemma; see also~\cite{blythe2007nonequilibrium} for an introduction.

\begin{lemma}\label{lem:MPA} Assume \eqref{eq:conditions open ASEP} holds. Also, assume that there exist matrices $\mathbf{D}$, $\mathbf{E}$, a row vector $\langle W|$, and a column vector $|V\rangle$, all of the same (possibly infinite) dimension, and that all admissible finite products of these matrices and vectors are well-defined (i.e., convergent if they are infinite-dimensional) and satisfy the property of associativity. We assume further that
\begin{equation}
    \label{eq:DEHP algebra} 
        \mathbf{D}\mathbf{E}-q\mathbf{E}\mathbf{D}=\mathbf{D}+\mathbf{E}, \qquad
        \langle W|(\alpha\mathbf{E}-\gamma\mathbf{D})=\langle W|, \qquad
        (\beta\mathbf{D}-\delta\mathbf{E})|V\rangle=|V\rangle
\end{equation}
    (which is commonly known as the DEHP algebra). Recall $\a,\b,\c,\d$ given by \eqref{eq:defining ABCD}.
    Then the stationary measure of the open ASEP can be written as
\begin{equation} \label{eq:MPA open ASEP}
\mu^{(q,\a,\b,\c,\d)}_N (\tau_1,\dots,\tau_N )=\frac{ \langle W | \prod_{i=1}^N({\bf1}_{\{\tau_i=0\}}\mathbf{E}+{\bf1}_{\{\tau_i=1\}} \mathbf{D} ) | V\rangle}{ \langle W |(\mathbf E+\mathbf D)^N|V \rangle},
\end{equation}
for any $(\tau_1,\dots,\tau_N)\in\{0,1\}^N$, assuming that the denominator $\langle W|(\mathbf E+\mathbf D)^N|V\rangle$ is nonzero. 
\end{lemma}

\begin{remark} \label{rmk:denomenator 0} 
As noted by \cite{mallick1997finite,essler1996representations}, assuming $\langle W| V\rangle\neq0$ and finite,  the denominator $\langle W|(\mathbf E+\mathbf D)^N|V\rangle$ of the matrix   ansatz \eqref{eq:MPA open ASEP} being nonzero can be guaranteed by ${q^{ \ell}}\a\b\c\d\neq1$ for all $\ell\in\mathbb Z_{\geq0}$.
As shown in \cite[Appendix A]{mallick1997finite}, in this case, for any $n_1,m_1,\dots,n_k,m_k\in\mathbb Z_+$, the ratio
$\langle W|\mathbf{D}^{n_1}\mathbf{E}^{m_1}\dots\mathbf{D}^{n_k}\mathbf{E}^{m_k}|V\rangle / \langle W|V\rangle$
only depends on the parameters $q,\alpha,\beta,\gamma,\delta$ and the choice of $n_i,m_i$. In particular, this ratio does not depend on the specific examples of $\mathbf{D}$, $\mathbf{E}$, $\langle W|$,   $|V\rangle$ that satisfy the DEHP algebra \eqref{eq:DEHP algebra}.  \end{remark} 

While it is well-known how to construction solutions to the matrix product ansatz, we will in the following require two different solutions --- one for the open ASEP and one specifically to the open TASEP --- whose exact forms will play a crucial role in our analysis.

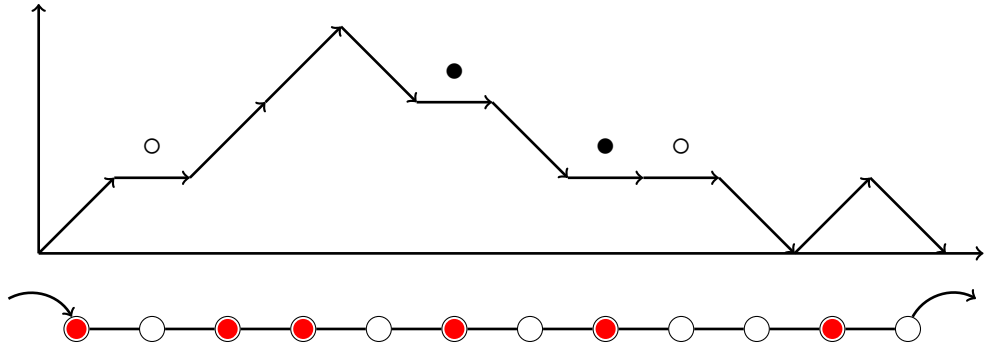
\begin{figure}
    \centering
\begin{tikzpicture}[scale=1]
\draw [->,line width=1pt] (0,0) to (12.5,0); 	
\draw [->,line width=1pt] (0,0) to (0,3.3); 	

	\node[shape=circle,scale=1,draw] (A1) at (0.5,-1){};
	\node[shape=circle,scale=1,draw] (A2) at (1.5,-1){};
	\node[shape=circle,scale=1,draw] (A3) at (2.5,-1){};
	\node[shape=circle,scale=1,draw] (A4) at (3.5,-1){};
	\node[shape=circle,scale=1,draw] (A5) at (4.5,-1){};
	\node[shape=circle,scale=1,draw] (A6) at (5.5,-1){};
	\node[shape=circle,scale=1,draw] (A7) at (6.5,-1){};
	\node[shape=circle,scale=1,draw] (A8) at (7.5,-1){};
	\node[shape=circle,scale=1,draw] (A9) at (8.5,-1){};
	\node[shape=circle,scale=1,draw] (A10) at (9.5,-1){};
	\node[shape=circle,scale=1,draw] (A11) at (10.5,-1){};
	\node[shape=circle,scale=1,draw] (A12) at (11.5,-1){};

 \draw [->,line width=1pt] (-0.4,-1+0.4) to [bend right,in=135,out=45] (A1);
 \draw [->,line width=1pt] (A12) to [bend right,in=135,out=45] (12.4,-1+0.4);	
	
	\node[shape=circle,scale=0.8,fill=red] (CB) at (0.5,-1) {}; 
	\node[shape=circle,scale=0.8,fill=red] (CB) at (2.5,-1) {}; 
	\node[shape=circle,scale=0.8,fill=red] (CB) at (3.5,-1) {}; 
	\node[shape=circle,scale=0.8,fill=red] (CB) at (5.5,-1) {}; 
	\node[shape=circle,scale=0.8,fill=red] (CB) at (7.5,-1) {}; 
	\node[shape=circle,scale=0.8,fill=red] (CB) at (10.5,-1) {};

	\draw [line width=1pt] (A1) to (A2); 	
	\draw [line width=1pt] (A2) to (A3); 	
	\draw [line width=1pt] (A3) to (A4); 	
	\draw [line width=1pt] (A4) to (A5); 	
	\draw [line width=1pt] (A5) to (A6); 	
	\draw [line width=1pt] (A6) to (A7); 	
	\draw [line width=1pt] (A7) to (A8); 	
	\draw [line width=1pt] (A8) to (A9); 	
	\draw [line width=1pt] (A9) to (A10); 	
	\draw [line width=1pt] (A10) to (A11); 	
	\draw [line width=1pt] (A11) to (A12); 	
 

	\draw [->,line width=1pt] (0,0) to (1,1); 	 
 	\draw [->,line width=1pt] (1,1) to (2,1); 	 
 	\draw [->,line width=1pt] (2,1) to (3,2); 	 
 	\draw [->,line width=1pt] (3,2) to (4,3); 	
 	\draw [->,line width=1pt] (4,3) to (5,2); 	 
 	\draw [->,line width=1pt] (5,2) to (6,2); 	 
 	\draw [->,line width=1pt] (6,2) to (7,1); 	 
 	\draw [->,line width=1pt] (7,1) to (8,1); 	
 	\draw [->,line width=1pt] (8,1) to (9,1); 	

 \node[scale=1.5] (H1) at (1.5,1.4) {$\circ$};	 	
 \node[scale=1.5] (H2) at (5.5,2.4) {$\bullet$};		
 \node[scale=1.5] (H3) at (7.5,1.4) {$\bullet$};	 	
 \node[scale=1.5] (H4) at (8.5,1.4) {$\circ$};	
 	 
 	\draw [->,line width=1pt] (9,1) to (10,0); 	
 	\draw [->,line width=1pt] (10,0) to (11,1); 	 
 	\draw [->,line width=1pt] (11,1) to (12,0); 		
 

\end{tikzpicture}
 \caption{Correspondence between bicolored Motzkin paths and the particle configuration of the open ASEP on a segment of length $N=12$. For all horizontal edges, we write $\circ$ for coloring the edge white, and $\bullet$ for coloring the edge black.}
    \label{fig:MotzkinPaths}
\end{figure}

\begin{definition}\label{def:measure Motzkin} Fix $N\in\mathbb Z_{\geq0}$ and $q\in[0,1)$, and assume
    $\a,\b\geq0$, $\c,\d\in(-1,0]$, and $\a\b<1$. Let $\mathrm{BCMot}_N=\bigcup_{x,y\geq0}\mathrm{BCMot}^{x,y}_N$. For each $\bg\in\mathrm{BCMot}_N$, $i\in\llbracket1,N\rrbracket$, let
    \[
    \upomega_i(\bg)=\begin{cases}
         1-q^{\gamma_i}&\qquad\mbox{if } \gamma_i-\gamma_{i-1}=1,\\
         1-\c\d q^{\gamma_i}&\qquad\mbox{if } \gamma_i-\gamma_{i-1}=-1,\\
         1+\d q^{\gamma_i} &\qquad\mbox{if } \gamma_i-\gamma_{i-1}=0\mbox{ is colored black},\\
         1+\c q^{\gamma_i} &\qquad\mbox{if } \gamma_i-\gamma_{i-1}=0\mbox{ is colored white}.
    \end{cases}
    \]
    Denote $\uptheta(h)= (\c\d;q)_{h}/(q;q)_{h}$ for $h\in\mathbb Z_{\geq0}$ and define, for each $\bg\in\mathrm{BCMot}_N$, 
    \[
    \mathrm{wt}_N^{(q,\a,\b,\c,\d)}(\bg)=\a^{\gamma_0}\b^{\gamma_N} \uptheta(\gamma_N) \prod_{i=1}^N\upomega_i(\bg).
    \]
    We then define the normalization constant
    \[Z_N^{(q,\a,\b,\c,\d)}=\sum_{\bg\in\mathrm{BCMot}_N}\mathrm{wt}_N^{(q,\a,\b,\c,\d)}(\bg),\]
    which, as shown in Theorem \ref{thm:Motzkin representation} below, is finite. Define a probability measure on $\mathrm{BCMot}_N$ by
    \begin{equation}\label{def:FullWeightMeasure}
        \mathbb{P}_N^{(q,\a,\b,\c,\d)}(\bg)=\mathrm{wt}_N^{(q,\a,\b,\c,\d)}(\bg)/Z_N^{(q,\a,\b,\c,\d)}.
    \end{equation}
\end{definition}
The next statement records an exact representation of the partition function $Z_N^{(q,\a,\b,\c,\d)}$. Here and in the following, we write
\begin{equation}\label{def:Pochhammer}
    (a;q)_j := \prod_{i=1}^{j} (1-aq^{i-1})
\end{equation} with $(a;q)_{\infty} := \lim_{j \rightarrow \infty} (a;q)_j$ for all $a \in \mathbb R$, $j\in \mathbb \mathbb N$, and $q \in [0,1)$ for the $q$\textit{-Pochhammer symbol}. Moreover, we use the convention that $(a,b;q)_{\infty}:=(a;q)_{\infty}(b;q)_{\infty}$. 
\begin{theorem}\label{thm:Motzkin representation}
Fix $N\in\mathbb N$ and $q\in[0,1)$.
Assume $\a,\b\geq0$, $\c,\d\in(-1,0]$ and $\a\b<1$. Then  $Z_N^{(q,\a,\b,\c,\d)}$ is finite and is a continuous function of $(\a,\b,\c,\d)$. In particular, $\mathbb{P}_N^{(q,\a,\b,\c,\d)}$ is a well-defined probability measure on $\mathrm{BCMot}_N$. The pushforward of this measure under the map $\Phi$ from Definition \ref{def:bimot project to particles} is the stationary measure $\mu_N^{(q,\a,\b,\c,\d)}$ of the open ASEP. Moreover, if additionally $q^{\ell}\a^2\neq1$ and
$q^{\ell}\b^2\neq1$ for every $\ell\in\mathbb Z_{\geq0}$, then
\[
Z_N^{(q,\a,\b,\c,\d)}
=
(q,\a\c,\a\d,\b\c,\b\d,\c\d;q)_\infty
\oint_{\mathcal C}
\frac{(1+z)^{2N}}{4\pi\mathtt{i}\,z^{N+1}} 
\frac{(z^2,z^{-2};q)_\infty}
{(\a z,\a/z,\b z,\b/z,\c z,\c/z,\d z,\d/z;q)_\infty}
\,dz ,
\]
where  \(\mathcal C\) is a positively oriented contour that encloses $\{0\}\cup\{
\mathfrak e q^k:
\mathfrak e\in\{\a,\b,\c,\d\},\ k\in\mathbb Z_{\geq0} \}$
and excludes $\{
(\mathfrak e q^k)^{-1}:
\mathfrak e\in\{\a,\b,\c,\d\},\ k\in\mathbb Z_{\geq0},\
\mathfrak e q^k\neq0\}$. 
\end{theorem}
\begin{proof}
We construct a bijection between $\mathrm{BCMot}_N$ and the state space $\mathcal{TL}_N$ consisting of 
\begin{equation}
\boldsymbol{\lambda}=(\boldsymbol{\lambda}_1,\boldsymbol{\lambda}_2)\in\mathbb{Z}^{2N+2},\quad \boldsymbol{\lambda}_i=(\lambda_i(0),\dots,\lambda_i(N)),\quad i=1,2,
\end{equation}
such that
\begin{equation*}
    \lambda_1(0)=0 , \qquad \lambda_1(j) \geq \lambda_2(j) \text{ for all } j \in \llbracket 0,N\rrbracket , \qquad \lambda_i(j) - \lambda_i(j-1) \in \{0,1\} \text{ for all } i\in \{1,2\}, j\in \llbracket 1,N\rrbracket ;
\end{equation*}
see \cite[Definition 2.4]{hegde2024large}. Indeed,  for $\bg=(\gamma_0,\dots,\gamma_N)\in\mathrm{BCMot}_N$, let
\[\boldsymbol{\lambda}=(\boldsymbol{\lambda}_1,\boldsymbol{\lambda}_2)\in\mathbb{Z}^{2N+2},\quad \boldsymbol{\lambda}_i=(\lambda_i(0),\dots,\lambda_i(N)),\quad i=1,2,\]
be given by
\begin{equation}\label{eq:lambda bij}
    \lambda_1(0)=0,\qquad \lambda_1(j)-\lambda_1(j-1)=\Phi(\bg)_j, \mbox{  }j\in\llbracket1,N\rrbracket,\qquad \lambda_2(j)=\lambda_1(j)-\gamma_j,\mbox{ }j\in\llbracket0,N\rrbracket.
\end{equation} One can check that this is a bijection between $\mathrm{BCMot}_N$ and $\mathcal{TL}_N$, and that it maps the weight $\mathrm{wt}_N^{(q,\a,\b,\c,\d)}$ from Definition \ref{def:measure Motzkin} to the two-layer Gibbs weight from \cite[(2.3)]{hegde2024large}.  
Therefore, \cite[Theorem~2.5]{hegde2024large} proves the asserted
pushforward identity and gives
$0<Z_N^{(q,\a,\b,\c,\d)}<\infty$.

    For the Enaud--Derrida representation $\mathbf D_{\mathrm{ED}}, \mathbf E_{\mathrm{ED}}, \langle W_{\mathrm{ED}}|, |V_{\mathrm{ED}}\rangle$ originally constructed in \cite{enaud2004large}, from \cite[(2.8)]{hegde2024large} we have
 $Z_N^{(q,\a,\b,\c,\d)}
=(1-q)^N
\langle W_{\mathrm{ED}}|
(\mathbf D_{\mathrm{ED}}+\mathbf E_{\mathrm{ED}})^N
|V_{\mathrm{ED}}\rangle$ under the assumption that $\a\b<1$. Moreover,  \cite[Lemma~2.3]{hegde2024large} gives 
$\langle W_{\mathrm{ED}}|V_{\mathrm{ED}}\rangle
= (\a\b\c\d;q)_\infty/(\a\b;q)_\infty$. 
The continuity of $Z_N^{(q,\a,\b,\c,\d)}$ follows from all the entries of this representation being continuous and the local uniform convergence of infinite sums.
In view of Remark \ref{rmk:denomenator 0},
\[ 
\langle W_{\mathrm{ED}}|
(\mathbf D_{\mathrm{ED}}+\mathbf E_{\mathrm{ED}})^N
|V_{\mathrm{ED}}\rangle/
\langle W_{\mathrm{ED}}|V_{\mathrm{ED}}\rangle 
= 
\langle W_{\mathrm{U}}|
(\mathbf D_{\mathrm{U}}+\mathbf E_{\mathrm{U}})^N
|V_{\mathrm{U}}\rangle/\langle W_{\mathrm{U}}|V_{\mathrm{U}}\rangle ,
\]
where \(\mathrm U\) denotes the Uchiyama--Sasamoto--Wadati representation \cite{uchiyama2004asymmetric}. By
\cite[(3.12), (4.18)]{uchiyama2004asymmetric}, 
$\langle W_{\mathrm U}|V_{\mathrm U}\rangle
= (\a\b\c\d;q)_\infty/ (q,\a\b,\a\c,\a\d,\b\c,\b\d,\c\d;q)_\infty $. 
The integral formula follows from \cite[(6.1)]{uchiyama2004asymmetric}.
\end{proof}

Let us stress that in the above theorem, we allow for the open ASEP with general boundary parameters, assuming only that $\a\b<1$. The next theorem gives an alternative representation for the open TASEP, which is reminiscent of the representation of the open ASEP stationary measure via signed measure on bicolored Motzkin paths from  \cite{brak2006combinatorial}; see also \cite{CW:CombinatoricsTASEP} for the connection between the matrix product ansatz, bicolored Motzkin paths and Dyck paths.

\begin{theorem}\label{thm:TASEP rep}
Fix $N\in\mathbb N$. Assume $\a,\b\in[0,1]$, $\c,\d\in(-1,0]$, and $\a\b<1$. 
Denote $\upchi_j(u,v)=\sum_{\ell=0}^{j}u^{j-\ell}v^\ell$ for every $u,v\in\mathbb R$ and $j\in\mathbb Z_{\geq0}$.
For each $\bg=(\gamma_0,\dots,\gamma_N)\in\mathrm{BCMot}_N$, define 
\[
\Uplambda(\bg)= \upchi_{\gamma_0}(\a,\c)\upchi_{\gamma_N}(\b,\d)/\mathcal Z_N,
\]
where $\mathcal Z_N=\sum_{\bg\in\mathrm{BCMot}_N}\upchi_{\gamma_0}(\a,\c)\upchi_{\gamma_N}(\b,\d)$. Then $\Uplambda$ defines a finite signed measure on $\mathrm{BCMot}_N$ of total mass one,
and its pushforward under the map $\Phi$ from
Definition~\ref{def:bimot project to particles} is the stationary
measure $\mu_N^{(0,\a,\b,\c,\d)}$ of open TASEP.  Moreover, we have the integral identity
\[\mathcal Z_N
=
\frac{2^{N+1}(1-\a\c)(1-\b\d)}{\pi}
\int_0^\pi
\frac{d\theta(1+ \cos\theta)^N\sin^2\theta}
{\left(1-2\a\cos\theta+\a^2\right)
 \left(1-2\b\cos\theta+\b^2\right)
 \left(1-2\c\cos\theta+\c^2\right)
 \left(1-2\d\cos\theta+\d^2\right)}. \]   
\end{theorem}

 \begin{proof}
One can check that the infinite matrices 
\begin{equation}\label{eq:DandE}
    \mathbf{D}= \begin{bmatrix}
  1 & 1 & 0 & 0 &\dots\\
  0 & 1 & 1 &0 &\dots\\
  0& 0 & 1 &  1 \\
   \vdots & & \ddots&\ddots \\
\end{bmatrix},\quad
   \mathbf{E}= \begin{bmatrix}
  1  & 0 & 0 & \dots& \\
  1  & 1  & 0&\dots \\
  0& 1  & 1  &  &\\
   \vdots & & \ddots & \ddots\\
\end{bmatrix}, 
\end{equation}     
and infinite vectors:
\[
\langle W | =[\upchi_0(\a,\c),\upchi_1(\a,\c),\upchi_2(\a,\c),\dots], \quad\quad\quad
|V\rangle= [\upchi_0(\b,\d),\upchi_1(\b,\d),\upchi_2(\b,\d),\dots]^T,
\]
satisfy the DEHP algebra \eqref{eq:DEHP algebra} with $q=0$. Since $|\upchi_h(u,v)|
\leq (h+1)\max(|u|,|v|)^{h}$ for each $u,v\in\mathbb R$ and $h\in\mathbb Z_{\geq0}$, we have 
$\sum_{x,y\geq0}
|W_x|\,
|(\mathbf M_1\cdots\mathbf M_m)_{x}^{y}|\,
|V_y|<\infty$ for every finite word $\mathbf M_1\cdots\mathbf M_m$, with
$\mathbf M_i\in\{\mathbf D,\mathbf E\}$. Moreover,
$\sum_{\bg\in\mathrm{BCMot}_N} |
\upchi_{\gamma_0}(\a,\c)\upchi_{\gamma_N}(\b,\d) |
=
\sum_{x,y\geq0}
Z_N(x,y)|W_xV_y|<\infty$, 
because $Z_N(x,y)=0$ unless $|x-y|\leq N$, and so 
$Z_N(x,y)\leq4^N$.  For every $\bg\in\mathrm{BCMot}_N$, let $\boldsymbol{\lambda}\in\mathcal{TL}_N$ be given by the bijection \eqref{eq:lambda bij}, then one can check  
\[
 \upchi_{\gamma_0}(\a,\c)\upchi_{\gamma_N}(\b,\d) = W_{\lambda_1(0)-\lambda_2(0)}\cdot\prod_{j=1}^N
    \left({\bf1}_{\lambda_1(j)-\lambda_1(j-1)=0}\mathbf{E}+{\bf1}_{\lambda_1(j)-\lambda_1(j-1)=1}\mathbf{D}\right)_{\lambda_1(j-1)-\lambda_2(j-1)}^{\lambda_1(j)-\lambda_2(j)}
    \cdot V_{\lambda_1(N)-\lambda_2(N)}.
\] 
Consequently, for every
$\boldsymbol{\tau}=(\tau_1,\dots,\tau_N)\in\{0,1\}^N$,  
\[
\sum_{ \bg\in\mathrm{BCMot}_N ,\,
                 \Phi(\bg)=\boldsymbol{\tau}} 
\upchi_{\gamma_0}(\a,\c)\upchi_{\gamma_N}(\b,\d)
=
\left\langle W\Bigg{|}
\prod_{j=1}^N
\left(
(1-\tau_j)\mathbf E
+
\tau_j\mathbf D
\right)
\Bigg{|}V\right\rangle.
\]
In particular,  
 $ \mathcal Z_N=\sum_{\bg\in\mathrm{BCMot}_N}
\upchi_{\gamma_0}(\a,\c)\upchi_{\gamma_N}(\b,\d)
=
\langle W|(\mathbf D+\mathbf E)^N|V\rangle$. 
We conclude the pushforward identity from the matrix product ansatz stated as Lemma \ref{lem:MPA}.
We next prove the integral identity. Compute
\begin{equation}\label{eq:wvytuye}
    \langle W|V \rangle=  (1-\a\b\c\d)/ ((1-\a\b)(1-\a\d)(1-\b\c)(1-\c\d)).\end{equation} 
For the Enaud--Derrida representation, taking $q=0$ in \cite[(2.8) and Lemma 2.3]{hegde2024large} gives 
\[
Z_N^{(0,\a,\b,\c,\d)}
= 
\langle W_{\mathrm{ED}}|
(\mathbf D_{\mathrm{ED}}+\mathbf E_{\mathrm{ED}})^N
|V_{\mathrm{ED}}\rangle,\qquad \langle W_{\mathrm{ED}}|V_{\mathrm{ED}}\rangle
= (1-\a\b\c\d)/(1-\a\b)   ,
\]
and hence by Remark \ref{rmk:denomenator 0} and \eqref{eq:wvytuye}, 
$\mathcal Z_N= 
 Z_N^{(0,\a,\b,\c,\d)}/( (1-\a\d)(1-\b\c)(1-\c\d) )$. When  $\a,\b<1$, the integral identity follows from Theorem \ref{thm:Motzkin representation}. When $\a=1$ or $\b=1$ it can be shown by a continuity argument.
\end{proof}

As a direct consequence of the above statements, we obtain  Proposition~\ref{prop:SpecialExcursion}.

\begin{proof}[Proof of Proposition~\ref{prop:SpecialExcursion}]
Since $\upchi_x(0,0)={\bf1}_{\{x=0\}}$ for $x\in\mathbb Z_{\geq0}$, at \(\a=\b=\c=\d=0\), the path measure \(\Uplambda\) in Theorem~\ref{thm:TASEP rep} is the uniform  measure on \(\mathrm{BCMot}_N^{0,0}\). Hence \(\mu_N^{(0,0,0,0,0)}\) is the pushforward under \(\Phi\) of the uniform probability measure on \(\mathrm{BCMot}_N^{0,0}\), and the claim follows from Lemma~\ref{lem:exc and mea projection}.
\end{proof}

Finally, we record the well-known particle-hole duality of the stationary measures of the open ASEP; see, for example, \cite[Section 4.2]{wang2023askey} for a proof of this result.
\begin{lemma} \label{lem:particle hole}
Fix $N\in\mathbb N$. Assume $\a,\b,\c,\d, q$ satisfy \eqref{eq:conditions qABCD}.
    The stationary measures $\mu_N^{(q,\a,\b,\c,\d)}$ and $\mu_N^{(q,\b,\a,\d,\c)}$ for open ASEPs are related by a transformation of occupation variables $\widehat{\tau}_i=1-\tau_{N+1-i}$ for $i\in\llbracket1,N\rrbracket$. 
\end{lemma}
 
The particle-hole duality will in particular allow us to transfer results between the high density and the low density phase of the open ASEP.

\subsection{Some useful estimates on combinatorial quantities }\label{sec:Combinatorial}

In the following, we collect some useful combinatorial estimates, most of them corresponding to  probabilities of certain events of lazy simple random walks, which we will apply throughout this article. While such bounds can in principle be deduced from local limit theorems for simple random walks --- see for example \cite{LL:RWIntroduction} --- we will provide self-contained elementary, purely combinatorial proofs of Lemma~\ref{lem:bounds on combinatorial numbers} to Lemma~\ref{lem:TechnicalCombiInput} in Appendix~\ref{app:Combinatorics}.

\begin{lemma}\label{lem:bounds on combinatorial numbers}
    For every $n\in\mathbb Z_{\geq0}$ and $m\in\llbracket-n,n\rrbracket$,
    \[
 \frac{4^{n-1}}{e^2\sqrt{n+1}}
\exp\!\left(-(\log4)\frac{m^2}{n+1}\right) \leq \binom{2n}{n+m}\leq  \frac{4^n}{\sqrt{n+1}}
\exp\!\left(- \frac{m^2}{n+1}\right).
    \]
\end{lemma} 

Next, we provide upper and lower bounds on the function $Z_N(x,y)$ from Definition~\ref{def:bicolored Motzkin path}.

\begin{lemma}\label{lem:bounds ZNxy}
For $N,x,y\in\mathbb Z_{\geq0}$, recall $Z_N(x,y)$ from Definition~\ref{def:bicolored Motzkin path}. Then for every $N\in\mathbb Z_{\geq0}$,
\begin{align*}
     Z_N(x,y)&\le
  4^{N+1}\frac{(x+1)(y+1)}{(N+1)^{3/2}}
 \exp\!\left(- \frac{(x-y)^2}{N+1}\right)\qquad\mbox{ for every }x,y\in\mathbb Z_{\geq0},\\
 Z_N(0,y)
 &\ge
 e^{-2}4^{N-1}  \, \frac{y+1}{(N+1)^{3/2}}
 \exp\!\left(-(\log 4)\frac{y^2}{N+1}\right)\qquad\mbox{ for every }y\in\llbracket0,N\rrbracket.
\end{align*}
\end{lemma}

For the next statement, $Z_N(x,\mathrm{free})4^{-N}$ can be interpreted as the probability for a lazy simple random walk started from $x$ to stay non-negative.

\begin{lemma}\label{lem:bounds ZNxfree}
For $N,x\in\mathbb Z_{\geq0}$, recall $Z_N(x,\mathrm{free})$ from Definition~\ref{def:bicolored Motzkin path}. Then for every $N,x\in\mathbb Z_{\geq0}$,
\begin{align*}
    & Z_N(x,\mathrm{free})
 \leq
 2\,4^N
 \min\left\{1,\frac{x+1}{\sqrt{N+1}}\right\}
 \qquad\mbox{ and }\qquad  Z_N(0,\mathrm{free}) 
 \geq
 \frac{e^{-2}4^{N-1}}{\sqrt{N+1}}.
\end{align*} 
Moreover, we have that for each fixed $L\in(0,\infty)$, uniformly for $n\geq3$ and $x\leq L\log n$,
\begin{equation}\label{eq:converse-short-free-expansion}
\frac{Z_n(x,\mathrm{free})}{4^n}
=\frac1{\sqrt{\pi n}}
\left(2(x+1)-\frac1n
\left(\frac23(x+1)^3+\frac7{12}(x+1)\right)
+O_L\!\left(\frac{(x+1)^5}{n^2}\right)\right).
\end{equation}
\end{lemma}

Intuitively, in the next statement, we compare the probability of a lazy simple random walk started from $x$ and $0$, respectively, to end at location $y$ while staying non-negative.  

\begin{lemma}\label{lem:comparison kernels}
       For every $N,x,y\in\mathbb Z_{\geq0}$ satisfying
  $(x+2)^2\leq  N+1 $, 
    \[
     \left|
 \frac{Z_N(x,y)}{x+1}-Z_N(0,y)
 \right|
 \leq 4^{N+3}\frac{(x+1)^2(x+y+2)}{(N+1)^{5/2}}
 \exp\left(-\frac{y^2}{12(N+1)}\right).
    \]
\end{lemma}

Similarly, we have the following combinatorial bound on the probability of a lazy simple random walk conditioned to stay non-negative, and to end at given locations $|n|-1$ and $|n+j|-1$.

\begin{lemma}\label{lem:TechnicalCombiInput}
    Set for every $N\in\mathbb{Z}_{\geq0},\,k\in\mathbb Z$
\[
y_N(k):=\frac{k}{N+1}\binom{2N+2}{N+k+1} =\mathrm{sgn}(k)Z_N(0,|k|-1).
\] Then for each $N\in\mathbb{Z}_{\geq0}$, $j\in\mathbb Z$, and $n\in\llbracket1,\lfloor\varepsilon(N+1)\rfloor\rrbracket$ with $\varepsilon \in (0,\frac{1}{4}]$, we have that
\begin{equation}\label{eq:compa}
      \frac{|y_N(n+j)-y_N(n)|}{y_N(n)}
\leq
8e^4\,|j| e^{8\varepsilon |j|}\left(\frac1n+\frac{n}{N+1}\right) .
\end{equation}

\end{lemma}

\section{Open TASEP in the maximal current phase: Proof of Theorem \ref{thm:main open TASEP}}\label{sec:TASEP}

In this section, we establish the approximation results for the stationary measure of the open TASEP from Theorem \ref{thm:main open TASEP}. As a main input, we require the following combinatorial estimates on bicolored Motzkin paths. 

\subsection{Comparisons of the marginal laws of bicolored Motzkin paths}
\begin{definition}
Fix $N\in\mathbb{Z}_{\geq0}$ and $m\in\llbracket0,N\rrbracket$.
For every $x,y\in\mathbb{Z}_{\geq0}$ with $|x-y|\leq N$, we write $\Xi_{N,m}^{x,y}$ to denote the law of the $m^{\text{th}}$ component of $(\gamma_i)_{i \in \llbracket 1,N \rrbracket}$  under the uniform measure on $\mathrm{BCMot}^{x,y}_N$. For every \(x\in\mathbb Z_{\geq0}\), we write \(\Xi_{N,m}^{x,\mathrm{free}}\) to denote the law of \(\gamma_m\) under the uniform measure on \(\mathrm{BCMot}^{x,\mathrm{free}}_N\). 
Equivalently, for every $z\in\mathbb Z_{\geq0}$,
\[
 \Xi_{N,m}^{x,y}(\{z\})
 =
 \frac{Z_m(x,z)Z_{N-m}(z,y)}{Z_N(x,y)}\qquad \mbox{ and }\qquad  \Xi_{N,m}^{x,\mathrm{free}}(\{z\})
 =
 \frac{Z_m(x,z)Z_{N-m}(z,\mathrm{free})}
 {Z_N(x,\mathrm{free})}.
\] 
\end{definition}

The following result is an upper bound on the total variation distance at time $m$ between two simple random walks started from $x$ and $0$, respectively, and with and without conditioning on the event to stay non-negative.

\begin{proposition}\label{prop:marginal law comparison}
Fix $N\in\mathbb{Z}_{\geq0}$ and $m\in\llbracket0,N\rrbracket$. Then 
\begin{align}
\label{eq:p1}   \TV{ \Xi_{N,m}^{x,y}  -\Xi_{N,m}^{0,y} }   \le
2^{26}\,\frac{(x+1)^4(y+1)^2}{m+1}\qquad&\mbox{ for every }x,y\in\mathbb Z_{\ge0}\mbox{ with }|x-y|\leq N,\,y\leq N,\\
\label{eq:p2}  \TV{\Xi_{N,m}^{x,\mathrm{free}}- \Xi_{N,m}^{0,\mathrm{free}}}
 \leq
 2^{26}\frac{(x+1)^2}{m+1}\qquad&\mbox{ for every }x\in\mathbb Z_{\geq0}.
\end{align} 
\end{proposition}
\begin{proof}
We split the proof into four steps, which will only require the combinatorial identities discussed in Section~\ref{sec:Combinatorial} together with elementary estimates. The first two steps prove \eqref{eq:p1} and the last two steps prove \eqref{eq:p2}.\\

\noindent\textbf{Step 1. } Define the (deterministic) quantity
\[
 \mathcal S
 =
 \sum_{z=0}^{\infty}(z+x+2)(z+1)
 \exp\left(
 -\frac{z^2}{12(m+1)}-\frac{(z-y)^2}{N-m+1}
 \right).
\]
In this step, we conclude the proof of \eqref{eq:p1}, assuming that for all $x\geq1,\,m\leq N-1,\,
 m+1\geq 2^{26}(x+1)^4(y+1)^2$, we have 
\begin{equation}\label{eq:one-time-endpoint-Gaussian-final}
 \mathcal S
 \leq
 40\cdot12^{3/2}(x+1)^2(y+1)^2
 \left(\frac{(m+1)(N-m+1)}{N+1}\right)^{3/2} . 
\end{equation}
Indeed, the total variation distance is always bounded by $1$, and the left side of \eqref{eq:p1} is zero when $x=0$ or $m=N$. Hence to prove \eqref{eq:p1} we can assume $x\geq1$,
$m\leq N-1$, and
$m+1\geq 2^{26}(x+1)^4(y+1)^2$. Denote
\begin{align*}
    &u_z=\frac{Z_m(x,z)}{x+1}Z_{N-m}(z,y),\qquad
 v_z=Z_m(0,z)Z_{N-m}(z,y),\\
 & U=\sum_{z\geq0}u_z=\frac{Z_N(x,y)}{x+1}>0,\qquad
   V=\sum_{z\geq0}v_z=Z_N(0,y)>0,
\end{align*}
where we use  $\sum_{z\geq0}Z_m(a,z)Z_{N-m}(z,y)=Z_N(a,y)$ and Lemma~\ref{lem:count biMot}. Then 
\begin{equation}
    \begin{split}
  2\TV{\Xi_{N,m}^{x,y}-\Xi_{N,m}^{0,y}} &=
 \sum_{z\geq0}\left|\frac{u_z}{U}-\frac{v_z}{V}\right| \leq
 \sum_{z\geq0}\left|\frac{u_z}{U}-\frac{u_z}{V}\right|
 +\frac1V\sum_{z\geq0}|u_z-v_z|  \\
 &=\frac{|U-V|}{V}+\frac {\sum_{z\geq0}|u_z-v_z|}{V}
 \leq\frac{2\sum_{z\geq0}|u_z-v_z|}{V}.
 \label{eq:one-time-endpoint-normalization}
     \end{split}
\end{equation} 
Note that the present assumptions imply
 $(x+2)^2\leq m+1 $,
so the hypothesis of Lemma~\ref{lem:comparison kernels} holds  with $(N,x,y)=(m,x,z)$.
The upper bound in Lemma~\ref{lem:bounds ZNxy} and Lemma \ref{lem:comparison kernels} give
\begin{equation}\label{eq:one-time-endpoint-E-bound}
 \sum_{z\geq0}|u_z-v_z|
 \leq
 4^{N+4}
 \frac{(x+1)^2(y+1)}{(m+1)^{5/2}(N-m+1)^{3/2}}\,
 \mathcal S.
\end{equation} Note that $N\geq m+1 \geq y^2$, hence
 $\exp\!\left(-(\log4)\frac{y^2}{N+1}\right)\geq\frac14$. 
The lower bound in Lemma~\ref{lem:bounds ZNxy} gives
\begin{equation}\label{eq:one-time-endpoint-V-lower}
 V=Z_N(0,y)
 \geq
 \frac{e^{-2}}{16}\,
 4^N\frac{y+1}{(N+1)^{3/2}}.
\end{equation} Combining \eqref{eq:one-time-endpoint-Gaussian-final},
\eqref{eq:one-time-endpoint-normalization},  
\eqref{eq:one-time-endpoint-E-bound}, and
\eqref{eq:one-time-endpoint-V-lower},  
\[
 \TV{\Xi_{N,m}^{x,y}-\Xi_{N,m}^{0,y}}
 \leq
 4096\cdot40\cdot12^{3/2}e^2\,
 \frac{(x+1)^4(y+1)^2}{m+1}
 <
 2^{26}\frac{(x+1)^4(y+1)^2}{m+1},
\]
which concludes the proof of \eqref{eq:p1}.\\ 

\noindent\textbf{Step 2. } In this step, we prove the bound \eqref{eq:one-time-endpoint-Gaussian-final}. Set 
\[
 A=m+1,\qquad B=N-m+1,\qquad
 D=\frac{12AB}{B+12A},
 \qquad
 \mu=\frac{12Ay}{B+12A}.
\]
Then a simple computation shows that
\begin{equation}\label{eq:complete square}
 \frac{z^2}{12A}+\frac{(z-y)^2}{B}
 =
 \frac{(z-\mu)^2}{D}+\frac{y^2}{B+12A}\geq \frac{(z-\mu)^2}{D}.
\end{equation}
One can also observe that 
\begin{equation}\label{eq:qwer4}
 1\leq D\leq 12\frac{AB}{A+B-1},
 \qquad 0\leq\mu\leq y.
\end{equation}

We next claim that 
\begin{equation}\label{eq:one-time-endpoint-shifted-Gaussian-sum}
 \sum_{k\in\mathbb Z}
e^{-(k-\theta)^2/c}
 \leq 2+\sqrt{\pi c}\qquad\mbox{ for every }c>0,\,\theta\in\mathbb R.
\end{equation}
Indeed, the terms corresponding to $k=\lfloor\theta\rfloor,\lfloor\theta\rfloor+1$ are at most $1$,  
\[
 \sum_{k\leq \lfloor\theta\rfloor-1}e^{-(k-\theta)^2/c}
 \leq
 \int_{-\infty}^{\lfloor\theta\rfloor}e^{-(t-\theta)^2/c}\,dt,\qquad\mbox{ and }\qquad
  \sum_{k\geq \lfloor\theta\rfloor+2}e^{-(k-\theta)^2/c}
 \leq
 \int_{\lfloor\theta\rfloor+1}^{\infty}e^{-(t-\theta)^2/c}\,dt,
\]
which concludes \eqref{eq:one-time-endpoint-shifted-Gaussian-sum}.
For every \(t\in\mathbb R\), we have $(1+t^2)e^{-t^2/D}
 \leq2D e^{-t^2/(2D)}$ since \(D\geq1\).
By \eqref{eq:one-time-endpoint-shifted-Gaussian-sum} and \eqref{eq:qwer4},
\begin{equation}\label{eq:qwer5}
\begin{split}
 \sum_{z=0}^{\infty}\left(1+|z-\mu|^2\right)e^{-|z-\mu|^2/D} \leq
 2D\sum_{k\in\mathbb Z}
 e^{- (k-\mu)^2/(2D) } \leq
 2D\left(2+\sqrt{2\pi D}\right)
 \leq10D^{3/2}.
\end{split}
\end{equation}
Since $0\leq\mu\leq y$, one can check that $(z+x+2)(z+1)
 \leq4(x+1)^2(y+1)^2 (1+|z-\mu|^2 )$. Combining with \eqref{eq:complete square},  \eqref{eq:qwer4}, and \eqref{eq:qwer5}, we obtain
\[
\begin{split}
 \mathcal S
 &\leq
 4(x+1)^2(y+1)^2
 \sum_{z=0}^{\infty}(1+|z-\mu|^2)e^{-|z-\mu|^2/D} \leq40(x+1)^2(y+1)^2D^{3/2}\\
 &\leq
 40\cdot12^{3/2}(x+1)^2(y+1)^2
 \left(\frac{AB}{A+B-1}\right)^{3/2}.
\end{split}
\] 
We conclude the proof of \eqref{eq:one-time-endpoint-Gaussian-final}.\\

\noindent\textbf{Step 3. } Define the (deterministic) quantity
\[
\mathcal H=
 \sum_{z=0}^{\infty}
 (z+x+2)e^{-z^2/(12(m+1))}
 \min\left\{1,\frac{z+1}{\sqrt{N-m+1}}\right\}.
\]
In this step, we conclude the proof of \eqref{eq:p2}, assuming that
\begin{equation}\label{eq:free-endpoint-Gaussian-sum}
\mathcal H\leq160\frac{(m+1)^{3/2}}{\sqrt{N+1}}\qquad\mbox{ when }x\geq1\mbox{ and } m+1\geq2^{26}(x+1)^2.
\end{equation}
Indeed, the total variation distance is always bounded by $1$, and the left side of \eqref{eq:p2} is zero when $x=0$. Hence to prove \eqref{eq:p2} we can assume $x\geq1$ and $m+1\geq2^{26}(x+1)^2$. For this step, denote
\begin{align*}
    &  \mathsf u_z=\frac{Z_m(x,z)}{x+1}Z_{N-m}(z,\mathrm{free}),
 \qquad
\mathsf v_z=Z_m(0,z)Z_{N-m}(z,\mathrm{free}),\\
 & \mathsf U=\sum_{z\geq0}\mathsf u_z=\frac{Z_N(x,\mathrm{free})}{x+1}>0,
 \qquad
\mathsf V=\sum_{z\geq0}\mathsf v_z=Z_N(0,\mathrm{free})>0,
\end{align*}
where we use $\sum_{z\geq0}Z_m(a,z)Z_{N-m}(z,\mathrm{free})=Z_N(a,\mathrm{free})$.
By the same calculation as in \eqref{eq:one-time-endpoint-normalization},  
\begin{equation} \label{eq:qwer6}
 \TV{\Xi_{N,m}^{x,\mathrm{free}}
      -\Xi_{N,m}^{0,\mathrm{free}}}
 \leq
 \frac{1}{\mathsf V}\sum_{z\geq0}|\mathsf u_z-\mathsf v_z|.
\end{equation}
Note that the present assumptions imply $(x+2)^2\leq4(x+1)^2<m+1 $, so Lemma~\ref{lem:comparison kernels} applies with
$(N,x,y)=(m,x,z)$. By Lemma \ref{lem:bounds ZNxfree} and Lemma~\ref{lem:comparison kernels},
\begin{equation*} 
 \sum_{z\geq0}|\mathsf u_z-\mathsf v_z|
 \leq
 4^{m+3}\frac{(x+1)^2}{(m+1)^{5/2}}
 \sum_{z=0}^{\infty}
 (z+x+2)e^{-z^2/(12(m+1))}Z_{N-m}(z,\mathrm{free}) \leq
 2\,4^{N+3}\frac{(x+1)^2}{(m+1)^{5/2}}\,
 \mathcal H.
\end{equation*}
Combining this estimate with \eqref{eq:free-endpoint-Gaussian-sum}, \eqref{eq:qwer6}, and Lemma \ref{lem:bounds ZNxfree}, we obtain
\[
\begin{aligned}
 \TV{\Xi_{N,m}^{x,\mathrm{free}}
      -\Xi_{N,m}^{0,\mathrm{free}}}
 &\leq
 320\cdot4^4e^2\frac{(x+1)^2}{m+1} <
 2^{26}\frac{(x+1)^2}{m+1},
\end{aligned}
\]
which proves \eqref{eq:p2}.\\

\noindent\textbf{Step 4. } In this step, we prove the remaining claim  \eqref{eq:free-endpoint-Gaussian-sum}. Denote
\[
 S_j=\sum_{z=0}^{\infty}z^j e^{-z^2/(12(m+1))},\qquad 
  I_j=\int_0^\infty t^j e^{-t^2/(12(m+1))}\,dt,\qquad j\in\{0,1,2\}.
\]
Note that for each $z\geq1$, $t\in[z-1,z]$, and $j\in\{0,1,2\}$, we have
 $z^j e^{-z^2/(12(m+1))}
 \leq
 (t+1)^j e^{-t^2/(12(m+1))}$.
Hence $S_0\leq1+I_0$, $S_1\leq I_1+I_0$, and $S_2\leq I_2+2I_1+I_0$. 
We compute the integrals explicitly: $I_0=\sqrt{3\pi (m+1)}$, $I_1=6(m+1)$, and $I_2=\frac{\sqrt\pi}{4}(12(m+1))^{3/2}$. 
By the assumptions of \eqref{eq:free-endpoint-Gaussian-sum}, 
\begin{equation}\label{eq:free-endpoint-Gaussian-moments}
 S_0\leq5\sqrt{m+1},\qquad
 S_1\leq10(m+1),\qquad
 S_2\leq35(m+1)^{3/2}.
\end{equation}
The assumptions of \eqref{eq:free-endpoint-Gaussian-sum}  imply
 $x+2\leq 2^{-12}\sqrt{m+1}$.
To prove \eqref{eq:free-endpoint-Gaussian-sum}, we split into two cases.
Note that when \(N\leq2m\),  we get that
\[
 \mathcal H
 \leq
 \sum_{z=0}^{\infty}(z+x+2)e^{-z^2/(12(m+1))}
 =
 S_1+(x+2)S_0
 \leq\left(10+5\cdot2^{-12}\right)(m+1)
 \leq20(m+1) \leq
 160\frac{(m+1)^{3/2}}{\sqrt{N+1}}.
\] 
Similarly, when \(N>2m\), we have \(N+1\leq2(N-m+1)\); hence
\[
\begin{aligned}
 \mathcal H
 &\leq
 \frac{1}{\sqrt{N-m+1}}
 \sum_{z=0}^{\infty}
 (z+x+2)(z+1)e^{-z^2/(12(m+1))}\ =
 \frac{1}{\sqrt{N-m+1}}
 \left(S_2+(x+3)S_1+(x+2)S_0\right)\\
 &\leq
 \left(35+15\cdot2^{-12}\right)
 \frac{(m+1)^{3/2}}{\sqrt{N-m+1}}
 \leq
 60\frac{(m+1)^{3/2}}{\sqrt{N-m+1}} \leq
 160\frac{(m+1)^{3/2}}{\sqrt{N+1}}.
\end{aligned}
\] 
We conclude the proof of \eqref{eq:free-endpoint-Gaussian-sum}.
\end{proof}

\subsection{Proof of Theorem \ref{thm:main open TASEP}}
In this subsection, we prove Theorem \ref{thm:main open TASEP} on approximating the stationary measure of the open TASEP in the maximal current phase. For convenience, we state it again as Theorem~\ref{thm:main open TASEP again} below.

\begin{theorem}\label{thm:main open TASEP again}
    Assume $\a,\b\geq0$ and $\mathsf{c},\mathsf{d}\in(-1,0]$. For $N\in\mathbb N$, consider the stationary measure $\mu_N^{(0,\a,\b,\c,\d)}$ of open TASEP. There exists a constant $C<\infty$ depending only on $\a,\b,\c$ and $\d$, such that the following hold.
    \begin{enumerate}
        \item [(1)] In the maximal current phase $\a,\b\in[0,1)$, for every $N,m,k\in\mathbb N$ satisfying 
        $1\le m\le N-k+1\le N$,  
\begin{equation}\label{eq:oni1}
\TV{
\mu_N^{(0,\a,\b,\c,\d)}
\left(
(\tau_i)_{i\in\llbracket m,N-k+1\rrbracket}\in\cdot
\right)
-
\mathbb{P}_N^{\mathsf{exc}}
\left(
(\tau_i)_{i\in\llbracket m,N-k+1\rrbracket}\in\cdot
\right)}
\le
C/\min(m,k) .
\end{equation}
In two special cases in the maximal current phase, slightly stronger approximation results hold.  
Specifically, in the case when $\a\in[0,1)$ and $\b=\d=0$, for every $N\in\mathbb N$ and $m\in\llbracket1,N\rrbracket$, 
\begin{equation}\label{eq:oni2}
\TV{
\mu_N^{(0,\a,0,\c,0)}
\left(
(\tau_i)_{i\in\llbracket m,N \rrbracket}\in\cdot
\right)
-
\mathbb{P}_N^{\mathsf{exc}}
\left(
(\tau_i)_{i\in\llbracket m,N \rrbracket}\in\cdot
\right)}
\le
C/m ;
\end{equation}
 and in the case when  $\a=\c=0$ and $\b\in[0,1)$, for every $N\in\mathbb N$ and $k\in\llbracket1,N\rrbracket$, 
\begin{equation}\label{eq:oni3}
\TV{
\mu_N^{(0,0,\b,0,\d)}
\left(
(\tau_i)_{i\in\llbracket 1,N-k+1\rrbracket}\in\cdot
\right)
-
\mathbb{P}_N^{\mathsf{exc}}
\left(
(\tau_i)_{i\in\llbracket 1,N-k+1\rrbracket}\in\cdot
\right)}
\le
C/k .
\end{equation}
        \item [(2)] On the boundary of the maximal current phase $\a\in[0,1)$ and $\b=1$, for every $N\in\mathbb N$ and $m\in\llbracket1,N\rrbracket$, 
        \[
\TV{
\mu_N^{(0,\a,1,\c,\d)}
\left(
(\tau_i)_{i\in\llbracket m,N \rrbracket}\in\cdot
\right)
-
\mathbb P_N^{\mathsf{mea}}
\left(
(\tau_i)_{i\in\llbracket m,N \rrbracket}\in\cdot
\right)}
\le
C/m .
\]
        \item [(3)] On the boundary of the maximal current phase $\a=1$ and $\b\in[0,1)$, for every $N\in\mathbb N$ and $k\in\llbracket1,N\rrbracket$, 
        \[
\TV{
\mu_N^{(0,1,\b,\c,\d)}
\left(
(\tau_i)_{i\in\llbracket 1,N-k+1 \rrbracket}\in\cdot
\right)
-
\mathbb P_N^{\mathsf{revmea}}
\left(
(\tau_i)_{i\in\llbracket 1,N-k+1 \rrbracket}\in\cdot
\right)}
\le
C/k .
\]
    \end{enumerate} 
\end{theorem}
\begin{proof} For \(x,y\in\mathbb Z_{\geq0}\) with \(|x-y|\leq N\), let \(\Uptheta_N^{x,y}\) (resp. \(\Uptheta_N^{x,\mathrm{free}}\)) denote the pushforward under \(\Phi\) of the uniform probability measure on \(\mathrm{BCMot}_N^{x,y}\) (resp. \(\mathrm{BCMot}_N^{x,\mathrm{free}}\)). 
For \(|x-y|>N\), set
\(\Uptheta_N^{x,y}=\Uptheta_N^{0,0}\). This convention does not affect
any formula below because \(Z_N(x,y)=0\) if \(|x-y|>N\).
For
\(I\subseteq\llbracket1,N\rrbracket\), a subscript \(I\) denotes the $(\tau_i:i\in I)$ marginal of a measure on $(\tau_1,\dots,\tau_N)$. 
Recall $\upchi_j(u,v)=\sum_{\ell=0}^{j}u^{j-\ell}v^\ell$ for $u,v\in\mathbb R$ and $j\in\mathbb Z_{\geq0}$.
For all the cases in the theorem, $\a,\b\in[0,1]$ and $\a\b<1$, hence Theorem~\ref{thm:TASEP rep} on the matrix product ansatz for the open TASEP applies with
\begin{equation}\label{eq:TASEP-endpoint-mixture}
\mu_N^{(0,\a,\b,\c,\d)}
=
\frac1{\mathcal Z_N}
\sum_{x,y\geq0}
\upchi_x(\a,\c)\upchi_y(\b,\d)Z_N(x,y)\Uptheta_N^{x,y},
\qquad
\mathcal Z_N
=
\sum_{x,y\geq0}\upchi_x(\a,\c)\upchi_y(\b,\d)Z_N(x,y)>0,
\end{equation}
where the series converges in total variation distance. We split the rest of the proof into two steps. \\

\noindent\textbf{Step 1. } In this step, we prove Statement (1). To prove \eqref{eq:oni1}, we assume $\a,\b\in[0,1)$,  $1\le m\le N-k+1\le N$,   and denote $I=\llbracket m,N-k+1\rrbracket$. Using \eqref{eq:TASEP-endpoint-mixture} and taking the $I$-marginal,
\begin{equation}\label{eq:TASEP-centered-endpoint-mixture}
\begin{split}
&\mu_N^{(0,\a,\b,\c,\d)}
\left((\tau_i)_{i\in I}\in\cdot\right)
-
\Uptheta_{N,I}^{0,0} =
\frac1{\mathcal Z_N}
\sum_{x,y\geq0}
\upchi_x(\a,\c)   \upchi_y(\b,\d)Z_N(x,y)
\left(\Uptheta_{N,I}^{x,y}-\Uptheta_{N,I}^{0,0}\right).
\end{split}
\end{equation}

We now derive a lower bound for $\mathcal Z_N$.
In view of $1-2r\cos\theta+r^2\leq (1+|r|)^2$ for each $r,\theta\in\mathbb R$, the integral formulas for $\mathcal Z_N$ and \(Z_N(0,0)\) from  Theorem~\ref{thm:TASEP rep} and Lemma~\ref{lem:count biMot} together give  
\[
\mathcal Z_N
\geq
\frac{(1-\a\c)(1-\b\d)}
{(1+\a)^2(1+|\c|)^2(1+\b)^2(1+|\d|)^2}
Z_N(0,0).
\]
By the lower bound in Lemma~\ref{lem:bounds ZNxy}, there exists $C_1\in(0,\infty)$ depending only on $\a,\b,\c,\d$ such that 
\begin{equation}\label{eq:TASEP-normalization-lower}
\mathcal Z_N
\geq
C_1\frac{4^N}{(N+1)^{3/2}}.
\end{equation}

We next claim that, for each $x,y\in\mathbb Z_{\geq0}$ with \(|x-y|\leq N\) and \(y\leq N\),
\begin{equation}\label{eq:TASEP-two-sided-bridge-comparison}
\TV{\Uptheta_{N,I}^{x,y}-\Uptheta_{N,I}^{0,0}}
\leq
2^{26}\left(
\frac{(x+1)^4(y+1)^2}{m}
+
\frac{(y+1)^4}{k}
\right).
\end{equation}
Indeed, under the uniform measure on \(\mathrm{BCMot}_N^{x,y}\), conditional on \(\gamma_{m-1}\), the law of $\gamma_i$, $i\in\llbracket m-1,N\rrbracket$ is uniform
on \(\mathrm{BCMot}_{N-m+1}^{\gamma_{m-1},y}\). 
By Proposition~\ref{prop:marginal law comparison},
\begin{equation*} 
\TV{\Uptheta_{N,I}^{x,y}-\Uptheta_{N,I}^{0,y}}
\leq
\TV{\Xi_{N,m-1}^{x,y}-\Xi_{N,m-1}^{0,y}}
\leq
2^{26}\frac{(x+1)^4(y+1)^2}{m}.
\end{equation*}
By the same reason, 
\[
\TV{\Uptheta_{N,I}^{0,y}-\Uptheta_{N,I}^{0,0}}
\leq
2^{26}\frac{(y+1)^4}{k}.
\] 
We conclude the proof of \eqref{eq:TASEP-two-sided-bridge-comparison} from the triangle inequality.

We next bound the numerator of \eqref{eq:TASEP-centered-endpoint-mixture}. 
Denote $\rho_L=\max(\a,-\c)<1$ and $\rho_R=\max (\b,-\d)<1$. We have 
\begin{equation}\label{eq:TASEP-endpoint-geometric-bound}
\left|\upchi_x(\a,\c)\right|\leq(x+1)\rho_L^x,
\qquad
\left|\upchi_y(\b,\d)\right|\leq(y+1)\rho_R^y.
\end{equation} 
For \(y\leq N\), use \eqref{eq:TASEP-two-sided-bridge-comparison}, \eqref{eq:TASEP-endpoint-geometric-bound}, and the upper bound in
Lemma~\ref{lem:bounds ZNxy}, we have
\begin{equation}\label{eq:TASEP-central-endpoint-bound}
\begin{split}
    &\sum_{ x,y\geq0,\,y\leq N}
\left|\upchi_x(\a,\c)\upchi_y(\b,\d)\right|Z_N(x,y)
\TV{\Uptheta_{N,I}^{x,y}-\Uptheta_{N,I}^{0,0}}\\
&\leq \frac{2^{26}4^{N+1}}{(N+1)^{3/2}}
\left[
\frac1m
\left(\sum_{x\geq0}(x+1)^6\rho_L^x\right)
\left(\sum_{y\geq0}(y+1)^4\rho_R^y\right) 
+
\frac1k
\left(\sum_{x\geq0}(x+1)^2\rho_L^x\right)
\left(\sum_{y\geq0}(y+1)^6\rho_R^y\right)
\right].
\end{split}\end{equation}
For \(y>N\), using
\(\TV{\Uptheta_{N,I}^{x,y}-\Uptheta_{N,I}^{0,0}}\leq1\) and the upper bound in Lemma~\ref{lem:bounds ZNxy},
\begin{equation}\begin{split}\label{eq:TASEP-tail-endpoint-bound}
    &\sum_{ x,y\geq0,\,y>N}
\left|\upchi_x(\a,\c)\upchi_y(\b,\d)\right|Z_N(x,y)
\TV{\Uptheta_{N,I}^{x,y}-\Uptheta_{N,I}^{0,0}}\\
&\leq \frac{4^{N+1}}{(N+1)^{3/2}}
\left(\sum_{x\geq0}(x+1)^2\rho_L^x\right)
\left(\sum_{y>N}(y+1)^2\rho_R^y\right)\leq C_2\frac{4^N}{(N+1)^{5/2}},
\end{split}\end{equation}
where $C_2\in(0,\infty)$ depends only on $\a,\b,\c,\d$. The last step above uses
\[
\sum_{y>N}(y+1)^2\rho_R^y
\leq
2\rho_R^{N+1}
\left(
\frac{(N+2)^2}{1-\rho_R}
+
\sum_{j\geq0}j^2\rho_R^j
\right)\qquad\mbox{ and }\qquad\sup_{N\geq1}(N+1)(N+2)^2\rho_R^{N+1}<\infty.
\]
Combining \eqref{eq:TASEP-central-endpoint-bound} and
\eqref{eq:TASEP-tail-endpoint-bound} and dividing by
\eqref{eq:TASEP-normalization-lower} gives
\[
\TV{
\mu_N^{(0,\a,\b,\c,\d)}
\left((\tau_i)_{i\in I}\in\cdot\right)
-
\Uptheta_{N,I}^{0,0}}
\leq
C_3\left(
\frac1m+\frac1k+\frac1{N+1}
\right)\leq \frac{3C_3}{\min(m,k)},
\]
where $C_3\in(0,\infty)$ depends only on $\a,\b,\c,\d$. 
By Lemma~\ref{lem:exc and mea projection},  \(\Uptheta_N^{0,0}=\mathbb P_N^{\exc}\), and we conclude the proof of \eqref{eq:oni1}. 

We next prove the estimate \eqref{eq:oni2} in Statement (1). Assume now that \(\b=\d=0\), and denote \(I=\llbracket m,N\rrbracket\). 
Note that $\upchi_y(\b,\d)=\upchi_y(0,0)={\bf1}_{\{y=0\}}$. When $x\leq N$, conditioning on $\gamma_{m-1}$ and using Proposition \ref{prop:marginal law comparison},  
\[
\TV{\Uptheta_{N,I}^{x,0}-\Uptheta_{N,I}^{0,0}}
\leq
2^{26}\frac{(x+1)^4}{m}.
\]
Combining 
\eqref{eq:TASEP-endpoint-mixture} and
\eqref{eq:TASEP-endpoint-geometric-bound}, the upper bound in
Lemma~\ref{lem:bounds ZNxy}, and
\eqref{eq:TASEP-normalization-lower},
\[
\TV{
\mu_N^{(0,\a,0,\c,0)}
\left((\tau_i)_{i\in\llbracket m,N\rrbracket}\in\cdot\right)
-
\Uptheta_{N,\llbracket m,N\rrbracket}^{0,0}}
\leq
\frac{2^{28}}{C_1m}
\sum_{x\geq0}(x+1)^6\rho_L^x
\leq
\frac{C_4}{m},
\]  where $C_4\in(0,\infty)$ depends only on $\a,\b,\c,\d$. We conclude the  estimate \eqref{eq:oni2} in Statement (1). The   estimate \eqref{eq:oni3} in Statement (1) then follows from \eqref{eq:oni2} through the particle-hole duality as in Lemma \ref{lem:particle hole}.\\

\noindent\textbf{Step 2. } In this step, we prove the rest of the statements. Note that we only need to prove Statement (2), and then Statement (3) follows from Statement (2) through the particle-hole duality as in Lemma \ref{lem:particle hole}.
Assume \(\a\in[0,1)\) and \(\b=1\), and let $I=\llbracket m,N\rrbracket$. 
Note that $\upchi_y(\b,\d)=\upchi_y(1,\d)=\frac1{1-\d}-\frac{\d^{y+1}}{1-\d}$. Using \eqref{eq:TASEP-endpoint-mixture},
\begin{align}
&\mu_N^{(0,\a,1,\c,\d)}
\left((\tau_i)_{i\in I}\in\cdot\right)-\Uptheta_{N,I}^{0,\mathrm{free}}
\notag\\
& =
\frac1{\mathcal Z_N}
\left[
\frac1{1-\d}\sum_{x\geq0}
\upchi_x(\a,\c)Z_N(x,\mathrm{free})
\left(\Uptheta_{N,I}^{x,\mathrm{free}}-\Uptheta_{N,I}^{0,\mathrm{free}}\right) 
-
\sum_{x,y\geq0}
\frac{\d^{y+1}}{1-\d}\upchi_x(\a,\c)  Z_N(x,y)
\left(\Uptheta_{N,I}^{x,y}-\Uptheta_{N,I}^{0,\mathrm{free}}\right)
\right].
\label{eq:TASEP-meander-decomposition}
\end{align}
These two regrouped series are absolutely convergent. Indeed, denote $\rho=\max (\a,-\c)<1$, then
\begin{align*}
  \sum_{x\geq0}\left|\upchi_x(\a,\c)\right|Z_N(x,\mathrm{free})
&\leq
2\,4^N\sum_{x\geq0}(x+1)\rho^x<\infty,\\
  \sum_{x,y\geq0}  \left|\upchi_x(\a,\c)\frac{\d^{y+1}}{1-\d}\right|Z_N(x,y)
&\leq
\frac{4^N}{1-\d}
\left(\sum_{x\geq0}(x+1)\rho^x\right)
\left(\sum_{y\geq0}|\d|^{y+1}\right)
<\infty,
\end{align*}
where we use $|\upchi_x(\a,\c)|\leq(x+1)\rho^x$ and \(Z_N(x,y)\leq4^N\). 

We now obtain a lower bound for \(\mathcal Z_N\). By Theorem~\ref{thm:TASEP rep} with $\b=1$,
\begin{equation}\label{eq:TASEP-meander-normalization-integral}
\mathcal Z_N
=
\frac{2(1-\a\c)(1-\d)}{\pi}
\int_0^\pi
\frac{(2+2\cos\theta)^N\cos^2(\theta/2)}
{\left(1-2\a\cos\theta+\a^2\right)\left(1-2\c\cos\theta+\c^2\right) \left(1-2\d\cos\theta+\d^2\right) }
\,d\theta.
\end{equation}
When \(\theta\in[0,(N+1)^{-1/2}]\), we have
 $(2+2\cos\theta)^N\cos^2(\theta/2)
=
4^N\cos^{2N+2}(\theta/2)
\geq
e^{-1/2}4^N$.
Indeed, \(\log\cos u\geq-u^2\) for each \(u\in[0,1/2]\), since \(\tan u\leq2u\) on this interval. Combining with $1-2r\cos\theta+r^2\leq (1+|r|)^2$, 
restricting the integral in
\eqref{eq:TASEP-meander-normalization-integral} to \(\theta\in[0,(N+1)^{-1/2}]\) yields
\begin{equation}\label{eq:TASEP-meander-normalization-lower}
\mathcal Z_N
\geq
\frac{2e^{-1/2}(1-\a\c)(1-\d)}
{\pi(1+\a)^2(1+|\c|)^2(1+|\d|)^2}\frac{4^N}{\sqrt{N+1}}.
\end{equation}

We next upper bound the numerator of \eqref{eq:TASEP-meander-decomposition}.
Under the uniform measure on \(\mathrm{BCMot}_N^{x,\mathrm{free}}\), conditional on \(\gamma_{m-1}\), the law of $\gamma_i$, $i\in\llbracket m-1,N\rrbracket$ is uniform on \(\mathrm{BCMot}_{N-m+1}^{\gamma_{m-1},\mathrm{free}}\). By Proposition \ref{prop:marginal law comparison},
\begin{equation}\label{eq:TASEP-free-suffix-comparison}
\TV{
\Uptheta_{N,I}^{x,\mathrm{free}}
-
\Uptheta_{N,I}^{0,\mathrm{free}}
}
\leq
2^{26}\frac{(x+1)^2}{m}.
\end{equation} Combining with Lemma~\ref{lem:bounds ZNxfree}, 
\begin{equation}\label{eq:TASEP-meander-main-error}
    \frac1{1-\d}\sum_{x\geq0}
\left|\upchi_x(\a,\c)\right|Z_N(x,\mathrm{free})
\TV{\Uptheta_{N,I}^{x,\mathrm{free}}-\Uptheta_{N,I}^{0,\mathrm{free}}}\leq\frac{2^{27} \,4^N}{(1-\d)m\sqrt{N+1}}
\sum_{x\geq0}(x+1)^4\rho^x
\leq
C_5\frac{4^N}{m\sqrt{N+1}},
\end{equation}
where $C_5\in(0,\infty)$ depends only on $\a,\c,\d$.
For the other summand, using \(\TV{\Uptheta_{N,I}^{x,y}-\Uptheta_{N,I}^{0,\mathrm{free}}}\leq1\),  $|\upchi_x(\a,\c)|\leq(x+1)\rho^x$, 
and the upper bound in Lemma~\ref{lem:bounds ZNxy},  
\begin{equation}\label{eq:TASEP-meander-endpoint-error}
\begin{split} 
& 
\sum_{x,y\geq0}
\frac{|\d|^{y+1}}{1-\d}
\left|\upchi_x(\a,\c)\right|Z_N(x,y)
\TV{\Uptheta_{N,I}^{x,y}
-\Uptheta_{N,I}^{0,\mathrm{free}}}\\
&\leq
\frac{4^{N+1}}{(1-\d)(N+1)^{3/2}}
\left(\sum_{x\geq0}(x+1)^2\rho^x\right)
\left(\sum_{y\geq0}(y+1)|\d|^{y+1}\right)
\leq
C_6\frac{4^N}{(N+1)^{3/2}}.
\end{split}
\end{equation}
where $C_6\in(0,\infty)$ depends only on $\a,\c,\d$.

Combining
\eqref{eq:TASEP-meander-decomposition},
\eqref{eq:TASEP-meander-main-error},
\eqref{eq:TASEP-meander-endpoint-error}, and
\eqref{eq:TASEP-meander-normalization-lower}, and using \(m\leq N\),
gives
\[
\TV{
\mu_N^{(0,\a,1,\c,\d)}
\left((\tau_i)_{i\in I}\in\cdot\right)
-
\Uptheta_{N,I}^{0,\mathrm{free}}}
\leq
C_7\left(\frac1m+\frac1{N+1}\right)
\leq
\frac{2C_7}{m},
\]
where $C_7\in(0,\infty)$ depends only on $\a,\c,\d$.
By  Lemma~\ref{lem:exc and mea projection},
\(\Uptheta_N^{0,\mathrm{free}}=\mathbb P_N^{\mea}\). We conclude the proof of Statement~(2). As mentioned, Statement (3) then follows from the particle-hole duality in Lemma \ref{lem:particle hole}.
\end{proof}

\section{Open ASEP in the maximal current phase: Proof of Theorem \ref{thm:main open ASEP}}\label{sec:ASEP}

In this section, our goal is to transfer the approximation results for the open TASEP in Theorem~\ref{thm:main open TASEP} to the open ASEP by controlling the weights from Theorem~\ref{thm:Motzkin representation}.

\subsection{A comparison estimate for the Motzkin path weights in the bulk} 

We assume in the following that \(q\in[0,1)\) and \(\c,\d\in(-1,0]\). For each
\(N,x,y\in\mathbb Z_{\geq0}\), let
\begin{equation}\label{eq:def of RNxy}
R_N^{(q,\c,\d)}(x,y)
=
\sum_{\bg\in\mathrm{BCMot}^{x,y}_N}
\prod_{i=1}^{N}\upomega_i(\bg),
\end{equation}
where we recall from Definition~\ref{def:measure Motzkin} that
\[
\upomega_i(\bg)=
\begin{cases}
1-q^{\gamma_i},
&\gamma_i-\gamma_{i-1}=1,\\
1-\c\d q^{\gamma_i},
&\gamma_i-\gamma_{i-1}=-1,\\
1+\d q^{\gamma_i},
&\gamma_i-\gamma_{i-1}=0\text{ is colored black},\\
1+\c q^{\gamma_i},
&\gamma_i-\gamma_{i-1}=0\text{ is colored white}.
\end{cases}
\]
In the following, we study a reweighting of the partition function $Z_N(\cdot,x)$ and the total weight functions $R_N^{(q,\c,\d)}(\cdot,x)$ and $R_N^{(q,\c,\d)}(z,\cdot)$ with respect to a given function $\varsigma$. 
\begin{definition}\label{def:ReweightingFunctions}
    Let \(C_0\in(0,\infty)\), \(\varrho\in[0,1)\), and let \(\varsigma:\mathbb Z_{\geq0}\to[0,\infty)\) be a function satisfying $\varsigma(0)>0$ and $\varsigma(j)\leq C_0\varrho^j$ for every $j\in\mathbb Z_{\geq0}$. For each \(x\in\mathbb Z_{\geq0}\), define
\[
G_N^\varsigma(x)
=\sum_{z\geq0}\varsigma(z)Z_N(z,x),
\qquad
F_N^\varsigma(x)
=\sum_{z\geq0}\varsigma(z)R_N^{(q,\c,\d)}(z,x),
\qquad H_N^\varsigma(x)
=\sum_{z\geq0}\varsigma(z)R_N^{(q,\c,\d)}(x,z).
\]
\end{definition}

Our goal in this section is to establish the following approximation result for the functions in Definition~\ref{def:ReweightingFunctions}.  
Recall the $q$-Pochhammer symbol from \eqref{def:Pochhammer}.

\begin{proposition}\label{prop:comparison bulk weights}
Assume \(q\in[0,1)\) and \(\c,\d\in(-1,0]\). 
Let \(C_0\in(0,\infty)\), \(\varrho\in[0,1)\), and let \(\varsigma:\mathbb Z_{\geq0}\to[0,\infty)\) be a function satisfying $\varsigma(0)>0$ and $\varsigma(j)\leq C_0\varrho^j$ for every $j\in\mathbb Z_{\geq0}$. Consider the functions $G_N^\varsigma, F_N^\varsigma $ and $ H_N^\varsigma$ from Definition~\ref{def:ReweightingFunctions} with respect to $\varsigma$.  Define
\begin{align*}
    &\mathfrak p_j
=[t^j]\frac{(\c t,\d t;q)_\infty}{(t,t;q)_\infty}\qquad\mbox{ for every  }j\in\mathbb Z_{\geq0},\\
&  \Upgamma_\varsigma
=
\frac{(q;q)_\infty^2}{(\c,\d;q)_\infty}
\frac{\sum_{j\geq0}\mathfrak p_j\varsigma(j) }
{\sum_{j\geq0}(j+1)\varsigma(j)},
\qquad\mbox{ and }\qquad
\Upupsilon_\varsigma
=
\frac{(q;q)_\infty(\c\d;q)_\infty}
{(\c,\d;q)_\infty}
\frac{\sum_{j\geq0}
\mathfrak p_j  (q;q)_j (\c\d;q)_j^{-1}\varsigma(j) }
{\sum_{j\geq0}(j+1)\varsigma(j)},
\end{align*} 
where we write  \([t^j]\) to denote the coefficient of \(t^j\) in the Taylor expansion at \(t=0\).
Then the two constants $\Upgamma_\varsigma$ and $\Upupsilon_\varsigma$ are well-defined and belong to \((0,1]\). Moreover, there exists a constant \(C\in(0,\infty)\), depending only on \(q,\c,\d,C_0,\varrho,\varsigma(0)\), such that, for every  \(N,x\in\mathbb Z_{\geq0}\) for which \(G_N^\varsigma(x)>0\),
\[ 
\left|
\frac{F_N^\varsigma(x)}{G_N^\varsigma(x)}
-\Upgamma_\varsigma
\right|
 \leq
C\left(\frac1{x+1}+\frac{x+1}{N+1}\right)  \qquad\mbox{ and }\qquad
\left|
\frac{H_N^\varsigma(x)}{G_N^\varsigma(x)}
-\Upupsilon_\varsigma
\right|
 \leq
C\left(\frac1{x+1}+\frac{x+1}{N+1}\right). 
\]
\end{proposition}

Intuitively, the functions $G_N^\varsigma, F_N^\varsigma $ and $ H_N^\varsigma$ act as expected weights with respect to a measure given by $\varsigma$, where we consider for $G_N^\varsigma$ unweighted bicolored Motzkin paths with a starting point according to $\varsigma$ and fixed ending point, for $F_N^\varsigma$ weighted bicolored Motzkin paths with a fixed ending point, and for $H_N^\varsigma$ weighted bicolored Motzkin paths with a fixed starting point. We show that the ratios between the different functions converge to constants as $N \rightarrow \infty$, and we quantify this approximation. Let us mention that, though determined explicit, the exact forms of these constants are not important for our arguments.

As a main input for the proof of Proposition \ref{prop:comparison bulk weights}, we will use the following Fourier type decomposition of the total weight function $R_N^{(q,\c,\d)}$, related to Askey--Wilson orthogonal polynomials; see \cite{askey1985some}.

\begin{lemma}\label{lem:orth}
    Assume \(q\in[0,1)\) and \(\c,\d\in(-1,0]\). For \(\theta\in[0,\pi]\), define
\begin{align*}
     \phi_j(\theta)
&=
[t^j]\frac{(\c t,\d t;q)_\infty}
{(te^{\mathtt i\theta},te^{-\mathtt i\theta};q)_\infty} \qquad\mbox{ for }j\in\mathbb Z_{\geq0},\\
 d\nu(\theta)
&=
\frac{(q,\c\d;q)_\infty}{2\pi}
\frac{(e^{2\mathtt i\theta},e^{-2\mathtt i\theta};q)_\infty}
{(\c e^{\mathtt i\theta},\c e^{-\mathtt i\theta},
\d e^{\mathtt i\theta},\d e^{-\mathtt i\theta};q)_\infty}
\,d\theta .
\end{align*} 
Then for every $N,x,y\in\mathbb Z_{\geq0}$,
\[
R_N^{(q,\c,\d)}(x,y)
=
\frac{(q;q)_y}{ (\c\d;q)_y}
\int_0^\pi(2+2\cos\theta)^N
\phi_x(\theta)\phi_y(\theta)\,d\nu(\theta).
\]
In particular, when $N=0$, we have $R_0^{(q,\c,\d)}(x,y)={\bf1}_{\{x=y\}}$, and hence for every $x,y\in\mathbb Z_{\geq0}$,
\[
\int_0^\pi\phi_x(\theta)\phi_y(\theta)\,d\nu(\theta)
=
\frac{ (\c\d;q)_x}{ (q;q)_x }\mathbf 1_{\{x=y\}}.
\]
\end{lemma}
\begin{proof}
We start by showing that $\phi_j(\theta)$, $j\in\mathbb Z_{\geq0}$ satisfies the orthogonality
\begin{equation}\label{eq:ort of phi in proof}
    \int_0^\pi\phi_j(\theta)\phi_k(\theta)\,d\nu(\theta)
=
\frac{ (\c\d;q)_j}{ (q;q)_j }\mathbf 1_{\{j=k\}}\qquad\mbox{ for }j,k\in\mathbb Z_{\geq0},
\end{equation}
and the recursion
\begin{equation}\label{eq:recurrencephi}
    \begin{split}
       (2+2\cos\theta)\phi_0(\theta) & =(2+\c+\d)\phi_0(\theta)+(1-q)\phi_1(\theta),\\
       (2+2\cos\theta)\phi_j(\theta) & = (1-q^{j+1})\phi_{j+1}(\theta)
+\left(2+(\c+\d)q^j\right)\phi_j(\theta)+\left(1-\c\d q^{j-1}\right)\phi_{j-1}(\theta)\qquad\mbox{ for } j\in\mathbb N.
    \end{split}
\end{equation}
For  $\theta\in[0,\pi]$ and $|z|<1$, let
\begin{equation}\label{eq:form F}
\mathcal F(z,\theta)
=\frac{(\c z,\d z;q)_\infty}
{(ze^{\mathtt i\theta},ze^{-\mathtt i\theta};q)_\infty}   ,
\end{equation}
then, by the definition of $\phi_j$, one can expand
\begin{equation}\label{eq:power expa F}
    \mathcal F(z,\theta)= \sum_{j=0}^\infty\phi_j(\theta)z^j .
\end{equation}
From \eqref{eq:form F}, we can observe that
\[
(1-\c z)(1-\d z)\mathcal F(qz,\theta)
=(1-2z\cos\theta+z^2)\mathcal F(z,\theta).
\] 
Plugging \eqref{eq:power expa F} into the above equation and comparing coefficients of $z^j$ gives the recurrence \eqref{eq:recurrencephi}. For the orthogonality, by the Askey--Wilson integral \cite{askey1985some}; see also \cite[Section~6.1]{gasper2011basic}, for each $|z|,|w|<1$,
\[
\int_0^\pi
\frac{(e^{2\mathtt i\theta},e^{-2\mathtt i\theta};q)_\infty}
{(\c e^{\mathtt i\theta},\c e^{-\mathtt i\theta},
\d e^{\mathtt i\theta},\d e^{-\mathtt i\theta},
z e^{\mathtt i\theta},z e^{-\mathtt i\theta},
w e^{\mathtt i\theta},w e^{-\mathtt i\theta};q)_\infty}
\,d\theta=\frac{2\pi(\c\d zw;q)_\infty}
{(q,\c\d,\c z,\c w,\d z,\d w,zw;q)_\infty}.
\]
Hence, by \eqref{eq:form F},
\[
\int_0^\pi
\mathcal F(z,\theta)\mathcal F(w,\theta)\,d\nu(\theta)=\frac{(\c\d zw;q)_\infty}{(zw;q)_\infty}=\sum_{n=0}^\infty
\frac{(\c\d;q)_n}{(q;q)_n}(zw)^n,
\]
where the second equality is the \(q\)-binomial theorem \cite[(1.3.2)]{gasper2011basic}. Plugging \eqref{eq:power expa F} into the above equation and comparing coefficients of $z^jw^k$ gives the orthogonality \eqref{eq:ort of phi in proof}. 

Now for each $j,k\in\mathbb Z_{\geq0}$, define
\[
T(j,k)=
\begin{cases}
1-q^{j+1}  \qquad&\mbox{ if }k=j+1,\\
2+(\c+\d)q^j  \qquad&\mbox{ if }k=j,\\
1-\c\d q^{j-1}   \qquad&\mbox{ if }k=j-1,\,j\geq1,\\
0  \qquad&\text{ otherwise}.
\end{cases}
\]
The recurrence \eqref{eq:recurrencephi} gives 
\begin{equation}\label{eq:recurrencephi in K}
\sum_{k\geq0}T(j,k)\phi_k(\theta)
=(2+2\cos\theta)\phi_j(\theta)\qquad\mbox{ for }j\in\mathbb Z_{\geq0}.
\end{equation}
Using the definition \eqref{eq:def of RNxy} of $R_N^{(q,\c,\d)}(x,y)$ and expanding over the first step of the path,
\begin{equation}\label{eq:orth-path-recursion}
R_0^{(q,\c,\d)}(x,y)=\mathbf 1_{\{x=y\}},
\qquad
R_{N+1}^{(q,\c,\d)}(x,y)
=\sum_{k\geq0}T(x,k)R_N^{(q,\c,\d)}(k,y)\qquad\mbox{ for }N,x,y\in\mathbb Z_{\geq0}.
\end{equation}
We now prove the formula 
\begin{equation}\label{eq:formula of R in proof}
R_N^{(q,\c,\d)}(x,y)
=
\frac{(q;q)_y}{ (\c\d;q)_y}
\int_0^\pi(2+2\cos\theta)^N
\phi_x(\theta)\phi_y(\theta)\,d\nu(\theta)\qquad\mbox{ for }x,y\in\mathbb Z_{\geq0}
\end{equation}
by induction over $N$. When $N=0$, \eqref{eq:formula of R in proof} is equivalent to \eqref{eq:ort of phi in proof}. Assume that it holds for $N\in\mathbb Z_{\geq0}$, then, for $N+1$, \eqref{eq:formula of R in proof} follows from the induction hypothesis by combining \eqref{eq:recurrencephi in K} and \eqref{eq:orth-path-recursion}. We conclude the proof of  \eqref{eq:formula of R in proof} for all $N\in\mathbb Z_{\geq0}$ by induction.  
\end{proof}

Next, we provide a Fourier type lower bound on the quantity $Z_N(0,n-1)$.

\begin{lemma}\label{lem:Fourier}
For each \(\upupsilon>0\), there exists a constant \(C\in(0,\infty)\), depending only on \(\upupsilon\), with the following property. 
Let \(M\in[0,\infty)\) and let \((\varphi_j)_{j\in\mathbb Z}\subset\mathbb R\) satisfy
\[
|\varphi_j|\leq M e^{-\upupsilon|j|}
\quad\text{for every }j\in\mathbb Z,
\qquad
\sum_{j\in\mathbb Z}\varphi_j=0.
\]
Define
\[
\varphi(\theta)
=
\sum_{j\in\mathbb Z}\varphi_j e^{\mathtt i\cdot j\theta}
\qquad \mbox{ for }\theta\in\mathbb R.
\]
Then  \(\varphi \colon \mathbb{R} \rightarrow \mathbb{C}\) is a bounded continuous function. Moreover, denote
\[
I_{N,n}(\varphi)
=
\frac2\pi\int_0^\pi
(2+2\cos\theta)^N\sin\theta\,
\operatorname{Im}\!\left(
e^{\mathtt i\cdot n\theta}\varphi(\theta)
\right)\,d\theta \qquad\mbox{ for }N,n\in\mathbb Z_{\geq0},
\]
then 
\[
|I_{N,n}(\varphi)|
\leq
CMZ_N(0,n-1)
\left(\frac1n+\frac n{N+1}\right)\qquad\mbox{ for every }N\in\mathbb Z_{\geq0},\,n\in\llbracket1,\lfloor\min ( 1/4, \upupsilon/16 )(N+1)\rfloor\rrbracket.
\]
\end{lemma}
\begin{proof} Throughout the proof, we denote $\varepsilon=\min ( 1/4, \upupsilon/16 )$. 
    The boundedness and continuity of \(\varphi\) follows from $|\varphi_j|\leq M e^{-\upupsilon|j|}$ using the Weierstrass test. 
Let
\[
y_N(k)=\frac{k}{N+1}\binom{2N+2}{N+k+1}\qquad\mbox{ for }N\in\mathbb{Z}_{\geq0},\,k\in\mathbb Z.
\] 
and recall from Lemma~\ref{lem:TechnicalCombiInput} that   for each $N\in\mathbb{Z}_{\geq0}$, $j\in\mathbb Z$, and $n\in\llbracket1,\lfloor\varepsilon(N+1)\rfloor\rrbracket$,
\begin{equation}\label{eq:compaApplied}
      \frac{|y_N(n+j)-y_N(n)|}{y_N(n)}
\leq
8e^4\,|j| e^{8\varepsilon |j|}\left(\frac1n+\frac{n}{N+1}\right) .
\end{equation}
This implies that 
\begin{equation}\label{eq:sagasetr}
\sum_{j\in\mathbb Z}
e^{-\upupsilon|j|}
|y_N(n+j)-y_N(n)|  \leq
Cy_N(n)\left(\frac1n+\frac{n}{N+1}\right)\qquad\mbox{ where }C= \frac{32 e^{4-\upupsilon/2}}{(1-e^{-\upupsilon/2})^2}.
\end{equation}
Note that $y_N(k)=\mathrm{sgn}(k)Z_N(0,|k|-1)$ for every  $k\in\mathbb Z\setminus\{0\}$. By the integral formula in Lemma~\ref{lem:count biMot},
\[
y_N(k)=
\frac2\pi\int_0^\pi
(2+2\cos\theta)^N\sin\theta\sin(k\theta)\,d\theta
\qquad\mbox{for }N\in\mathbb Z_{\geq0},\,k\in\mathbb Z.
\]
Using the definition of $I_{N,n}(\varphi)$ and $\sum_j\varphi_j=0$,
\begin{equation*} 
I_{N,n}(\varphi)
=\sum_{j\in\mathbb Z}\varphi_j y_N(n+j)
=
\sum_{j\in\mathbb Z}\varphi_j
\left(y_N(n+j)-y_N(n)\right),
\end{equation*}
where the swap between integral and infinite summation signs is justified by the bound on \(\varphi_j\). 
Combining the above display with \eqref{eq:sagasetr}  and the bound on \(\varphi_j\),  we obtain
\begin{align*}
|I_{N,n}(\varphi)|
&\leq
CMZ_N(0,n-1)
\left(\frac1n+\frac n{N+1}\right),
\end{align*} 
and we conclude the proof of the lemma.
\end{proof}

Using the function $I_{N,n}$ from Lemma~\ref{lem:Fourier}, we can provide a common formulation of the functions $G_N^\varsigma, F_N^\varsigma $ and $ H_N^\varsigma$ from Definition~\ref{def:ReweightingFunctions}.

\begin{lemma}\label{lem:identify Fourier}
Assume that \(q\in[0,1)\) and \(\c,\d\in(-1,0]\).  As in Proposition \ref{prop:comparison bulk weights}, let  \(C_0\in(0,\infty)\), \(\varrho\in[0,1)\), and \(\varsigma:\mathbb Z_{\geq0}\to[0,\infty)\) satisfy $\varsigma(0)>0$ and $\varsigma(j)\leq C_0\varrho^j$ for all $j\in\mathbb Z_{\geq0}$, and recall the functions $G_N^\varsigma, F_N^\varsigma $ and $ H_N^\varsigma$ from Definition~\ref{def:ReweightingFunctions}. 
Fix an arbitrary $r>1$ satisfying \(qr<1\). For \(j\in\mathbb Z_{\geq0}\) and \(\theta\in[0,\pi]\), define
\[
\begin{aligned}
\phi_j(\theta)
&=[t^j]\frac{(\c t,\d t;q)_\infty}
{(te^{\mathtt i\theta},te^{-\mathtt i\theta};q)_\infty},\\
\uplambda_j(\theta)
&=\frac{1}{2\pi\mathtt i}\oint_{|z|=r}
\frac{(\c z,\d z;q)_\infty}
{(ze^{\mathtt i\theta},ze^{-\mathtt i\theta};q)_\infty}
\frac{dz}{z^{j+1}},\\
h(\theta)
&=\frac{(\c\d;q)_\infty(qe^{2\mathtt i\theta};q)_\infty}
{(\c e^{\mathtt i\theta},\d e^{\mathtt i\theta};q)_\infty},\\
d\nu(\theta)
&=
\frac{(q,\c\d;q)_\infty}{2\pi}
\frac{(e^{2\mathtt i\theta},e^{-2\mathtt i\theta};q)_\infty}
{(\c e^{\mathtt i\theta},\c e^{-\mathtt i\theta},
\d e^{\mathtt i\theta},\d e^{-\mathtt i\theta};q)_\infty}
\,d\theta .
\end{aligned}
\]  
For \(\theta\in[0,\pi]\), set
\[
\begin{aligned}
S_0(\theta)
=\sum_{j\geq0}\varsigma(j)
\frac{\sin((j+1)\theta)}{\sin\theta},\qquad
S_1(\theta)
=\sum_{j\geq0}\varsigma(j)\phi_j(\theta),\qquad
S_2(\theta)
=\sum_{j\geq0}
\frac{(q;q)_j}{(\c\d;q)_j}\varsigma(j)\phi_j(\theta),
\end{aligned}
\]
where the quotient in \(S_0\) is understood by continuous extension in the cases $\theta/\pi\in\mathbb Z$.  For \(\theta\in[0,\pi]\), let
\[
\upvarphi_0(\theta)=S_0(\theta),
\qquad
\upvarphi_1(\theta)=h(\theta)S_1(\theta),
\qquad
\upvarphi_2(\theta)=h(\theta)S_2(\theta).
\]
Then the functions defined in the above two displays are continuous, and  for every \(N,x\in\mathbb Z_{\geq0}\),
\[
\begin{aligned}
G_N^\varsigma(x)
&=I_{N,x+1}(\upvarphi_0),\\
F_N^\varsigma(x)
&=\frac{(q;q)_x}{(\c\d;q)_x}
\left(
I_{N,x+1}(\upvarphi_1)
+\int_0^\pi(2+2\cos\theta)^N
S_1(\theta)\uplambda_x(\theta)\,d\nu(\theta)
\right),\\
H_N^\varsigma(x)
&=I_{N,x+1}(\upvarphi_2)
+\int_0^\pi(2+2\cos\theta)^N
S_2(\theta)\uplambda_x(\theta)\,d\nu(\theta),
\end{aligned}
\]
where, as in Lemma \ref{lem:Fourier}, we denote 
\[
I_{N,n}(\varphi)
=
\frac{2}{\pi}\int_0^\pi
(2+2\cos\theta)^N\sin\theta\,
\operatorname{Im}\!\left(e^{\mathtt i n\theta}\varphi(\theta)\right)
\,d\theta .
\]
\end{lemma}
\begin{proof}
    We first show that the series $S_0,S_1,S_2$ converge absolutely and uniformly. Let $r_0\in(\varrho,1)$, then
\begin{equation}\label{eq:Cauchy phij}
r_0^{j}\sup_{\theta\in[0,\pi]}|\phi_j(\theta)|
\leq    \sup_{\theta\in [0,\pi]}\sup_{|z|=r_0}\left|\frac{(\c z,\d z;q)_\infty}
{(ze^{\mathtt i\theta},ze^{-\mathtt i\theta};q)_\infty}\right|
\qquad\text{for every }j\in\mathbb Z_{\geq0},
\end{equation}
which follows from Cauchy's coefficient estimate.
Note that the right side above is finite. Observe
\[
\left|\frac{\sin((j+1)\theta)}{\sin\theta}\right|
\leq j+1,\qquad0<\frac{(q;q)_j}{(\c\d;q)_j}
\leq\frac{1}{(\c\d;q)_\infty}, \qquad j\in\mathbb Z_{\geq0},
\]
It follows from the Weierstrass test that $S_0,S_1,S_2$ converge absolutely and uniformly in \(\theta\).
Therefore, $S_i(\theta)$ and $\upvarphi_i(\theta)$ for $i\in\{0,1,2\}$ are well-defined continuous functions. Observe that \(S_0,S_1,S_2\) are real-valued.

We next show that
\begin{equation}\label{eq:identify-Fourier-density}
    \frac{(q,\c\d;q)_\infty}{2\pi}
\frac{(e^{2\mathtt i\theta},e^{-2\mathtt i\theta};q)_\infty}
{(\c e^{\mathtt i\theta},\c e^{-\mathtt i\theta},
\d e^{\mathtt i\theta},\d e^{-\mathtt i\theta};q)_\infty}\left(\phi_j(\theta)-\uplambda_j(\theta)\right)=\frac{2}{\pi}\sin\theta\,
\operatorname{Im}\!\left(
e^{\mathtt i(j+1)\theta}h(\theta)
\right).
\end{equation}
Indeed, by continuity we only need to prove it for \(\theta\in(0,\pi)\). For  \(r_1\in(0,1)\), by Cauchy's coefficient formula,  
\[
\phi_j(\theta)
=
\frac{1}{2\pi\mathtt i}\oint_{|z|=r_1}
\frac{(\c z,\d z;q)_\infty}
{(ze^{\mathtt i\theta},ze^{-\mathtt i\theta};q)_\infty}\frac{dz}{z^{j+1}}.
\]
Comparing with the definition of $\uplambda_j(\theta)$, using the residue theorem and computing the residues at $z=e^{\pm\mathtt i\theta}$, we see that
\begin{equation*}
\begin{aligned}
\phi_j(\theta)-\uplambda_j(\theta) 
=
\frac{
e^{-\mathtt i j\theta}
(\c e^{\mathtt i\theta},\d e^{\mathtt i\theta};q)_\infty
}{
(q;q)_\infty(e^{2\mathtt i\theta};q)_\infty
} +
\frac{
e^{\mathtt i j\theta}
(\c e^{-\mathtt i\theta},\d e^{-\mathtt i\theta};q)_\infty
}{
(q;q)_\infty(e^{-2\mathtt i\theta};q)_\infty
},
\end{aligned} 
\end{equation*}
which is equivalent to \eqref{eq:identify-Fourier-density}.

  By the integral formula in Lemma~\ref{lem:count biMot}, uniform convergence of the series \(S_0\), and   that \(S_0\) is real-valued,
\[
\begin{aligned}
G_N^\varsigma(x)
&=
\frac{2}{\pi}\int_0^\pi
(2+2\cos\theta)^N
\sin((x+1)\theta)
\sum_{j\geq0}\varsigma(j)\sin((j+1)\theta)\,d\theta\\
&=
\frac{2}{\pi}\int_0^\pi
(2+2\cos\theta)^N\sin\theta\,
\operatorname{Im}\!\left(
e^{\mathtt i(x+1)\theta}S_0(\theta)
\right)\,d\theta =I_{N,x+1}(\upvarphi_0).
\end{aligned}
\]
For the remaining two identities, by Lemma~\ref{lem:orth} and uniform convergence of \(S_1,S_2\),
 \begin{equation}
\begin{aligned}
F_N^\varsigma(x)
&=
\frac{(q;q)_x}{(\c\d;q)_x}
\int_0^\pi
(2+2\cos\theta)^N S_1(\theta)\phi_x(\theta)\,d\nu(\theta),\\
H_N^\varsigma(x)
&=
\int_0^\pi
(2+2\cos\theta)^N S_2(\theta)\phi_x(\theta)\,d\nu(\theta).
\end{aligned}
\label{eq:identify-Fourier-spectral}
\end{equation}
Using \eqref{eq:identify-Fourier-density} with the index in that identity set equal to \(x\), and using that \(S_1,S_2\) are
real-valued,
 \begin{equation}
\begin{aligned}
&\int_0^\pi(2+2\cos\theta)^N
S_j(\theta)\left(\phi_x(\theta)-\uplambda_x(\theta)\right)
\,d\nu(\theta)\\
& =
\frac{2}{\pi}\int_0^\pi(2+2\cos\theta)^N\sin\theta\,
\operatorname{Im}\!\left(
e^{\mathtt i(x+1)\theta}h(\theta)S_j(\theta)
\right)\,d\theta =I_{N,x+1}(hS_j)\qquad\mbox{ for } j=1,2.
\end{aligned}
\label{eq:identify-Fourier-splitting}
\end{equation}
Combining with \eqref{eq:identify-Fourier-spectral}, in view of \(hS_1=\upvarphi_1\) and \(hS_2=\upvarphi_2\), we conclude the formulas for
\(F_N^\varsigma(x)\) and \(H_N^\varsigma(x)\).
\end{proof}

We have now all tools to show the approximation results on the functions from Definition~\ref{def:ReweightingFunctions}.

\begin{proof}[Proof of Proposition~\ref{prop:comparison bulk weights}]
Throughout this proof,  we use $C_i$, $i\in\mathbb N$ to denote constants in $(0,\infty)$ that depend only on $q,\c,\d,C_0,\varrho$ and $\varsigma(0)$.
Denote
\begin{equation}\label{eq:PochApplied}
    \upvarpi_j=\frac{(q;q)_j}{(\c\d;q)_j}\qquad\mbox{ for } j\in\mathbb  Z_{\geq0}\cup\{\infty\}.
\end{equation}
We recall from the statement of the proposition that
\begin{align*}
    &\mathfrak p_j
=[t^j]\frac{(\c t,\d t;q)_\infty}{(t,t;q)_\infty}\qquad\mbox{ for every  }j\in\mathbb Z_{\geq0},\\
&  \Upgamma_\varsigma
=
\frac{(q;q)_\infty^2}{(\c,\d;q)_\infty}
\frac{\sum_{j\geq0}\mathfrak p_j\varsigma(j) }
{\sum_{j\geq0}(j+1)\varsigma(j)},
\qquad\mbox{ and }\qquad
\Upupsilon_\varsigma
=
\frac{(q;q)_\infty(\c\d;q)_\infty}
{(\c,\d;q)_\infty}
\frac{\sum_{j\geq0}
\mathfrak p_j  \upvarpi_j\varsigma(j) }
{\sum_{j\geq0}(j+1)\varsigma(j)}.
\end{align*} 
We first show that $\Upgamma_\varsigma$ and $\Upupsilon_\varsigma $ are well-defined and finite. Indeed, as in \eqref{eq:Cauchy phij}, by Cauchy's estimate,
\begin{equation}\label{eq:estimate mathfrak pj}
    |\mathfrak p_j|\leq C_1r_1^{-j}\qquad \mbox{ for all }j\in\mathbb Z_{\geq0},\qquad \mbox{ where }r_1=(1+\varrho)/2.
\end{equation} 
In view of 
\begin{equation}\label{eq:estimate varsigma}
    \varsigma(j)\leq C_0\varrho^j,\qquad \upvarpi_j \leq 1/(\c\d;q)_{\infty}\qquad\mbox{ for all }j\in\mathbb Z_{\geq0},
\end{equation}  we conclude that the infinite sums in $\Upgamma_\varsigma$ and $\Upupsilon_\varsigma$ converge absolutely.
Since  \(\mathfrak p_j\geq0\), \(\upvarpi_j>0\), \(\mathfrak p_0=1\), and $\varsigma(0)>0$,  we know that $\Upgamma_\varsigma$ and $\Upupsilon_\varsigma$ are strictly positive.
 By Lemma~\ref{lem:identify Fourier} with \(n=x+1\) and \(r=2/(1+q)\), 
\begin{equation}\label{eq:comparison-bulk-GFH}
    G_N^\varsigma(x)
=I_{N,n}(\upvarphi_0),\qquad 
F_N^\varsigma(x)
=\upvarpi_x
\left(I_{N,n}(\upvarphi_1)+\mathcal E_{1,N,x}\right),\qquad
H_N^\varsigma(x)
=I_{N,n}(\upvarphi_2)+\mathcal E_{2,N,x},
\end{equation}
where 
\[
\mathcal E_{i,N,x}
=
\int_0^\pi(2+2\cos\theta)^N
S_i(\theta)\uplambda_x(\theta)\,d\nu(\theta)\qquad
\mbox{ for }  i\in\{1,2\},
\]
and  for all $\theta\in[0,\pi]$ and $j\in\mathbb Z_{\geq0}$, we have 
\begin{align*}
 \phi_j(\theta)
&=[t^j]\frac{(\c t,\d t;q)_\infty}
{(te^{\mathtt i\theta},te^{-\mathtt i\theta};q)_\infty}, \qquad\qquad\qquad\qquad
\uplambda_j(\theta)
=\frac{1}{2\pi\mathtt i}\oint_{|z|=\frac{2}{1+q}}
\frac{(\c z,\d z;q)_\infty}
{(ze^{\mathtt i\theta},ze^{-\mathtt i\theta};q)_\infty}
\frac{dz}{z^{j+1}}, 
 \\
h(\theta)
&=\frac{(\c\d;q)_\infty(qe^{2\mathtt i\theta};q)_\infty}
{(\c e^{\mathtt i\theta},\d e^{\mathtt i\theta};q)_\infty},\qquad\qquad\qquad\qquad
d\nu(\theta)
=
\frac{(q,\c\d;q)_\infty}{2\pi}
\frac{(e^{2\mathtt i\theta},e^{-2\mathtt i\theta};q)_\infty}
{(\c e^{\mathtt i\theta},\c e^{-\mathtt i\theta},
\d e^{\mathtt i\theta},\d e^{-\mathtt i\theta};q)_\infty}
\,d\theta ,\\
 S_0(\theta)
&=\sum_{j\geq0}\varsigma(j)
\frac{\sin((j+1)\theta)}{\sin\theta},\qquad\qquad
S_1(\theta)
=\sum_{j\geq0}\varsigma(j)\phi_j(\theta),\qquad\qquad
S_2(\theta)
=\sum_{j\geq0}
\upvarpi_j\varsigma(j)\phi_j(\theta),\\
\upvarphi_0(\theta)&=S_0(\theta),
\qquad\qquad
\upvarphi_1(\theta)=h(\theta)S_1(\theta),
\qquad\qquad
\upvarphi_2(\theta)=h(\theta)S_2(\theta).
\end{align*}
Using \eqref{eq:comparison-bulk-GFH}, we obtain
\begin{equation}\label{eq:FGH}
    \begin{split}
        F_N^\varsigma(x)-\Upgamma_\varsigma G_N^\varsigma(x)
&=
\upvarpi_x I_{N,x+1}(\psi_1)
+\upvarpi_x\mathcal E_{1,N,x}
+\mathfrak k_1(\upvarpi_x-\upvarpi_\infty)
G_N^\varsigma(x),\\
H_N^\varsigma(x)-\Upupsilon_\varsigma G_N^\varsigma(x)
&=
I_{N,x+1}(\psi_2)+\mathcal E_{2,N,x},
    \end{split}
\end{equation} 
where we denote 
\[ 
\mathfrak k_1=h(0)\frac{\sum_{j\geq0}\mathfrak p_j\varsigma(j)}{\sum_{j\geq0}(j+1)\varsigma(j)},
\qquad
\mathfrak k_2=h(0)\frac{\sum_{j\geq0}\mathfrak p_j\upvarpi_j\varsigma(j)}{\sum_{j\geq0}(j+1)\varsigma(j)},
\qquad
\psi_1=\upvarphi_1-\mathfrak k_1\upvarphi_0,
\qquad
\psi_2=\upvarphi_2-\mathfrak k_2\upvarphi_0.
\] 
Using \eqref{eq:estimate mathfrak pj}, \eqref{eq:estimate varsigma}, and \(\sum_{j\geq0}(j+1)\varsigma(j)\geq\varsigma(0)\),
\begin{equation}\label{eq:bulk-uniform-constants}
\Upgamma_\varsigma+\Upupsilon_\varsigma
+|\mathfrak k_1|+|\mathfrak k_2|
\leq C_2.
\end{equation}  We split the rest of the proof into two steps.\\

\noindent\textbf{Step 1. } 
Let $r_2= 2/(1+\max (\varrho,|\c|,|\d|))$ and \(r_3= (\varrho+r_2^{-1})/2\). 
In this step, we apply Lemma~\ref{lem:Fourier} with \(\upupsilon=\log r_2\).
For
\(r_2^{-1}\leq |w|\leq r_2\), define
\begin{align*}
    &\upkappa_j(w)
=
[t^j]\frac{(\c t,\d t;q)_\infty}{(tw,t/w;q)_\infty},
\qquad
U_j(w)=\sum_{\ell=0}^j w^{j-2\ell},\qquad
\mathscr h(w)=
\frac{(\c\d;q)_\infty(qw^2;q)_\infty}
{(\c w,\d w;q)_\infty},\\
&\mathscr S_0(w)=\sum_{j\geq0}\varsigma(j)U_j(w), \qquad
\mathscr S_1(w)=\sum_{j\geq0}\varsigma(j)\upkappa_j(w), \qquad
\mathscr S_2(w)=\sum_{j\geq0}\upvarpi_j\varsigma(j) 
\upkappa_j(w) . 
\end{align*}
Define also 
\[
\Psi_i(w)=\mathscr h(w)\mathscr S_i(w)
-\mathfrak k_i\mathscr S_0(w)=\sum_{k\in\mathbb Z}a_{i,k}w^k ,\qquad i\in\{1,2\}.
\]
Note that $\Psi_i(e^{\mathtt i\theta})=\psi_i(\theta) $ for $\theta\in\mathbb R$. It suffices to show that, for  $i\in\{1,2\}$,
\begin{equation}\label{eq:bulk-Fourier-decay}
|a_{i,k}|\leq C_3 r_2^{-|k|}
\qquad\text{ for } \ k\in\mathbb Z,\qquad\qquad\sum_{k\in\mathbb Z}a_{i,k} =0.
\end{equation}  For \(|t|=r_3\), \(r_2^{-1}\leq|w|\leq r_2\), both
\(|tw|\) and \(|t/w|\) are at most \(r_2r_3<1\). Cauchy's coefficient estimate gives
\[
\sup_{r_2^{-1}\leq|w|\leq r_2}
|\upkappa_j(w)|
\leq C_4 r_3^{-j}\qquad\mbox{ for }j\in\mathbb Z_{\geq0}.
\] Note also  that  $\sup_{r_2^{-1}\leq|w|\leq r_2}|U_j(w)|
\leq(j+1)r_2^j$ for all $j\in\mathbb Z_{\geq0}$.
Hence $\Psi_i(w)$, $i\in\{1,2\}$ are holomorphic in a neighborhood of \(r_2^{-1}\leq|w|\leq r_2\). 
Cauchy's estimates on \(|w|=r_2\) and \(|w|=r_2^{-1}\), together with \eqref{eq:bulk-uniform-constants}, give the estimates in \eqref{eq:bulk-Fourier-decay}. Furthermore, note that $\upkappa_j(1)=\mathfrak p_j$, $U_j(1)=j+1$, and $\mathscr h(1)=h(0)$.  The definitions of \(\mathfrak k_1,\mathfrak k_2\)  imply $\sum_{k\in\mathbb Z}a_{i,k}=\Psi_i(1) =0$. Since $\Psi_i(\overline w)=\overline{\Psi_i(w)}$, we have $a_{i,k}\in\mathbb R$ for all $i\in\{1,2\}$ and $k\in\mathbb Z$.
We conclude the proof of \eqref{eq:bulk-Fourier-decay}. Therefore, \(\psi_1\) and \(\psi_2\) satisfy the assumptions of Lemma~\ref{lem:Fourier} . \\

\noindent\textbf{Step 2. } In this step, we conclude the proof of the proposition. We choose $r_4=2/(1+q)$ and then 
 $\varepsilon
=
\min(
 1/4, \log(r_2)/16,
 \log r_4/(2\log4))>0$.
 Note that by the definition of $G_N^\varsigma$, we get
\begin{equation}\label{eq:bulk-G-lower}
G_N^\varsigma(x)\geq\varsigma(0)Z_N(0,x),\qquad \varsigma(0)>0.
\end{equation}
We first prove the proposition for \(n=x+1\leq\varepsilon(N+1)\). In this case, combining Lemma~\ref{lem:Fourier}, \eqref{eq:bulk-Fourier-decay}, and \eqref{eq:bulk-G-lower},
\begin{equation}\label{eq:bulk-I-over-G}
|I_{N,n}(\psi_1)|+|I_{N,n}(\psi_2)|
\leq
C_5 G_N^\varsigma(x)
\left(\frac1{x+1}+\frac{x+1}{N+1}\right).
\end{equation}

We next estimate the remainder terms. We claim that 
\begin{equation}\label{eq:bulk-error-over-G}
|\mathcal E_{i,N,x}|
\leq
\frac{C_6}{x+1}G_N^\varsigma(x)
\qquad\mbox{ for }     x+1\leq\varepsilon(N+1)\mbox{ and }i\in\{1,2\}.
\end{equation}
Indeed,   for \(|z|=r_4\) and \(\theta\in[0,\pi]\),  
 $\left|(ze^{\pm\mathtt i\theta};q)_\infty\right|
\geq
(r_4-1)(r_4q;q)_\infty
>0$  since \(r_4>1\), \(r_4q<1\).
The numerator in the definition of \(\uplambda_x\) is also uniformly bounded on \(|z|=r_4\). 
As in Lemma \ref{lem:identify Fourier},  \(S_i\) uniformly converge, hence
\[
\sup_{\theta\in[0,\pi]}|\uplambda_x(\theta)|
\leq C_7 r_4^{-x},\qquad \sup_{\theta\in[0,\pi]}|S_i(\theta)|\leq C_7,
\qquad i\in\{1,2\} .
\]
Note also that
\[
0\leq\frac{d\nu}{d\theta}(\theta)
=
\frac{2(q,\c\d;q)_\infty}{\pi}\sin^2\theta
\frac{|(qe^{2\mathtt i\theta};q)_\infty|^2}
{|(\c e^{\mathtt i\theta},\d e^{\mathtt i\theta};q)_\infty|^2}\leq C_8\sin^2\theta.
\] By Lemma~\ref{lem:count biMot},
\begin{equation}\label{eq:bulk-error-Z00}
\begin{aligned}
|\mathcal E_{i,N,x}|
&\leq
C_9 r_4^{-x}
\int_0^\pi(2+2\cos\theta)^N\sin^2\theta\,d\theta \leq C_{10} r_4^{-x}Z_N(0,0)\qquad\mbox{ for } i\in\{1,2\}.
\end{aligned}
\end{equation}
When $x+1\leq\varepsilon(N+1)$ we have \(x\leq N\). By Lemma~\ref{lem:bounds ZNxy},
\begin{equation}\label{eq:rnbtygsen}
      Z_N(0,0)
\leq
\frac{4^{N+1}}{(N+1)^{3/2}},\qquad\qquad
Z_N(0,x)\geq e^{-2}4^{N-1}
\frac{x+1}{(N+1)^{3/2}}
e^{-(\log4) x^2/(N+1) } .
\end{equation}
Note $ x^2/(N+1)\leq\varepsilon x$ and $\varepsilon\log4\leq \log(r_4)/2$. Combining \eqref{eq:bulk-G-lower}, \eqref{eq:bulk-error-Z00}, \eqref{eq:rnbtygsen}, we conclude the proof of \eqref{eq:bulk-error-over-G}.

We finally estimate \(\upvarpi_x-\upvarpi_\infty\) from \eqref{eq:PochApplied} in order to bound \eqref{eq:FGH}. We claim that for some $C_{11}>0$
\begin{equation}\label{eq:bulk-varpi-tail}
|\upvarpi_x-\upvarpi_\infty|
\leq\frac{C_{11}}{x+1}.
\end{equation}
Indeed, using the fact that for \(u\in[0,1)\), $0\leq1-(u;q)_\infty\leq  u/(1-q)$, and that \((\c\d q^x;q)_\infty\geq(\c\d;q)_\infty>0\),
\[
\left|
\frac{\upvarpi_\infty}{\upvarpi_x}-1
\right|
=
\left|\frac{(q^{x+1};q)_\infty-(\c\d q^x;q)_\infty}{(\c\d q^x;q)_\infty}\right|
\leq
\frac{|1-(q^{x+1};q)_\infty|+|1-(\c\d q^x;q)_\infty|}{(\c\d q^x;q)_\infty}
\leq C_{12}q^x.
\]
Combining with $\upvarpi_x \leq 1/(\c\d;q)_{\infty}$, we conclude the proof of \eqref{eq:bulk-varpi-tail}.

Combining \eqref{eq:FGH}, \eqref{eq:bulk-I-over-G}, \eqref{eq:bulk-error-over-G}, \eqref{eq:bulk-varpi-tail}, and using the boundedness of
\(\upvarpi_x\) and \(\mathfrak k_1\), when $x+1\leq\varepsilon(N+1)$,
\begin{equation}\label{eq:bulk-comparison-small-x}
    \max\left(|F_N^\varsigma(x)-\Upgamma_\varsigma G_N^\varsigma(x)|, \, |H_N^\varsigma(x)-\Upupsilon_\varsigma G_N^\varsigma(x)|\right)\leq C_{13} G_N^\varsigma(x)
\left(\frac1{x+1}+\frac{x+1}{N+1}\right).
\end{equation}

It remains to consider the case \(x+1>\varepsilon(N+1)\). Since every step weight in \eqref{eq:def of RNxy} belongs to \((0,1]\), we have
 $0\leq R_N^{(q,\c,\d)}(x,y)\leq Z_N(x,y)$,
and the first inequality is strict if \(|x-y|\leq N\). Note also that \(Z_N(x,y)=Z_N(y,x)\). Hence, whenever \(G_N^\varsigma(x)>0\), we have
\[
0<F_N^\varsigma(x)\leq G_N^\varsigma(x),
\qquad
0<H_N^\varsigma(x)
\leq
\sum_{z\geq0}\varsigma(z)Z_N(x,z)
=
G_N^\varsigma(x).
\]
Thus both ratios belong to \((0,1]\). In view of \eqref{eq:bulk-uniform-constants}, the two comparison bounds in the proposition hold also when \(x+1>\varepsilon(N+1)\), after enlarging the constant \(C_{13}\) by a factor depending only on \(\varepsilon^{-1}\).

Recall that we have shown $\Upgamma_\varsigma,\Upupsilon_\varsigma>0$. It remains to show that 
$\Upgamma_\varsigma,\Upupsilon_\varsigma\leq1$. For \(m\in\mathbb N\), take \(N=m^2\) and \(x=m\), then $G_{m^2}^\varsigma(m)
\geq
\varsigma(0)Z_{m^2}(0,m)>0$.  
The two comparison estimates we just proved imply
\[
\left|
\frac{F_{m^2}^\varsigma(m)}{G_{m^2}^\varsigma(m)}
-\Upgamma_\varsigma
\right|
+
\left|
\frac{H_{m^2}^\varsigma(m)}{G_{m^2}^\varsigma(m)}
-\Upupsilon_\varsigma
\right|
\leq
C\left(
\frac1{m+1}+\frac{m+1}{m^2+1}
\right)
\to0.
\]
Since we have shown that both ratios on the left side belong to \((0,1]\), taking \(m\to\infty\) gives  $\Upgamma_\varsigma,\Upupsilon_\varsigma\leq1$.
\end{proof}

\subsection{An estimate under the endpoint measures}

So far, we studied the partition function $R_N^{(q,\c,\d)}$ for bicolored Motzkin paths without including a reweighting of the endpoints. We will now include a suitable endpoint measure $\varsigma$ and show the following proposition. 

\begin{proposition}\label{prop:endpt wt compare}
Assume \(q,\a\in[0,1)\), $\b\in[0,1]$, and \(\c,\d\in(-1,0]\). Recall from Definition \ref{def:measure Motzkin} that $\uptheta(h)= (\c\d;q)_{h}/(q;q)_{h}$ for $h\in\mathbb Z_{\geq0}$. For  \(N\in\mathbb Z_{\geq0}\), let 
\[
\mathcal Z_N^{\a,\b;\uptheta}
        =
        \sum_{\bg\in\mathrm{BCMot}_N}
        \a^{\gamma_0}\b^{\gamma_N}\uptheta(\gamma_N),
\]
then $\mathcal Z_N^{\a,\b;\uptheta}\in(0,\infty)$.  Define a probability measure on \(\mathrm{BCMot}_N\) by
\[
        \mathrm{P}_N^{\a,\b;\uptheta}(\bg)
        = \a^{\gamma_0}\b^{\gamma_N}\uptheta(\gamma_N)/\mathcal Z_N^{\a,\b;\uptheta} .\]
Denote its expectation by \(\mathrm{E}_N^{\a,\b;\uptheta}\).  Then there exists a constant \(C\in(0,\infty)\), depending only on \(\a,\b,\c,\d,q\), such that, for every \(N\in\mathbb Z_{\geq2}\) and \(m\in\llbracket1,N-1\rrbracket\),
\[
\begin{aligned}
\mathrm{E}_N^{\a,\b;\uptheta}
 \left[
 \frac1{\gamma_m+1}
 +\gamma_m\left(
 \frac1m+\frac{{\bf1}_{\{\b<1\}}}{N-m}
 \right)
 \right]
 &\leq
 C\left(
 \frac1{\sqrt m}
 +\frac{{\bf1}_{\{\b<1\}}}{\sqrt{N-m}}
 \right),\\
\mathrm{E}_N^{\a,\b;\uptheta}
 \left[q^{\gamma_m}\right]
 &\leq \frac{C}{(m+1)^{3/2}}\left({\bf1}_{\{\b<1\}}\left(\frac{N+1}{N-m+1}\right)^{3/2}+{\bf1}_{\{\b=1\}}\left(\frac{N+1}{N-m+1}\right)^{1/2}\right) ,  \\
 \mathrm{E}_N^{\a,1;\uptheta}
\left[q^{\gamma_N}\right]&\leq\frac{C}{N+1}.
\end{aligned}
\]
\end{proposition}
In order to show Proposition \ref{prop:endpt wt compare}, we require the following discrete Gaussian estimates. We defer the proof to Appendix~\ref{app:Gauss}. 

\begin{lemma}\label{lem:discrete Gaussian}
For each \(s\geq1\), \(u\in[0,1)\), and \(z\in\mathbb Z_{\geq0}\), define
\begin{equation}\label{def:Gaussian}
    \mathfrak{X}_{s,u}(z)
=
\sum_{x\geq0}(x+1)u^xe^{-(x-z)^2/s}.
\end{equation}
Then there exist finite constants $C_u$ depending on $u\in[0,1)$; $C_{u,k}$ depending on $u\in[0,1)$, $k\in\mathbb Z_{\geq0}$; and $C_{u,v,k}$ depending on $u,v\in[0,1)$, $k\in\mathbb Z_{\geq0}$, such that, for each fixed $k\in\mathbb Z_{\geq0}$, \(s\geq1\), $x,y\geq2$, \(u,v\in[0,1)\),
\begin{align}
 \label{eq:endpoint-one-gaussian-moment}  \sum_{z\geq0}(z+1)^k\mathfrak{X}_{s,u}(z)
&\leq
C_{u,k} s^{(k+1)/2},\\
\label{eq:endpoint-two-gaussian-moment} \sum_{z\geq0}(z+1)^k
\mathfrak{X}_{x,u}(z)\mathfrak{X}_{y,v}(z)
&\leq
C_{u,v,k}\left(\frac{xy}{x+y}\right)^{(k+1)/2},\\
\label{eq:endpoint-free-inverse-moment} \sum_{z\geq0}\mathfrak X_{x,u}(z)\min\left(1,\frac{z+1}{\sqrt y}\right)
&\leq
C_u\frac{x}{\sqrt{x+y}},\\
\label{eq:endpoint-free-first-moment}\sum_{z\geq0}z(z+1)\mathfrak X_{x,u}(z)\min\left(1,\frac{z+1}{\sqrt y}\right)
&\leq
C_u\frac{x^2}{\sqrt{x+y}}.
\end{align} 
\end{lemma}

\begin{proof}[Proof of Proposition \ref{prop:endpt wt compare}]
Throughout this proof,  we use $C_i$, $i\in\mathbb N$ to denote constants in $(0,\infty)$ that depend only on \(\a,\b,\c,\d,q\).
Note that 
\begin{equation}\label{eq:boundsthetanmi}
0< (\c\d;q)_\infty
\leq \uptheta(h)
\leq   (q;q)_\infty^{-1}<\infty 
\qquad\mbox{ for } h\in\mathbb Z_{\geq0}. 
\end{equation} For every
\(v\in[0,1]\),
\[
\mathcal Z_N^{\a,v;\uptheta}
\leq
 (q;q)_\infty^{-1} 4^N\sum_{x\geq0}\a^x
=4^N (q;q)_\infty^{-1} (1-\a)^{-1}
<\infty.
\]
We also have \(\mathcal Z_N^{\a,v;\uptheta}>0\) since the path that remains at height zero and with all black steps has endpoint weight one. We have shown that the probability measure $\mathrm{P}_N^{\a,\b;\uptheta}$ is well-defined. 
In the remainder of the proof, we prove the claimed estimates, and we will always assume  \(N\in\mathbb Z_{\geq2}\) and \(m\in\llbracket1,N-1\rrbracket\). Denote 
\[
\mathrm A (z)=\sum_{x\geq0}\a^xZ_m(x,z),
\qquad
\mathrm B (z)=\sum_{y\geq0}\b^y\uptheta(y)Z_{N-m}(z,y),\qquad z\in\mathbb Z_{\geq0}.
\]
Concatenation at $\gamma_m$ gives, for every function \(F:\mathbb Z_{\geq0}\to[0,\infty)\), 
\begin{equation}\label{eq:endpoint-measure-cut-factorization}
\mathcal Z_N^{\a,\b;\uptheta}
\mathrm{E}_N^{\a,\b;\uptheta}[F(\gamma_m)]
=
\sum_{z\geq0}F(z)\mathrm A (z)\mathrm B (z).
\end{equation}
We split the rest of the proof into two steps.\\

\noindent\textbf{Step 1. } In this step, we prove the estimates for $\b<1$. The upper bound in
Lemma~\ref{lem:bounds ZNxy} gives, for $z\in\mathbb Z_{\geq0}$,
\begin{equation}\label{eq:endpoint-prefix-bound}
   \mathrm A (z)\leq
\frac{4^{m+1}(z+1)}{(m+1)^{3/2}}
\mathfrak X_{m+1,\a}(z),\qquad
\mathrm B (z)\leq
\frac{(q;q)_\infty^{-1} 4^{N-m+1}(z+1)}{(N-m+1)^{3/2}}
\mathfrak X_{N-m+1,\b}(z), 
\end{equation}
where we recall \eqref{def:Gaussian}.
By the lower bound in Lemma~\ref{lem:bounds ZNxy},  
$\mathcal Z_N^{\a,\b;\uptheta}
\geq Z_N(0,0)
\geq
e^{-2} 4^{N-1}/(N+1)^{3/2}$. Combining this with \eqref{eq:endpoint-prefix-bound}, for every function \(F:\mathbb Z_{\geq0}\to[0,\infty)\), we have
\begin{equation}\label{eq:endpoint-pinned-master-bound}
\mathrm{E}_N^{\a,\b;\uptheta}[F(\gamma_m)]
\leq
C_1\frac{(N+1)^{3/2}}{(m+1)^{3/2}(N-m+1)^{3/2}}
\sum_{z\geq0}
F(z)(z+1)^2
\mathfrak X_{m+1,\a}(z)\mathfrak X_{N-m+1,\b}(z).
\end{equation}
Using the estimate \eqref{eq:endpoint-two-gaussian-moment} from Lemma~\ref{lem:discrete Gaussian} with $k\in\{1,3\}$, 
\begin{align*}
   \sum_{z\geq0}(z+1)
\mathfrak X_{m+1,\a}(z)\mathfrak X_{N-m+1,\b}(z) &\leq C_2\frac{(m+1)(N-m+1)}{N+2},\\
\sum_{z\geq0}(z+1)^3
\mathfrak X_{m+1,\a}(z)\mathfrak X_{N-m+1,\b}(z)&\leq C_2\frac{(m+1)^2(N-m+1)^2}{(N+2)^2}.
\end{align*} 
Taking \(F(z)=(z+1)^{-1}\) in
\eqref{eq:endpoint-pinned-master-bound},  
\[
\mathrm{E}_N^{\a,\b;\uptheta}
\left[\frac1{\gamma_m+1}\right]
\leq
C_3\frac{(N+1)^{3/2} }{(N+2)(m+1)^{1/2}(N-m+1)^{1/2}} 
\leq
C_4\left(\frac1{\sqrt m}+\frac1{\sqrt{N-m}}\right).
\]
Similarly,  
\[
\begin{aligned} 
\mathrm{E}_N^{\a,\b;\uptheta}
\left[\gamma_m\left(\frac1m+\frac1{N-m}\right)\right]
 &\leq
C_5\frac{(N+1)^{3/2}(m+1)^{1/2}(N-m+1)^{1/2}}{(N+2)^2 }
\left(\frac1m+\frac{1}{N-m}\right) \\& \leq
C_6\left(\frac1{\sqrt m}+\frac1{\sqrt{N-m}}\right).
\end{aligned}
\]
We conclude the proof of the first asserted estimate in Proposition~\ref{prop:endpt wt compare} for \(\b<1\).

We next prove the second asserted estimate when $\b<1$. Note that
\begin{equation}\label{eq:X cross}
    \mathfrak X_{\sigma,u}(z)
\leq
\sum_{x\geq0}(x+1)u^x
= (1-u)^{-2}\qquad\mbox{ for each }u\in[0,1), \, \sigma\geq1, \, z\geq0.
\end{equation} Combining with the fact that $\sum_{z\geq0}(z+1)^2q^z \leq C_7$ for some $C_7>0$,  
\[
\sum_{z\geq0}q^z(z+1)^2
\mathfrak X_{m+1,\a}(z)\mathfrak X_{N-m+1,\b}(z)
\leq C_8.
\]
We take \(F(z)=q^z\) in \eqref{eq:endpoint-pinned-master-bound} and obtain
\[
\mathrm{E}_N^{\a,\b;\uptheta}[q^{\gamma_m}]
\leq 
\frac{C_9(N+1)^{3/2}}{(m+1)^{3/2}(N-m+1)^{3/2}} .
\]
We conclude the proof of the second estimate for \(\b<1\).\\

\noindent\textbf{Step 2. } In this step, we prove the estimates for $\b=1$. By Lemma~\ref{lem:bounds ZNxfree} and \eqref{eq:boundsthetanmi},  
\begin{align}
    \label{eq:endpoint-suffix-free-bound}\mathrm B (z)
&\leq
(q;q)_{\infty}^{-1} Z_{N-m}(z,\mathrm{free})
\leq
2(q;q)_{\infty}^{-1}   4^{N-m}
\min\left(1,\frac{z+1}{\sqrt{N-m+1}}\right),\\
\label{eq:endpoint-free-normalizer-lower}\mathcal Z_N^{\a,1;\uptheta}
&\geq
(\c\d;q)_\infty\sum_{y\geq0}Z_N(0,y)
=
(\c\d;q)_\infty Z_N(0,\mathrm{free})
\geq
(\c\d;q)_\infty e^{-2}\frac{4^{N-1}}{\sqrt{N+1}}.
\end{align} 
Combining \eqref{eq:endpoint-measure-cut-factorization},
\eqref{eq:endpoint-prefix-bound},
\eqref{eq:endpoint-suffix-free-bound}, and
\eqref{eq:endpoint-free-normalizer-lower}, we obtain, for every \(F:\mathbb Z_{\geq0}\to[0,\infty)\),
\begin{equation}\label{eq:endpoint-free-master-bound}
\mathrm{E}_N^{\a,1;\uptheta}[F(\gamma_m)]
\leq
C_{10}\frac{\sqrt{N+1}}{(m+1)^{3/2}}
\sum_{z\geq0}
F(z)(z+1)\mathfrak X_{m+1,\a}(z)
\min\left(1,\frac{z+1}{\sqrt{N-m+1}}\right).
\end{equation} 
Using the estimates \eqref{eq:endpoint-free-inverse-moment} and \eqref{eq:endpoint-free-first-moment} from
Lemma~\ref{lem:discrete Gaussian} with \(x=m+1\), \(y=N-m+1\), and \(u=\a\), and taking the function $F(z)=\frac{1}{z+1}+\frac{z}{m}$
\[
\begin{aligned}
\mathrm{E}_N^{\a,1;\uptheta}
\left[\frac1{\gamma_m+1}+\frac{\gamma_m}{m}\right]
\leq
C_{11}\frac{\sqrt{N+1}}{(m+1)^{3/2}\sqrt{N+2}}
\left(m+1+\frac{(m+1)^2}{m}\right) 
\leq
\frac{C_{12}}{\sqrt m}.
\end{aligned}
\]
We conclude the proof of the first asserted estimate for \(\b=1\).
We also have 
\[
\begin{aligned}
&\sum_{z\geq0}q^z(z+1)\mathfrak X_{m+1,\a}(z)
\min\left(1,\frac{z+1}{\sqrt{N-m+1}}\right) \leq
\frac1{\sqrt{N-m+1}(1-\a)^2}
\sum_{z\geq0}q^z(z+1)^2
\leq
\frac{C_{13}}{\sqrt{N-m+1}}.
\end{aligned}
\]
Taking \(F(z)=q^z\) in \eqref{eq:endpoint-free-master-bound} gives
\[
\mathrm{E}_N^{\a,1;\uptheta}[q^{\gamma_m}]
\leq 
\frac{C_{14}}{(m+1)^{3/2}}
\left(\frac{N+1}{N-m+1}\right)^{1/2}.
\]

It remains to prove the terminal height estimate when \(\b=1\). Note that 
$\sum_{y\geq0}(y+1)q^y
\leq
\sum_{y\geq0}\mathfrak X_{1,q}(y)
\leq C_{15}$. By the upper bound in Lemma~\ref{lem:bounds ZNxy}, \eqref{eq:X cross}, and
\eqref{eq:boundsthetanmi},
\[
\begin{aligned}
\mathcal Z_N^{\a,1;\uptheta}
\mathrm{E}_N^{\a,1;\uptheta}[q^{\gamma_N}] =
\sum_{x,y\geq0}
\a^xq^y\uptheta(y)Z_N(x,y) \leq
(q;q)^{-1}_{\infty}\frac{4^{N+1}}{(N+1)^{3/2}}
\sum_{y\geq0}q^y(y+1)\mathfrak X_{N+1,\a}(y)
\leq
C_{16}\frac{4^N}{(N+1)^{3/2}},
\end{aligned}
\]
 Dividing by
\eqref{eq:endpoint-free-normalizer-lower} gives
$\mathrm{E}_N^{\a,1;\uptheta}[q^{\gamma_N}]
\leq  C/(N+1)$,
as desired.
\end{proof}

\subsection{Proof of Theorem \ref{thm:main open ASEP}} In this subsection, we prove Theorem \ref{thm:main open ASEP}. For convenience, we restate Theorem \ref{thm:main open ASEP} as Theorem \ref{thm:main open ASEP again} below.
\begin{theorem}\label{thm:main open ASEP again}
    Assume $q\in[0,1)$, $\a,\b\geq0$ and $\mathsf{c},\mathsf{d}\in(-1,0]$. For $N\in\mathbb N$, consider the stationary measure $\mu_N^{(q,\a,\b,\c,\d)}$ of open ASEP. There exists $C<\infty$ depending only on $\a,\b,\c,\d$ and $q$, such that the following hold.
    \begin{enumerate}
        \item [(1)] In the maximal current phase $\a,\b\in[0,1)$, for every $N,m,k\in\mathbb N$ satisfying 
        $1\le m\le N-k+1\le N$,  
\[
\TV{
\mu_N^{(q,\a,\b,\c,\d)}
\left(
(\tau_i)_{i\in\llbracket m,N-k+1\rrbracket}\in\cdot
\right)
-
\mathbb{P}_N^{\mathsf{exc}}
\left(
(\tau_i)_{i\in\llbracket m,N-k+1\rrbracket}\in\cdot
\right)}
\le
C/\min(m,k)^{1/2} .
\]
        \item [(2)] On the boundary of the maximal current phase $\a\in[0,1)$ and $\b=1$, for every $N\in\mathbb N$ and $m\in\llbracket1,N\rrbracket$, 
        \[
\TV{
\mu_N^{(q,\a,1,\c,\d)}
\left(
(\tau_i)_{i\in\llbracket m,N \rrbracket}\in\cdot
\right)
-
\mathbb P_N^{\mathsf{mea}}
\left(
(\tau_i)_{i\in\llbracket m,N \rrbracket}\in\cdot
\right)}
\le
C/m^{1/2} .
\]
        \item [(3)] On the boundary of the maximal current phase $\a=1$ and $\b\in[0,1)$, for every $N\in\mathbb N$ and $k\in\llbracket1,N\rrbracket$, 
        \[
\TV{
\mu_N^{(q,1,\b,\c,\d)}
\left(
(\tau_i)_{i\in\llbracket 1,N-k+1 \rrbracket}\in\cdot
\right)
-
\mathbb P_N^{\mathsf{revmea}}
\left(
(\tau_i)_{i\in\llbracket 1,N-k+1 \rrbracket}\in\cdot
\right)}
\le
C/k^{1/2} .
\]
    \end{enumerate} 
\end{theorem}

For the proof, we will use the following useful upper bound on the total variation distance.

\begin{lemma}\label{lem:TV RN}
    Let $\upmu$ be a probability measure on a measurable space $(X,\mathcal F)$. Let $f:X\to[0,\infty]$ be a measurable function satisfying 
    $\int_Xfd\upmu\in(0,\infty)$. Define a new probability measure $\upnu$ by 
    \[
\upnu (A)
=
\frac{\int_A f\,d\upmu}{\int_X f\,d\upmu}\qquad\mbox{ for each }A\in\mathcal F.
\] Then, for each \(h>0\), we have
\[
\TV{\upnu-\upmu}
\leq
\frac{2}{h}
\int_X|f-h|\,d\upmu.
\]
\end{lemma}
\begin{proof}
    If $\int|f-h|\,d\upmu\geq h/2$, the proof follows from that the total variation is bounded by $1$. Otherwise,
\[
\int f\,d\upmu
\geq
h-\int|f-h|\,d\upmu
\geq
\frac{h}{2},
\]
and
\[
\begin{aligned}
&\TV{\upnu-\upmu}
 =
\frac1{2\int f\,d\upmu}
\int
\left|
f-\int f\,d\upmu
\right|d\upmu\\
&\leq
\frac1{2\int f\,d\upmu}
\left(
\int|f-h|\,d\upmu
+
\left|\int f\,d\upmu-h\right|
\right) \leq
\frac{\int|f-h|\,d\upmu}{\int f\,d\upmu}
\leq
\frac2{h}
\int|f-h|\,d\upmu.
\end{aligned}\] This concludes the proof.
\end{proof}

\begin{proof}[Proof of Theorem \ref{thm:main open ASEP again}]
Throughout this proof, we write $C_i$, $i\in\mathbb N$ to denote constants in $(0,\infty)$ depending only on $\a,\b,\c,\d$ and $q$.
As in the proof of Theorem \ref{thm:main open TASEP again}, for \(x,y\in\mathbb Z_{\geq0}\) with \(|x-y|\leq N\), we denote by \(\Uptheta_N^{x,y}\) (resp. \(\Uptheta_N^{x,\mathrm{free}}\)) the pushforward under \(\Phi\) of the uniform probability measure on \(\mathrm{BCMot}_N^{x,y}\) (resp. \(\mathrm{BCMot}_N^{x,\mathrm{free}}\)). 
For \(|x-y|>N\), set \(\Uptheta_N^{x,y}=\Uptheta_N^{0,0}\).  This convention does not affect any formula below because \(Z_N(x,y)=0\) if \(|x-y|>N\).
For \(I\subseteq\llbracket1,N\rrbracket\), a subscript \(I\) denotes the $(\tau_i:i\in I)$ marginal of a measure on $(\tau_1,\dots,\tau_N)$. 
We now recall the bulk weights from Definition \ref{def:measure Motzkin}: for each $\bg\in\mathrm{BCMot}_N$, $i\in\llbracket1,N\rrbracket$,  
    \[
    \upomega_i(\bg)=\begin{cases}
         1-q^{\gamma_i}&\qquad\mbox{if } \gamma_i-\gamma_{i-1}=1,\\
         1-\c\d q^{\gamma_i}&\qquad\mbox{if } \gamma_i-\gamma_{i-1}=-1,\\
         1+\d q^{\gamma_i} &\qquad\mbox{if } \gamma_i-\gamma_{i-1}=0\mbox{ is colored black},\\
         1+\c q^{\gamma_i} &\qquad\mbox{if } \gamma_i-\gamma_{i-1}=0\mbox{ is colored white}.
    \end{cases}
    \]
In this proof, we denote 
\[
W_{s,t}(\bg)=\prod_{i\in\llbracket s,t\rrbracket} \upomega_i(\bg)\qquad\mbox{ for each }1\leq s\leq t\leq N.
\] 
Recall the boundary weight $\uptheta(h)=(\c\d;q)_h/(q;q)_h$ for  $h\in\mathbb Z_{\geq0}$.
By Theorem \ref{thm:Motzkin representation}, 
 \[Z_N^{(q,\a,\b,\c,\d)}=\sum_{\bg\in\mathrm{BCMot}_N}\a^{\gamma_0}\b^{\gamma_N} \uptheta(\gamma_N) W_{1,N}(\bg)\in(0,\infty),\]
and hence  
\begin{equation}\label{eq:path tilt measure}
    \mathbb{P}_N^{(q,\a,\b,\c,\d)}(\bg)=\a^{\gamma_0}\b^{\gamma_N} \uptheta(\gamma_N) W_{1,N}(\bg)/Z_N^{(q,\a,\b,\c,\d)},\qquad\bg\in\mathrm{BCMot}_N,
\end{equation}
is a well-defined probability measure. Since $\a\b<1$  in all the cases of this theorem, by Theorem \ref{thm:Motzkin representation}, 
\begin{equation}\label{eq:pushf stat}
\Phi_\#\mathbb{P}_N^{(q,\a,\b,\c,\d)}
=
\mu_N^{(q,\a,\b,\c,\d)}.
\end{equation}
We next recall the probability measure $\mathrm{P}_N^{\a,\b;\uptheta}$.
By Proposition \ref{prop:endpt wt compare}, when $\a\in[0,1)$ and $\b\in[0,1]$,
\[
\mathcal Z_N^{\a,\b;\uptheta}
        =
        \sum_{\bg\in\mathrm{BCMot}_N}
        \a^{\gamma_0}\b^{\gamma_N}\uptheta(\gamma_N)\in(0,\infty),
\] and hence
\begin{equation}\label{eq:endpoint tilt measure}
        \mathrm{P}_N^{\a,\b;\uptheta}(\bg)
        = \a^{\gamma_0}\b^{\gamma_N}\uptheta(\gamma_N)/\mathcal Z_N^{\a,\b;\uptheta},\qquad\bg\in \mathrm{BCMot}_N,\end{equation}
is a well-defined probability measure.  Denote its expectation by \(\mathrm{E}_N^{\a,\b;\uptheta}\).
Combining \eqref{eq:path tilt measure} and   \eqref{eq:endpoint tilt measure},
\begin{equation}\label{eq:full-density-ASEP-main}
\frac{d \mathbb{P}_N^{(q,\a,\b,\c,\d)}}{d\mathrm{P}_N^{\a,\b;\uptheta}}
=
\frac{W_{1,N}}
{\mathrm{E}_N^{\a,\b;\uptheta}[W_{1,N}]}.
\end{equation}

We will use the following estimates for the bulk weights
\begin{equation}\label{eq:estimate bulk weights}
    0\leq1-W_{s,t}(\bg)
\leq
\sum_{i\in\llbracket s,t\rrbracket}   q^{\gamma_i}\qquad\mbox{ for each } \bg\in \mathrm{BCMot}_N,\,1\leq s\leq t\leq N.
\end{equation}
Indeed, this follows from the fact that for every local step, $0\leq1-\upomega_i\leq q^{\gamma_i}$.

We will also use the following estimates for the boundary weight $\uptheta$. Set $\uptheta_\infty
= (\c\d;q)_\infty/(q;q)_\infty$. We have
\begin{equation}\label{eq:boundsthetanmi again}
    0< (\c\d;q)_\infty
\leq \uptheta(h)
\leq   (q;q)_\infty^{-1}<\infty ,\qquad \left|
\uptheta(h)-\uptheta_\infty
\right|
\leq
C_1q^h,
\qquad\mbox{ for } h\in\mathbb Z_{\geq0}. 
\end{equation}
Indeed, the first estimate is  \eqref{eq:boundsthetanmi}. The second one follows from $\uptheta(h)=\uptheta_\infty (q^{h+1};q)_\infty/(\c\d q^h;q)_\infty$, in view of 
\[
0\leq1-(\c\d q^h;q)_\infty
\leq \c\d q^h (1-q)^{-1},
\quad
0\leq1-(q^{h+1};q)_\infty
\leq q^{h+1} (1-q)^{-1},\quad(\c\d q^h;q)_\infty
\geq
(\c\d;q)_\infty>0.
\]

For clarity, we split the rest of the proof into two steps.\\

\noindent\textbf{Step 1. } In this step, we prove Statement (1). Throughout this step, we assume  $\a,\b\in[0,1)$ and $N,m,k\in\mathbb N$ satisfying  $1\le m\le N-k+1\le N$, and we set $I=\llbracket m,N-k+1\rrbracket$. If $\min(m,k)=1$, the estimate follows from the fact that the total variation distance is at most $1$. We assume $m,k\geq2$ for the rest of this step.

We first prove the claim that
\begin{equation}\label{eq:endpoint-excursion-comparison-ASEP-main}
\TV{
\left(\Phi_\#\mathrm{P}_N^{\a,\b;\uptheta}\right)_I
-
\Uptheta_{N,I}^{0,0}}
\leq
C_2\left(\frac1m+\frac1k\right).
\end{equation}
Indeed, by the definition \eqref{eq:endpoint tilt measure} of the measure $\mathrm{P}_N^{\a,\b;\uptheta}$,   
\begin{equation}\label{eq:ndrsty}
\left(\Phi_\#\mathrm{P}_N^{\a,\b;\uptheta}\right)_I
-
\Uptheta_{N,I}^{0,0}
=
\frac1{\mathcal Z_N^{\a,\b;\uptheta}}
\sum_{x,y\geq0}
\a^x\b^y\uptheta(y)Z_N(x,y)
\left(
\Uptheta_{N,I}^{x,y}
-
\Uptheta_{N,I}^{0,0}
\right).
\end{equation}
As in \eqref{eq:TASEP-two-sided-bridge-comparison} in the proof of Theorem \ref{thm:main open TASEP again}, for each $x,y\in\mathbb Z_{\geq0}$ with \(|x-y|\leq N\) and \(y\leq N\),
\begin{equation}\label{eq:uniform-bridge-comparison-ASEP-main}
\TV{\Uptheta_{N,I}^{x,y}-\Uptheta_{N,I}^{0,0}}
\leq
2^{26}\left(
\frac{(x+1)^4(y+1)^2}{m}
+
\frac{(y+1)^4}{k}
\right).
\end{equation}
By Lemma \ref{lem:bounds ZNxy}, 
\begin{equation}\label{eq:endpoint-excursion-normalizer-ASEP-main}
\mathcal Z_N^{\a,\b;\uptheta}
\geq
Z_N(0,0)
\geq   C_3\frac{4^N}{(N+1)^{3/2}}.
\end{equation}
Using \eqref{eq:boundsthetanmi again}, \eqref{eq:uniform-bridge-comparison-ASEP-main}, and Lemma \ref{lem:bounds ZNxy},
\begin{equation*}
    \begin{split}
        &\sum_{x=0}^{\infty}\sum_{y=0}^N
\a^x\b^y\uptheta(y)Z_N(x,y)\TV{
\Uptheta_{N,I}^{x,y}
-
\Uptheta_{N,I}^{0,0} }\\
&\leq \frac{C_44^N}{(N+1)^{3/2}}
\left[
\frac1m
\left(\sum_{x\geq0}(x+1)^5\a^x\right)
\left(\sum_{y\geq0}(y+1)^3\b^y\right)
+
\frac1k
\left(\sum_{x\geq0}(x+1)\a^x\right)
\left(\sum_{y\geq0}(y+1)^5\b^y\right)
\right].
    \end{split}
\end{equation*}
By Lemma \ref{lem:bounds ZNxy}, the fact that the  total variation distance is bounded by $1$, and a generous upper bound for the second inequality using $\a, \b < 1$, 
\begin{equation*}
    \begin{split}
        &\sum_{x=0}^{\infty}\sum_{y=N+1}^{\infty}
\a^x\b^y\uptheta(y)Z_N(x,y)\TV{
\Uptheta_{N,I}^{x,y}
-
\Uptheta_{N,I}^{0,0} }\\
&\leq\frac{C_54^N}{(N+1)^{3/2}}
\left(\sum_{x\geq0}(x+1)\a^x\right)
\left(\sum_{y>N}(y+1)\b^y\right)
\leq
\frac{C_64^N}{(N+1)^{5/2}}.
    \end{split}
\end{equation*} 
Adding the last two displayed equations and dividing by \eqref{eq:endpoint-excursion-normalizer-ASEP-main}, using \eqref{eq:ndrsty}, we conclude the proof of \eqref{eq:endpoint-excursion-comparison-ASEP-main}.

Next, we let
 $$\mathcal G_{m,N-k+1}
=
\sigma\! (
\gamma_i:\ i\in\llbracket m-1,N-k+1\rrbracket;\
\mathbf 1_{\{\gamma_i=\gamma_{i-1}\ {\rm and\ the\ }i{\rm th\ step\ is\ black}\}}:
\ i\in I ).$$
 Moreover, we define
\[
\varsigma_L(h)=\a^h,
\qquad
\varsigma_R(h)=\b^h\uptheta(h),
\qquad h\in\mathbb Z_{\geq0}.
\]
We claim that,
\(\mathrm{P}_N^{\a,\b;\uptheta}\)-almost surely
\begin{equation}\label{eq:conditional-density-ASEP-main}
    D_{m,N-k+1}
=
\mathrm{E}_N^{\a,\b;\uptheta}
\left[
W_{1,N}\mid\mathcal G_{m,N-k+1}
\right]
=
\frac{
F_{m-1}^{\varsigma_L}(\gamma_{m-1})
}{
G_{m-1}^{\varsigma_L}(\gamma_{m-1})
}
W_{m,N-k+1}
\frac{
H_{k-1}^{\varsigma_R}(\gamma_{N-k+1})
}{
G_{k-1}^{\varsigma_R}(\gamma_{N-k+1})
}.
\end{equation}
Indeed, for each $u,v\in\mathbb Z_{\geq0}$,
\begin{align*}
   & G_{m-1}^{\varsigma_L}(u)
=
\sum_{x\geq0}\a^xZ_{m-1}(x,u),\qquad
F_{m-1}^{\varsigma_L}(u)=\sum_{x\geq0}\a^xR^{(q,\c,\d)}_{m-1}(x,u),\\
&G_{k-1}^{\varsigma_R}(v)=\sum_{y\geq0}\b^y\uptheta(y)Z_{k-1}(v,y),\qquad H_{k-1}^{\varsigma_R}(v)=\sum_{y\geq0}\b^y\uptheta(y)R^{(q,\c,\d)}_{k-1}(v,y).
\end{align*} Factoring the weighted sums independently over the left and
right path pieces gives us \eqref{eq:conditional-density-ASEP-main}.

By Proposition \ref{prop:comparison bulk weights}, there exist constants $\Upgamma_{\varsigma_L}, \Upupsilon_{\varsigma_R}\in(0,1]$,  such that
\begin{align}
&\left|
D_{m,N-k+1}
-
\Upgamma_{\varsigma_L}\Upupsilon_{\varsigma_R}
\right| \leq
\left|
\frac{
F_{m-1}^{\varsigma_L}(\gamma_{m-1})
}{
G_{m-1}^{\varsigma_L}(\gamma_{m-1})
}
-
\Upgamma_{\varsigma_L}
\right|
+
\left|
\frac{
H_{k-1}^{\varsigma_R}(\gamma_{N-k+1})
}{
G_{k-1}^{\varsigma_R}(\gamma_{N-k+1})
}
-
\Upupsilon_{\varsigma_R}
\right|
+
1-W_{m,N-k+1}.
\label{eq:three-errors-ASEP-main}
\end{align}

Using the first estimate of Proposition \ref{prop:endpt wt compare} at times $m-1,N-k+1\in\llbracket1,N-1\rrbracket$ and the
$q^{\gamma_i}$ estimate for every $i\in\llbracket m,N-k+1\rrbracket$,  we obtain
\begin{equation}\label{eq:in proof apply endpt wt compare}
    \begin{split}
        \mathrm{E}_N^{\a,\b;\uptheta}
\left[
\frac1{\gamma_{m-1}+1}
+
\frac{\gamma_{m-1}+1}{m}
\right]
&\leq
C_7\left(
\frac1{\sqrt m}
+
\frac1{\sqrt k}
\right), \\
\mathrm{E}_N^{\a,\b;\uptheta}
\left[
\frac1{\gamma_{N-k+1}+1}
+
\frac{\gamma_{N-k+1}+1}{k}
\right]
&\leq
C_8\left(
\frac1{\sqrt m}
+
\frac1{\sqrt k}
\right),\\
\mathrm{E}_N^{\a,\b;\uptheta}
\left[
\sum_{i=m}^{N-k+1} q^{\gamma_i}
\right]
&\leq
C_9\sum_{i=m}^{N-k+1}
\frac{(N+1)^{3/2}}
{(i+1)^{3/2}(N-i+1)^{3/2}}\leq
C_{10}\left(
\frac1{\sqrt m}
+
\frac1{\sqrt k}
\right).
    \end{split}
\end{equation}
Here, in the last inequality, we split  at \(N/2\); on the left half, the summand is at most \(C_9(i+1)^{-3/2}\), whose tail from \(m\) is \(O(m^{-1/2})\); on the right half, it is at most \(C_9(N-i+1)^{-3/2}\), whose tail down to \(k\) is \(O(k^{-1/2})\).

Combining Proposition \ref{prop:comparison bulk weights} with \eqref{eq:estimate bulk weights}, \eqref{eq:three-errors-ASEP-main}, and \eqref{eq:in proof apply endpt wt compare}, we obtain
\[
\mathrm{E}_N^{\a,\b;\uptheta}
\left|
D_{m,N-k+1}
-
\Upgamma_{\varsigma_L}\Upupsilon_{\varsigma_R}
\right|
\leq
C_{11}\left(
\frac1{\sqrt m}
+
\frac1{\sqrt k}
\right).
\]

Taking conditional expectations in \eqref{eq:full-density-ASEP-main}, we know that the restriction of \(\mathbb P_N^{(q,\a,\b,\c,\d)}\) to \(\mathcal G_{m,N-k+1}\) has Radon--Nikodym derivative
 $D_{m,N-k+1}/\mathrm{E}_N^{\a,\b;\uptheta}[D_{m,N-k+1}]$
with respect to the restriction of \(\mathrm{P}_N^{\a,\b;\uptheta}\), where
 $\mathrm{E}_N^{\a,\b;\uptheta}[D_{m,N-k+1}]
=
\mathrm{E}_N^{\a,\b;\uptheta}[W_{1,N}]
\in(0,\infty)$.
Note   the map
 $\Phi_I(\bg)
= (\Phi(\bg)_i )_{i\in I}$
is \(\mathcal G_{m,N-k+1}\)-measurable. Lemma \ref{lem:TV RN} with \(f=D_{m,N-k+1}\), \(h=\Upgamma_{\varsigma_L}\Upupsilon_{\varsigma_R}>0\), followed by contraction of the total variation distance under \(\Phi_I\), gives
\begin{equation*} 
\TV{
\left(\Phi_\#\mathbb{P}_N^{(q,\a,\b,\c,\d)}\right)_I
-
\left(\Phi_\#\mathrm{P}_N^{\a,\b;\uptheta}\right)_I}
\leq
C_{12}\left(
\frac1{\sqrt m}
+
\frac1{\sqrt k}
\right).
\end{equation*}
Combining this with
\eqref{eq:endpoint-excursion-comparison-ASEP-main}, \eqref{eq:pushf stat}, and Lemma
\ref{lem:exc and mea projection},  we conclude the proof of Statement (1). \\

\noindent\textbf{Step 2. } In this step, we prove the rest of the statements. Note that we only need to prove Statement (2), and then Statement (3) follows from Statement (2) through the particle-hole duality as in Lemma \ref{lem:particle hole}. Assume \(\a\in[0,1)\) and \(\b=1\), fix
\(m\in\llbracket1,N\rrbracket\), and set \(I=\llbracket m,N\rrbracket\).
If $m=1$, the   estimate follows from the fact that the total variation is at most $1$. We assume $m\geq2$ for the rest of this step.
We first claim that
\begin{equation}\label{eq:endpoint-meander-comparison-ASEP-main}
\TV{
\left(\Phi_\#\mathrm{P}_N^{\a,1;\uptheta}\right)_I
-
\Uptheta_{N,I}^{0,\mathrm{free}}
}
\leq \frac{C_{13}}{m}.
\end{equation}
Indeed, using \eqref{eq:endpoint tilt measure}, we obtain
 \begin{equation}\label{eq:endpoint-meander-decomposition-ASEP-main}
    \begin{split}
&
\left(\Phi_\#\mathrm{P}_N^{\a,1;\uptheta}\right)_I
-
\Uptheta_{N,I}^{0,\mathrm{free}} \\  
&=
\frac{1}{\mathcal Z_N^{\a,1;\uptheta}}
\left[
\uptheta_\infty
\sum_{x\geq0}\a^x Z_N(x,\mathrm{free})
\left(
\Uptheta_{N,I}^{x,\mathrm{free}}
-
\Uptheta_{N,I}^{0,\mathrm{free}}
\right) 
+
\sum_{x,y\geq0}
\a^x\left(\uptheta(y)-\uptheta_\infty\right)Z_N(x,y)
\left(
\Uptheta_{N,I}^{x,y}
-
\Uptheta_{N,I}^{0,\mathrm{free}}
\right)
\right].
\end{split}
\end{equation}
In view of \eqref{eq:boundsthetanmi again}, Lemmas~\ref{lem:bounds ZNxy} and \ref{lem:bounds ZNxfree}, and \(\a,q\in[0,1)\), both series are absolutely convergent in total variation.
By \eqref{eq:boundsthetanmi again} and Lemma~\ref{lem:bounds ZNxfree},
\begin{equation}\label{eq:endpoint-meander-normalizer-ASEP-main}
\mathcal Z_N^{\a,1;\uptheta}
\geq
(\c\d;q)_\infty Z_N(0,\mathrm{free})
\geq
C_{14}\frac{4^N}{\sqrt{N+1}}.
\end{equation}
As in \eqref{eq:TASEP-free-suffix-comparison} in the proof of Theorem \ref{thm:main open TASEP again},  Proposition \ref{prop:marginal law comparison} gives
\begin{equation*} 
\TV{
\Uptheta_{N,I}^{x,\mathrm{free}}
-
\Uptheta_{N,I}^{0,\mathrm{free}}
}
\leq
2^{26}\frac{(x+1)^2}{m},
\qquad x\in\mathbb Z_{\geq0}.
\end{equation*}
Combining this with the upper bound in Lemma~\ref{lem:bounds ZNxfree},
\[
\begin{aligned}
&\uptheta_\infty
\sum_{x\geq0}\a^x Z_N(x,\mathrm{free})
\TV{
\Uptheta_{N,I}^{x,\mathrm{free}}
-
\Uptheta_{N,I}^{0,\mathrm{free}}
} \leq
C_{15}\frac{4^N}{m\sqrt{N+1}}
\sum_{x\geq0}(x+1)^3\a^x
\leq
C_{16}\frac{4^N}{m\sqrt{N+1}}.
\end{aligned}
\]
On the other hand, using
\(\TV{\Uptheta_{N,I}^{x,y}-\Uptheta_{N,I}^{0,\mathrm{free}}}\leq1\),
the second estimate in \eqref{eq:boundsthetanmi again}, and
Lemma~\ref{lem:bounds ZNxy},
\[
\begin{aligned}
&
\sum_{x,y\geq0}
\a^x\left|\uptheta(y)-\uptheta_\infty\right|Z_N(x,y)
\TV{
\Uptheta_{N,I}^{x,y}
-
\Uptheta_{N,I}^{0,\mathrm{free}}
}\\
& \leq
C_{17}\frac{4^N}{(N+1)^{3/2}}
\left(\sum_{x\geq0}(x+1)\a^x\right)
\left(\sum_{y\geq0}(y+1)q^y\right)
\leq
C_{18}\frac{4^N}{(N+1)^{3/2}}.
\end{aligned}
\]
Dividing the last two bounds by
\eqref{eq:endpoint-meander-normalizer-ASEP-main}, in view of
\eqref{eq:endpoint-meander-decomposition-ASEP-main}, we conclude the proof of
\eqref{eq:endpoint-meander-comparison-ASEP-main}.

Let \(\mathcal G_{m,N}\) be the sigma field generated by
\(\gamma_i\), \(i\in\llbracket m-1,N\rrbracket\), and by the colors of
the horizontal steps in \(\llbracket m,N\rrbracket\). Set
\(\varsigma_L(h)=\a^h\). Concatenation at \(\gamma_{m-1}\) gives,
\(\mathrm{P}_N^{\a,1;\uptheta}\)-almost surely,
\begin{equation}\label{eq:conditional-density-meander-ASEP-main}
\begin{aligned}
D_{m,N}
=
\mathrm{E}_N^{\a,1;\uptheta}
\left[W_{1,N}\mid\mathcal G_{m,N}\right]=
\frac{
F_{m-1}^{\varsigma_L}(\gamma_{m-1})
}{
G_{m-1}^{\varsigma_L}(\gamma_{m-1})
}
W_{m,N}.
\end{aligned}
\end{equation} Since all bulk weights, the ratio
\(F_{m-1}^{\varsigma_L}/G_{m-1}^{\varsigma_L}\), and
\(\Upgamma_{\varsigma_L}\) belong to \((0,1]\), by Proposition
\ref{prop:comparison bulk weights} and
\eqref{eq:estimate bulk weights},
\begin{equation}\label{eq:tnyhed}
\left|D_{m,N}-\Upgamma_{\varsigma_L}\right|
\leq
C_{19}\left(
\frac{1}{\gamma_{m-1}+1}
+
\frac{\gamma_{m-1}+1}{m}
\right)
+
\sum_{i=m}^Nq^{\gamma_i}.
\end{equation}

Using the first estimate of Proposition \ref{prop:endpt wt compare} at time $m-1$, its $q^{\gamma_i}$ estimate for every $i\in\llbracket m,N-1\rrbracket$, and its $\gamma_N$ estimate, we obtain
\begin{equation}\label{eq:apply prop compare bdry weight in proof}
    \begin{split}
        \mathrm{E}_N^{\a,1;\uptheta}
\left[
\frac{1}{\gamma_{m-1}+1}
+
\frac{\gamma_{m-1}+1}{m}
\right]
&\leq
\frac{C_{20}}{\sqrt m},\\
\mathrm{E}_N^{\a,1;\uptheta}
\left[\sum_{i=m}^Nq^{\gamma_i}\right]
&\leq
C_{21}
\sum_{i=m}^{N-1}
\frac{1}{(i+1)^{3/2}}
\left(\frac{N+1}{N-i+1}\right)^{1/2}
+
\frac{C_{21}}{N+1} \leq
\frac{C_{22}}{\sqrt m}.
    \end{split}
\end{equation} 
Here, in the last inequality, we split the sum at \(N/2\); on the left half the summand is at most \(C(i+1)^{-3/2}\); and the right sum is at most
$C/(N+1)\cdot\sum_{r=1}^{N+1}r^{-1/2}
\leq
 C/\sqrt{N+1}
\leq C/\sqrt m$.
Combining \eqref{eq:tnyhed} with \eqref{eq:apply prop compare bdry weight in proof}, 
\[
\mathrm{E}_N^{\a,1;\uptheta}
\left|D_{m,N}-\Upgamma_{\varsigma_L}\right|
\leq
\frac{C_{23}}{\sqrt m}.
\]

Taking conditional expectations in \eqref{eq:full-density-ASEP-main}, we know that the restriction of \(\mathbb P_N^{(q,\a,1,\c,\d)}\) to \(\mathcal G_{m,N}\) has
Radon--Nikodym derivative
 $ D_{m,N}/\mathrm{E}_N^{\a,1;\uptheta}[D_{m,N}] $
with respect to the restriction of
\(\mathrm{P}_N^{\a,1;\uptheta}\), where
 $\mathrm{E}_N^{\a,1;\uptheta}[D_{m,N}]
=
\mathrm{E}_N^{\a,1;\uptheta}[W_{1,N}]
\in(0,\infty)$.
Note also that
$\Phi_I(\bg) = (\Phi(\bg)_i )_{i\in I}$
is \(\mathcal G_{m,N}\)-measurable. Using Lemma~\ref{lem:TV RN} 
with \(f=D_{m,N}\) and \(h=\Upgamma_{\varsigma_L}>0\), followed by
contraction of total variation under \(\Phi_I\), we obtain
\[
\TV{
\left(\Phi_\#\mathbb P_N^{(q,\a,1,\c,\d)}\right)_I
-
\left(\Phi_\#\mathrm{P}_N^{\a,1;\uptheta}\right)_I
}
\leq
\frac{C_{24}}{\sqrt m}.
\] 
Combining this with
\eqref{eq:endpoint-meander-comparison-ASEP-main}, \eqref{eq:pushf stat}, and Lemma
\ref{lem:exc and mea projection},  we conclude the proof of Statement (2).  As already mentioned, Statement (3) then follows from the particle-hole duality in Lemma \ref{lem:particle hole}.
\end{proof}

\section{Open ASEP in the high and low density phases: Proof of Theorem \ref{thm:main open ASEP HLD}}\label{sec:HighLow}

In this section, we show the approximation of the stationary measure of the open ASEP in the high density phase and the lower density phase from Theorem \ref{thm:main open ASEP HLD}, which we state below again for clarity.

\begin{theorem}\label{thm:main open ASEP HLD again}     Assume $q\in[0,1)$, $\a,\b\geq0$ and $\mathsf{c},\mathsf{d}\in(-1,0]$. For $N\in\mathbb N$, consider the stationary measure $\mu_N^{(q,\a,\b,\c,\d)}$ of the open ASEP. There exist $C,c\in(0,\infty)$ depending only on $\a,\b,\c,\d,q$, such that the following hold.
    \begin{enumerate}
        \item [(1)] In the high density phase of the fan region, i.e., $\b>1$ and $\a\b<1$,  for every $N\in\mathbb N$ and $m\in\llbracket1,N\rrbracket$, 
        \[
           \TV{ \mu_N^{(q,\a,\b,\c,\d)}
    \left( (\tau_i)_{i\in\llbracket m,N\rrbracket}\in\cdot\,\right)
    -\ber_{\llbracket m,N\rrbracket}\left( \b/(1+\b)\right)}
    \le C e^{-cm}.
        \]
                \item [(2)] In the low density phase of the fan region, i.e., $\a>1$ and $\a\b<1$,  for every $N\in\mathbb N$ and $k\in\llbracket1,N\rrbracket$, 
        \[
           \TV{ \mu_N^{(q,\a,\b,\c,\d)}
    \left( (\tau_i)_{i\in\llbracket 1,N-k+1\rrbracket}\in\cdot\,\right)
    -\ber_{\llbracket 1,N-k+1\rrbracket}\left( 1/(1+\a)\right)}
    \le C e^{-ck}.
        \]\end{enumerate}
Moreover, in Statement~\textup{(1)} the constant \(c\) may be chosen to depend only on \(q,\b\), and in Statement~\textup{(2)} it may be chosen to depend only on \(q,\a\).
\end{theorem}

\begin{definition}\label{def:generalized Motzkin}
    Fix $N\in\mathbb{Z}_{\geq0}$. A bicolored path of length $N$ is a sequence $\bg=(\gamma_0,\dots,\gamma_N)$ such that $\gamma_i\in\mathbb{Z}$ for $i\in\llbracket0,N\rrbracket$ with increments $\gamma_i-\gamma_{i-1}\in\{-1,0,1\}$ for $i\in\llbracket1,N\rrbracket$. Each horizontal step of the path, i.e., $\gamma_i-\gamma_{i-1}=0$, is colored either black or white; the $\pm1$ steps are uncolored. Denote by $\mathrm{BCPath}_N$ the space of bicolored paths of length $N$. Then the space of bicolored Motzkin paths $\mathrm{BCMot}_N$ is the subset of $\mathrm{BCPath}_N$ containing all the paths that stay nonnegative. 

    Assume $\a,\b\geq0$, $\b>1$ and $\a\b<1$. We define a probability measure $\mathbb{Q}=\mathbb{Q}_N^{(\a,\b)}$ on $\mathrm{BCPath}_N$ as follows. We first sample $\gamma_0$ according to the geometric distribution $\mathbb{Q}(\gamma_0=x)=(1-\a\b)(\a\b)^x$ for each $x\in\mathbb Z_{\geq0}$. We then let the steps $\gamma_i-\gamma_{i-1}$ for each $i\in\llbracket1,N\rrbracket$ be i.i.d. and also independent from $\gamma_0$, with distribution 
\begin{align*}
    &\mathbb{Q}(\gamma_i-\gamma_{i-1}=1)=\b^2(1+\b)^{-2},\qquad\qquad\mathbb{Q}(\gamma_i-\gamma_{i-1}=-1)= (1+\b)^{-2},\\
    & \mathbb{Q}(\gamma_i-\gamma_{i-1}=0\mbox{ is colored black})=\mathbb{Q}(\gamma_i-\gamma_{i-1}=0\mbox{ is colored white})=\b(1+\b)^{-2}.
\end{align*} 
\end{definition}
We record the following properties under the measure $\mathbb{Q}$.

\begin{lemma}\label{lem:realization HD phase}
Fix $N\in\mathbb Z_{\geq0}$. Assume $q\in[0,1)$, $\a,\b\geq0$, and $\c,\d\in(-1,0]$ satisfy $\b>1$ and $\a\b<1$. Let $\mathbb Q=\mathbb Q_N^{(\a,\b)}$ be the probability measure from Definition~\ref{def:generalized Motzkin}, and let $\upomega_i$ and $\uptheta$ be the weights from Definition~\ref{def:measure Motzkin}. Define $\mathsf Y:\mathrm{BCPath}_N\to[0,\infty)$ by 
\[
\mathsf Y(\bg)=\uptheta(\gamma_N)\prod_{i=1}^N\upomega_i(\bg)\qquad\mbox{ for }\bg\in\mathrm{BCMot}_N
\]
and $\mathsf Y(\bg)=0$ for $\bg\notin\mathrm{BCMot}_N$. Then  we have
\begin{align*}
\mathbb E_{\mathbb Q}[\mathsf Y]
&=(1-\a\b)\b^N(1+\b)^{-2N}
  Z_N^{(q,\a,\b,\c,\d)},\\
 \mathbb Q(\bg)\mathsf Y(\bg)/\mathbb E_{\mathbb Q}[\mathsf Y] 
&=\mathbb P_N^{(q,\a,\b,\c,\d)}(\bg) 
\qquad \mbox{ for each }\bg\in\mathrm{BCMot}_N,
\end{align*}
where   $Z_N^{(q,\a,\b,\c,\d)}$ and $\mathbb P_N^{(q,\a,\b,\c,\d)}$ are as in Definition~\ref{def:measure Motzkin}.
\end{lemma} 
\begin{proof}
We compute that $\mathbb Q(\bg)=(1-\a\b)\b^N(1+\b)^{-2N}\a^{\gamma_0}\b^{\gamma_N}$ for each $\bg\in\mathrm{BCMot}_N$, and therefore
$\mathbb Q(\bg)\mathsf Y(\bg)
=(1-\a\b)\b^N(1+\b)^{-2N}
\mathrm{wt}_N^{(q,\a,\b,\c,\d)}(\bg)$ for each $\bg\in\mathrm{BCMot}_N$. The conclusions follow by \eqref{def:FullWeightMeasure}.
\end{proof}

\begin{lemma}\label{lem:normalization HD phase}
Assume the same notation as in Lemma \ref{lem:realization HD phase}. Then $\inf_{N\in\mathbb N}\mathbb{E}_{\mathbb Q}[\mathsf Y]>0$.
\end{lemma}
\begin{proof}
    We compute   $\mathbb E_{\mathbb Q}[\gamma_1-\gamma_0]= (\b-1)/(\b+1)$. Throughout this proof, we denote $\upsilon=(\b-1)/(2(\b+1))$. Observe that the function 
    $f(t)=\mathbb E_{\mathbb Q} [e^{-t(\gamma_1-\gamma_0-\upsilon)} ]$ is differentiable at zero, with $f(0)=1$ and $f'(0)=-\upsilon<0$, hence there exists $s>0$ satisfying $f(s)<1$. Choose a large enough integer $K\in\mathbb N$ satisfying $e^{-s(1-\upsilon)K}f(s)/(1-f(s))\leq 1/2$. We now consider the event $E_N$ given by, if $0\leq N\leq K$,
 \[
E_N= \left\{
\gamma_0\geq0,\quad\ 
\gamma_i-\gamma_{i-1}=1\text{ for }i\in\llbracket1,N\rrbracket
\right\};
 \]
and if $N>K$,
\[
E_N= \left\{
\gamma_0\geq0,\quad\ 
\gamma_i-\gamma_{i-1}=1\text{ for }i\in\llbracket1,K\rrbracket,\quad\,\sum_{\ell=K+1}^{K+j}(\gamma_\ell-\gamma_{\ell-1}-\upsilon)
\geq-(1-\upsilon)K 
\text{ for }j\in\llbracket1,N-K\rrbracket
\right\}.
\]
We first claim that
\begin{equation}\label{eq:inf of EN}
    \inf_{N\in\mathbb Z_{\geq0}}
\mathbb Q_N(E_N)
\geq   \b^{2K}/\left(2(1+\b)^{2K}\right)         .
\end{equation}
Indeed, when $0\leq N\leq K$, we have 
$\mathbb Q_N(E_N)
= ( \b^2/(1+\b)^2 )^N
\geq   \b^{2K}/\left(2(1+\b)^{2K}\right)    $. When $N>K$, by a union bound and the exponential Markov inequality,
\begin{equation*}
    \begin{split}
        &\mathbb Q_N\left(
\min_{j\in\llbracket1, N-K\rrbracket}
\sum_{\ell=K+1}^{K+j}(\gamma_\ell-\gamma_{\ell-1}-\upsilon)
<-(1-\upsilon)K
\right)\leq \sum_{j=1}^{N-K}\mathbb Q_N\left(
\sum_{\ell=K+1}^{K+j}(\gamma_\ell-\gamma_{\ell-1}-\upsilon)
<-(1-\upsilon)K
\right) \\
&\leq \sum_{j=1}^{N-K} e^{-s(1-\upsilon)K}
\mathbb E_{\mathbb Q_N}\left[
\exp\left(
-s\sum_{\ell=K+1}^{K+j}(\gamma_\ell-\gamma_{\ell-1}-\upsilon)
\right)\right] \leq \sum_{j=1}^{N-K}e^{-s(1-\upsilon)K}f(s)^j \leq
e^{-s(1-\upsilon)K}\frac{f(s)}{1-f(s)}
\leq\frac12. 
    \end{split}
\end{equation*} 
Since the first $K$ increments are independent of the remaining increments, we conclude the proof of \eqref{eq:inf of EN}.

We next claim that, setting $\mathcal{R}^{\upsilon}=\prod_{i=K+1}^{\infty}(1-q^{\upsilon i})\in(0,1 ] $, we have
\begin{equation}\label{eq:lower bound YN}
\mathsf Y_N
\geq
(\c\d;q)_\infty(q;q)_K\mathcal{R}^{\upsilon}
\qquad\text{ on }E_N.
\end{equation}
Indeed, on $E_N$, we have $\gamma_i=\gamma_0+i\geq i$ for $i\in\llbracket1,\min (K,N)\rrbracket$. If $N>K$, then the definition of $E_N$ gives  $\gamma_i\geq \upsilon i$ for $i\in\llbracket K, N\rrbracket$.  Therefore
for each $\bg\in E_N$, $\prod_{i=1}^N\upomega_i(\bg) \geq\prod_{i=1}^N(1-q^{\gamma_i})\geq(q;q)_K\mathcal{R}^{\upsilon}$. 
Moreover, as in \eqref{eq:boundsthetanmi again}, $0< (\c\d;q)_\infty
\leq \uptheta(h)
\leq   (q;q)_\infty^{-1}<\infty $ for each  $h\in\mathbb Z_{\geq0}$.
Combining these two estimates proves \eqref{eq:lower bound YN}. Combining \eqref{eq:inf of EN} and \eqref{eq:lower bound YN}, for every $N\in\mathbb Z_{\geq0}$,
\[
\mathbb E_{\mathbb Q_N}[\mathsf Y_N]
\geq
\b^{2K}/\left(2(1+\b)^{2K}\right)\cdot
(\c\d;q)_\infty(q;q)_K
\prod_{i=K+1}^{\infty}(1-q^{\upsilon i})
>0.
\]
We conclude the proof of the lemma.
\end{proof}

Next, we compare $\mathsf Y (\bg)$ to the respective weights under a truncation of the path $\bg$.
 
\begin{lemma}\label{lem:compare RN}
Assume the same notation as in Lemma \ref{lem:realization HD phase}. For \(m\in\llbracket0,N\rrbracket\) and \(\bg=(\gamma_0,\ldots,\gamma_N)\in\mathrm{BCPath}_N\), let
$\bg^{(m)}=(\gamma_0,\ldots,\gamma_m)$ be the truncated path, with the colors of its horizontal steps retained. Let
\[\mathsf L_m(\bg)= \frac{(\c\d;q)_\infty}{(q;q)_\infty}\cdot
\prod_{i=1}^{m}\upomega_i\left(\bg^{(m)}\right)\] when $\bg^{(m)}\in\mathrm{BCMot}_m$, and $\mathsf L_m(\bg)=0$ when $\bg^{(m)}\notin\mathrm{BCMot}_m$.
Then there exist some \(\upeta\in(0,1)\), depending only on \(q,\b\), and some  \(C\in(0,\infty)\), depending only on \(q,\a,\b,\c,\d\), such that
\[
\mathbb E_{\mathbb Q_N^{(\a,\b)}}
\left[\left|\mathsf Y-\mathsf L_m\right|\right]
\leq C\upeta^m\qquad\mbox{ for each }N\in\mathbb Z_{\geq0},\,m\in\llbracket0,N\rrbracket.
\]
\end{lemma}
\begin{proof}
    Denote $\uptheta_\infty= (\c\d;q)_\infty/(q;q)_\infty$. As in \eqref{eq:boundsthetanmi again}, 
\begin{equation} \label{eq: theta bounds}
    0< (\c\d;q)_\infty
\leq \uptheta(h)
\leq   (q;q)_\infty^{-1}<\infty ,\qquad \left|
\uptheta(h)-\uptheta_\infty
\right|
\leq
C_1q^h,
\qquad\mbox{ for } h\in\mathbb Z_{\geq0},
\end{equation}
where \(C_1\in(0,\infty)\) depends only on \(q,\a,\b,\c,\d\). For \(j\in\llbracket0,N\rrbracket\), let
\[
A_j
=
\left\{
\gamma_i\geq0
\text{ for every }i\in\llbracket0,j\rrbracket
\right\}.
\]
We claim that 
\begin{equation}\label{eq:compare-RN-three-errors}
    |\mathsf Y-\mathsf L_m|
\leq
\uptheta_\infty
{\bf1}_{A_m\setminus A_N}
+
{\bf1}_{A_N}
\left|\uptheta(\gamma_N)-\uptheta_\infty\right|  +
\uptheta_\infty{\bf1}_{A_N}
\left(
1-\prod_{i=m+1}^{N}\upomega_i(\bg)
\right).
\end{equation}
Indeed, on \(A_m^{\complement}\), both \(\mathsf Y\) and \(\mathsf L_m\) vanish; on \(A_m\setminus A_N\), one has \(\mathsf Y=0\) and
\(0\leq\mathsf L_m\leq\uptheta_\infty\); on \(A_N\), we have  \(\bg\in\mathrm{BCMot}_N\), \(\bg^{(m)}\in\mathrm{BCMot}_m\), 
$\upomega_i (\bg^{(m)} )=\upomega_i(\bg)$ for $i\in\llbracket1,m\rrbracket$, and therefore
\[
\mathsf Y
=
\left(\prod_{i=1}^{m}\upomega_i(\bg)\right)
\uptheta(\gamma_N)
\prod_{i=m+1}^{N}\upomega_i(\bg),
\qquad
\mathsf L_m
=
\left(\prod_{i=1}^{m}\upomega_i(\bg)\right)
\uptheta_\infty;
\]
combining these three cases gives us \eqref{eq:compare-RN-three-errors}. 

Let $\upvarsigma=\max(q,\b^{-1})$ and $\upeta= (\b^2\upvarsigma+2\b+\upvarsigma^{-1})(1+\b)^{-2}\in(0,1)$. For \(i\in\llbracket1,N\rrbracket\),  
$\mathbb E_{\mathbb Q}[\upvarsigma^{\gamma_i-\gamma_{i-1}}]=\upeta$. Hence, 
\begin{equation}\label{eq:compare-RN-height-moment}
    \mathbb E_{\mathbb Q}[\upvarsigma^{\gamma_i}] =
\mathbb E_{\mathbb Q}[\upvarsigma^{\gamma_0}]
\prod_{j=1}^{i}
\mathbb E_{\mathbb Q}[\upvarsigma^{\gamma_j-\gamma_{j-1}}]=\frac{1-\a\b}{1-\a\b\upvarsigma}\,\upeta^i\leq
\upeta^i\qquad\mbox{ for }i\in\llbracket0,N\rrbracket.
\end{equation} 

We next estimate the expecations of the terms on the right side of \eqref{eq:compare-RN-three-errors}. First,
\begin{equation}\label{eq:compare-RN-crossing}
\mathbb Q(A_m\setminus A_N)\leq \sum_{i=m+1}^{N}\mathbb Q(\gamma_i<0)
\leq
\sum_{i=m+1}^{N}
\mathbb E_{\mathbb Q}[\upvarsigma^{\gamma_i}]
\leq
\sum_{i=m+1}^{\infty}\upeta^i
=
\frac{\upeta^{m+1}}{1-\upeta}.
\end{equation}
Second, on \(A_N\) one has \(\gamma_N\geq0\), and hence \(q^{\gamma_N}\leq\upvarsigma^{\gamma_N}\). Combining \eqref{eq: theta bounds} and  \eqref{eq:compare-RN-height-moment} give 
\begin{equation}\label{eq:compare-RN-terminal-error}
\mathbb E_{\mathbb Q}
\left[
{\bf1}_{A_N}
\left|\uptheta(\gamma_N)-\uptheta_\infty\right|
\right]
\leq
C_1\upeta^N
\leq
C_1\upeta^m.
\end{equation}
Finally, on \(A_N\),
 $1-\prod_{i=m+1}^{N}\upomega_i(\bg)
\leq
\sum_{i=m+1}^{N}q^{\gamma_i}
\leq
\sum_{i=m+1}^{N}\upvarsigma^{\gamma_i}$, hence \eqref{eq:compare-RN-height-moment}  gives
\begin{equation}\label{eq:compare-RN-bulk-error}
\mathbb E_{\mathbb Q}
\left[
{\bf1}_{A_N}
\left(
1-\prod_{i=m+1}^{N}\upomega_i(\bg)
\right)
\right]
\leq
\sum_{i=m+1}^{\infty}\upeta^i
=
\frac{\upeta^{m+1}}{1-\upeta}.
\end{equation} 

Taking expectations in \eqref{eq:compare-RN-three-errors} and applying
\eqref{eq: theta bounds},
\eqref{eq:compare-RN-crossing},
\eqref{eq:compare-RN-terminal-error}, and
\eqref{eq:compare-RN-bulk-error}, we obtain
\[
\mathbb E_{\mathbb Q_N^{(\a,\b)}}
\left[
|\mathsf Y-\mathsf L_m|
\right]
\leq
\left(
C_1+
\frac{2\upeta}
{(q;q)_\infty(1-\upeta)}
\right)\upeta^m.
\] 
Note that $\upeta\in(0,1)$ depends only on $q$ and $\b$, we conclude the proof.
\end{proof}

\begin{proof}[Proof of Theorem \ref{thm:main open ASEP HLD again}]
It suffices to prove the approximation in the high density phase $\b>1$ and $\a\b<1$, as the result for the low density phase follows from the  particle-hole duality in Lemma~\ref{lem:particle hole}. For the rest of the proof, we assume $\b>1$, $\a\b<1$, $N\geq m \geq1$, and set $I=\llbracket m,N\rrbracket$. 
    We use the notation from Lemma~\ref{lem:compare RN}. By Lemma~\ref{lem:normalization HD phase}, and by Lemma~\ref{lem:compare RN} with $m-1$ instead of $m$,  respectively,
    \begin{equation}\label{eq:inproofHD1}
        \mathbb{E}_{\mathbb Q_N^{(\a,\b)}}[\mathsf Y]\geq C_1,\qquad\qquad\mathbb E_{\mathbb Q_N^{(\a,\b)}}
\left[\left|\mathsf Y-\mathsf L_{m-1}\right|\right]
\leq C_2\upeta^{m-1},
    \end{equation}  
where \(\upeta\in(0,1)\) depends only on \(q,\b\), and $C_1,C_2\in(0,\infty)$ depend only on $q,\a,\b,\c,\d$. 
We next observe that $\mathbb E_{\mathbb Q}[\mathsf L_{m-1}]\in(0,\infty)$, since $\mathsf L_{m-1}\in[0,(\c\d;q)_\infty/(q;q)_\infty]$, and on the positive-\(\mathbb Q\)-probability event $\left\{\gamma_i-\gamma_{i-1}=1\text{ for all }i\in\llbracket1,m-1\rrbracket\right\}$ we have $\mathsf L_{m-1}>0$. Let $\mathrm Q_{N,m}$ denote the probability measure with density $\mathsf L_{m-1}/\mathbb E_{\mathbb Q}[\mathsf L_{m-1}]$ with respect to $\mathbb Q$.
Note that the map $\Phi$ from Definition~\ref{def:bimot project to particles}, originally defined on $\mathrm{BCMot}_N$, can be extended by the   same rule to $\mathrm{BCPath}_N$. We also regard $\mathbb P_N^{(q,\a,\b,\c,\d)}$ as a probability measure on $\mathrm{BCPath}_N$ by assigning mass zero outside $\mathrm{BCMot}_N$.
We have
\[
\left(\Phi_{\#}\mathrm Q_{N,m}\right)_I
=
\left(\Phi_{\#}\mathbb Q\right)_I
=
\ber_I\left( \b/(1+\b)\right),
\]
since the density $\mathsf L_{m-1}/\mathbb E_{\mathbb Q}[\mathsf L_{m-1}]$ is measurable with respect to $\gamma_0$ and the colored increments $\gamma_i-\gamma_{i-1}$, $i\in\llbracket1,m-1\rrbracket$. 
Combining this with Theorem \ref{thm:Motzkin representation}, Lemma \ref{lem:realization HD phase}, and \eqref{eq:inproofHD1},
\[
\begin{aligned}
\Big\lVert
\mu_N^{(q,\a,\b,\c,\d)}
\left((\tau_i)_{i\in I}\in\cdot\right)
&-
\ber_I\left( \b/(1+\b\right) \Big\rVert_{\text{TV}}
=
\TV{
\left(\Phi_{\#}\mathbb P_N^{(q,\a,\b,\c,\d)}\right)_I
-
\left(\Phi_{\#}\mathrm Q_{N,m}\right)_I}
\\
&\leq
\TV{
\mathbb P_N^{(q,\a,\b,\c,\d)}
-
\mathrm Q_{N,m}} =\frac12\int
\left|
\frac{\mathsf Y}{\mathbb E[\mathsf Y]}
-
\frac{\mathsf L_{m-1}}{\mathbb E[\mathsf L_{m-1}]}
\right|\,d\mathbb Q \\
&\leq
\frac{\mathbb E[|\mathsf Y-\mathsf L_{m-1}|]}{\mathbb E[\mathsf Y]}
+
\frac{|\mathbb E[\mathsf Y]-\mathbb E[\mathsf L_{m-1}]|}{\mathbb E[\mathsf Y]}  \leq
\frac{2\mathbb E[|\mathsf Y-\mathsf L_{m-1}|]}{\mathbb E[\mathsf Y]}\leq 2C_2\upeta^{m-1}/C_1,
\end{aligned}
\] 
and we conclude the proof of the theorem in the high density phase. As mentioned, the theorem in the low density phase follows from the particle-hole duality as in Lemma \ref{lem:particle hole}.
\end{proof}

\section{Converse result for open TASEP in the maximal current phase: Proof of Theorem \ref{thm:converse TASEP} }\label{sec:Converse}

We now show that the upper bounds on the total variation distance in Theorem~\ref{thm:main open TASEP} are sharp. This is the content of Theorem \ref{thm:converse TASEP}, which we state below again for clarity.

\begin{theorem}\label{thm:converse TASEP again} 
For each $x,y\in\mathbb R\setminus\{1\}$, we denote $\eta(x,y)=x(1-x)^{-2}+y(1-y)^{-2}$. Assume $\a,\b\in[0,1)$ and $\c,\d\in(-1,0]$ satisfying  $\eta(\a,\c)\eta(\b,\d)\neq0$. Then there exists a constant $C>0$ depending only on $\a,\b,\c,\d$, such that for every $m,k\in\mathbb N$ there exists an integer $N_0\geq m+k-1$ such that for each $N\geq N_0$,
            \[
\TV{
\mu_{N}^{(0,\a,\b,\c,\d)}
\left(
(\tau_i)_{i\in\llbracket m,N-k+1\rrbracket}\in\cdot
\right)
-
\mathbb{P}_{N}^{\mathsf{exc}}
\left(
(\tau_i)_{i\in\llbracket m,N-k+1\rrbracket}\in\cdot
\right)}
\ge
C/\min(m,k) .
\] 
\end{theorem}

As a key ingredient, we use the  matrix product ansatz representation for open TASEP as in the proof of Theorem \ref{thm:TASEP rep}, where
\begin{equation}\label{eq:DEWV TASEP rep again}
    \begin{split}
            &     \mathbf{D}= \begin{bmatrix}
  1 & 1 & 0 & 0 &\dots\\
  0 & 1 & 1 &0 &\dots\\
  0& 0 & 1 &  1 \\
   \vdots & & \ddots&\ddots \\
\end{bmatrix},\qquad
   \mathbf{E}= \begin{bmatrix}
  1  & 0 & 0 & \dots& \\
  1  & 1  & 0&\dots \\
  0& 1  & 1  &  &\\
   \vdots & & \ddots & \ddots\\
\end{bmatrix}, \\
& \langle W | =[\upchi_0(\a,\c),\upchi_1(\a,\c),\upchi_2(\a,\c),\dots], \quad\quad\quad
|V\rangle= [\upchi_0(\b,\d),\upchi_1(\b,\d),\upchi_2(\b,\d),\dots]^T,
    \end{split}
\end{equation} 
and we recall that $\upchi_j(u,v)=\sum_{\ell=0}^{j}u^{j-\ell}v^\ell$ for every $u,v\in\mathbb R$ and $j\in\mathbb Z_{\geq0}$.
Denote $\mathbf C=\mathbf D+\mathbf E$.
Then the stationary measure of open TASEP is given by 
\begin{equation}\label{eq:MPA again}
\mu_N^{(0,\a,\b,\c,\d)}(\tau_1,\dots,\tau_N)=\frac{1}{\left\langle W\big{|} 
\mathbf{C}^N
\big{|}V\right\rangle} \left\langle W\Big{|}
\prod_{j=1}^N
\left(
(1-\tau_j)\mathbf E
+
\tau_j\mathbf D
\right)
\Big{|}V\right\rangle \qquad \mbox{for }\tau_1,\dots,\tau_N\in\{0,1\}. \end{equation}

The following lemma relates the stationary measure $\mu_N^{(0,\a,\b,\c,\d)}$ to a suitable probability measure $\mu_{\mathrm{HL}}$ on the half-line $\{0,1\}^{\mathbb N}$, which can be interpreted as the limiting measure of $\mu_N^{(0,\a,\b,\c,\d)}$ when taking $N\rightarrow \infty$.

\begin{lemma}\label{lem:limit to half space}
   Assume $\a,\b\in[0,1)$ and $\c,\d\in(-1,0]$. Suppose $\mathbf D$, $\mathbf E$, $\langle W|$, $|V\rangle$ are given by \eqref{eq:DEWV TASEP rep again}, and let
$|H\rangle= [1,2,3,\dots]^T$.
Then there exists a unique probability measure on $\{0,1\}^{\mathbb N}$, denoted by $\mu_{\mathrm{HL}}^{\a,\c}$, such that  
\[
\mu_{\mathrm{HL}}^{\a,\c}((\tau_1,\dots,\tau_n)=(\upepsilon_1,\dots,\upepsilon_n))=4^{-n}\frac{(1-\a)^2(1-\c)^2}{1-\a\c }\Big\langle W\Big{|}
\prod_{j=1}^n
\left(
(1-\upepsilon_j)\mathbf E
+
\upepsilon_j\mathbf D
\right)
\Big{|}H\Big{\rangle} ,
\]
for each $n\in\mathbb N$ and $\upepsilon_1,\dots,\upepsilon_n\in\{0,1\}$.
Moreover, for each $n\in\mathbb N$,
\[
\lim_{N\to\infty}\TV{ \mu_N^{(0,\a,\b,\c,\d)}\left((\tau_1,\dots,\tau_n)\in\,\cdot\,\right) - \mu_{\mathrm{HL}}^{\a,\c}\left((\tau_1,\dots,\tau_n)\in\,\cdot\,\right) }=0.
\]
\end{lemma}
\begin{proof}
    We first note that it suffices to show that, for each $n\in\mathbb N$ and $\upepsilon_1,\dots,\upepsilon_n\in\{0,1\}$,
    \begin{equation}\label{eq:converse-short-halfline-limit}
\lim_{N\to\infty}
\mu_N^{(0,\a,\b,\c,\d)}
\left((\tau_1,\dots,\tau_n)=(\upepsilon_1,\dots,\upepsilon_n)\right)
=4^{-n}\frac{(1-\a)^2(1-\c)^2}{1-\a\c }\Big\langle W\Big{|}
\prod_{j=1}^n
\left(
(1-\upepsilon_j)\mathbf E
+
\upepsilon_j\mathbf D
\right)
\Big{|}H\Big{\rangle}.
\end{equation}
Indeed, assume that \eqref{eq:converse-short-halfline-limit} holds and denote its right side by \(p_n(\upepsilon_1,\dots,\upepsilon_n)\). Then \(p_n\) is a probability measure and satisfies
\[
\sum_{\upepsilon_{n+1}\in\{0,1\}}
p_{n+1}(\upepsilon_1,\dots,\upepsilon_n,\upepsilon_{n+1})
=
p_n(\upepsilon_1,\dots,\upepsilon_n)\qquad\mbox{ for each }n\in\mathbb N,\, \upepsilon_1,\dots,\upepsilon_n \in\{0,1\} .
\] 
Kolmogorov's extension theorem gives the probability measure $\mu_{\mathrm{HL}}^{\a,\c}$ on \(\{0,1\}^{\mathbb N}\) with desired properties.

We next show that
\begin{equation}\label{eq:converse-short-edge}
\lim_{R\to\infty}\frac{R^{3/2}}{4^R}\left(\mathbf C^R|V\rangle\right)_x
=
\frac4{\sqrt\pi}(x+1)
\frac{1-\b\d}{(1-\b)^2(1-\d)^2}\qquad\mbox{ for each }x\in\mathbb Z_{\geq0}.
\end{equation}
Indeed, note $(\mathbf C^R)_{x,y}=Z_R(x,y)$ for $x,y\in\mathbb Z_{\geq0}$, and each $R \in \mathbb{N}$. Note also  
$\sum_{y\geq0}\upchi_y(\b,\d)z^y
= (1-\b z)^{-1}(1-\d z)^{-1} $. Hence, one can check that
\[
\begin{aligned}
\sum_{y\geq0}\upchi_y(\b,\d)\sin((y+1)\theta)
&=
\operatorname{Im}
\frac{e^{\mathtt{i}\theta}}
{(1-\b e^{\mathtt{i}\theta})(1-\d e^{\mathtt{i}\theta})} =
\frac{(1-\b\d)\sin\theta}
{(1-2\b\cos\theta+\b^2)
 (1-2\d\cos\theta+\d^2)},
\end{aligned}
\]
where the sum absolutely converges since $|\upchi_y(\b,\d)|\leq(y+1)\max(\b,|\d|)^y$.
Using Lemma~\ref{lem:count biMot}, 
\begin{align*}
\left(\mathbf C^R|V\rangle\right)_x
&=\frac2\pi\int_0^\pi(2+2\cos\theta)^R
\sin((x+1)\theta)\sin\theta\,
\uprho_{\b,\d}(\theta)\,d\theta,\\
\uprho_{\b,\d}(\theta)
&=\frac{1-\b\d}
{(1-2\b\cos\theta+\b^2)(1-2\d\cos\theta+\d^2)}.
\end{align*}
Since \(|\b|,|\d|<1\), the function \(\uprho_{\b,\d}\) is bounded and continuous on \([0,\pi]\). Taking \(\theta=t/\sqrt R\),
\begin{equation}\label{eq:bnyrtn}
\frac{R^{3/2}}{4^R}\left(\mathbf C^R|V\rangle\right)_x
=
\frac2\pi\int_0^\infty
{\bf1}_{\{t\leq\pi\sqrt R\}}\,
R\cos^{2R}\!\left(\frac{t}{2\sqrt R}\right)
\sin\!\left(\frac{(x+1)t}{\sqrt R}\right)
\sin\!\left(\frac{t}{\sqrt R}\right)
\uprho_{\b,\d}\!\left(\frac{t}{\sqrt R}\right)\,dt.
\end{equation}
Using $\cos u\leq e^{-u^2/2}$ for $u\in[0,\pi/2]$ and $|\sin s|\leq|s|$, the absolute value of the integrand is bounded by $C(\b,\d)(x+1)t^2e^{-t^2/4}$, where $C(\b,\d)$ depends only on $\b$ and $\d$. 
In particular, we have
\begin{equation}\label{eq:converse-short-edge-uniform}
\sup_{R\geq1}\frac{R^{3/2}}{4^R}\left|\left(\mathbf C^R|V\rangle\right)_x\right|
\leq C(\b,\d)(x+1),
\qquad x\in\mathbb Z_{\geq0}.
\end{equation}
For each fixed \(t\geq0\), the integrand on the right side of \eqref{eq:bnyrtn} converges to $(x+1)t^2e^{-t^2/4}\uprho_{\b,\d}(0)$.
Since $\int_0^\infty t^2e^{-t^2/4}\,dt=2\sqrt\pi$,
$\uprho_{\b,\d}(0)
= (1-\b\d)(1-\b)^{-2}(1-\d)^{-2}$, we conclude \eqref{eq:converse-short-edge} by dominated convergence. 

We return to the proof of \eqref{eq:converse-short-halfline-limit}. Let $N=n+R$ in \eqref{eq:MPA again} and sum over \(\tau_{n+1},\dots,\tau_{n+R}\),
\[
\mu_{n+R}^{(0,\a,\b,\c,\d)}
\left((\tau_1,\dots,\tau_n)
=(\upepsilon_1,\dots,\upepsilon_n)\right)
=
\frac{ \langle W|\mathbf A\mathbf C^R|V\rangle}{\langle W|\mathbf C^{n+R}|V\rangle},
\]
where we denote $\mathbf A=\left((1-\upepsilon_1)\mathbf E+\upepsilon_1\mathbf D\right)\dots\left((1-\upepsilon_n)\mathbf E+\upepsilon_n\mathbf D\right)$.
Using \eqref{eq:converse-short-edge}, \eqref{eq:converse-short-edge-uniform}, and the dominated convergence theorem (in view of $|\upchi_y(\a,\c)|\leq(y+1)\max(\a,|\c|)^y$),  we obtain
\[
\begin{aligned}
\lim_{R\to\infty}\frac{R^{3/2}}{4^R}
\langle W|\mathbf A\mathbf C^R|V\rangle
 =
\upkappa_{\b,\d}\langle W|\mathbf A|H\rangle,\qquad
\lim_{R\to\infty}\frac{R^{3/2}}{4^R}
\langle W|\mathbf C^{n+R}|V\rangle
 =
\upkappa_{\b,\d}\langle W|\mathbf C^n|H\rangle,
\end{aligned}
\]
where $\upkappa_{\b,\d}= 4/\sqrt\pi\cdot (1-\b\d)(1-\b)^{-2}(1-\d)^{-2}$.
Combining with the facts that $\mathbf C|H\rangle=4|H\rangle$ and 
\begin{equation}\label{eq:sum upchi}
\langle W|H\rangle=\sum_{x\geq0}(x+1)\upchi_x(\a,\c) = 
\left.
\frac{d}{dz}
\left(
\frac{z}{(1-\a z)(1-\c z)}
\right)
\right|_{z=1} =
\frac{1-\a\c}{(1-\a)^2(1-\c)^2},
\end{equation}
we conclude the proof of \eqref{eq:converse-short-halfline-limit}.
\end{proof}

Next, we use the above notation to define a signed measure $\upnu^{\a,\c}_n $ on $\mathbb Z_{\geq 0}$, which we will allow us compare the measure $\mu_{\mathrm{HL}}^{\a,\c}$ for different choices of $\a$ and $\c$.

\begin{lemma}\label{lem:diff of nu}
 Assume $\a\in[0,1)$ and  $\c\in(-1,0]$. Let $\mathbf D$, $\mathbf E$, $\langle W|$ be as in \eqref{eq:DEWV TASEP rep again}, and $\mathbf C=\mathbf D+\mathbf E$. Define 
 \[ 
\upnu^{\a,\c}_n (x)
:=4^{-n}\frac{(1-\a)^2(1-\c)^2}{1-\a\c }(x+1)\left(\langle W |\mathbf C^n\right)_x \qquad \mbox{ for each }n,x\in\mathbb Z_{\geq0}.
\] 
Then for each $n\in\mathbb Z_{\geq0}$, $\upnu^{\a,\c}_n$ is a finite signed measure of total mass $1$. Moreover,
\[
\sup_{n\in\mathbb Z_{\geq0}}(n+1)\sum_{x=0}^{\infty}\left|\upnu_n^{\a,\c}(x)-\upnu_n^{0,0}(x)\right| <\infty.
\]
\end{lemma}
\begin{proof}
We define, for $n,x,y\in\mathbb Z_{\geq0}$,
\[
\upzeta_y=\frac{(1-\a)^2(1-\c)^2}{1-\a\c}(y+1)\upchi_y(\a,\c),\qquad  U_{n,y}(x)
=
4^{-n}\frac{x+1}{y+1}Z_n(y,x).
\]
We claim that $\upzeta_y$ is a finite signed measure of total mass $1$, and that $U_{n,y}(x)$ is a probability measure. Indeed, since  $|\upchi_y(\a,\c)|\leq(y+1)\max(\a,|\c|)^y$ we have $\sum_{y\geq0}|\upzeta_y|<\infty$. The fact that
$\sum_{y\geq0}\upzeta_y=1$ follows from \eqref{eq:sum upchi}.
The fact that $\sum_{x\geq0}U_{n,y}(x)=1$ follows from $\sum_{x\geq0}(x+1)Z_n(y,x)=(\mathbf C^n|H\rangle)_y=4^n(y+1)$ where $|H\rangle=[1,2,3,\ldots]^T$, which is due to $(\mathbf C^n)_{y,x}=Z_n(y,x)$ and $\mathbf C|H\rangle=4|H\rangle$. By definition,  for $n,x\in \mathbb Z_{\geq0}$,
    \begin{equation}\label{eq:decomp of upnu}
        \upnu_n^{\a,\c}(x)
=
\sum_{y\geq0}\upzeta_yU_{n,y}(x),
    \end{equation}
which converges absolutely in total variation because
$\sum_{y\geq0}|\upzeta_y|
\sum_{x\geq0}U_{n,y}(x)
=
\sum_{y\geq0}|\upzeta_y|<\infty$. Hence we know that $\upnu^{\a,\c}_n$ is a finite signed measure of total mass  $\sum_{y\geq0}\upzeta_y=1$. We have
\begin{equation}\label{eq:diff-nu-mixture-bound}
\sum_{x\geq0}
\left|\upnu_n^{\a,\c}(x)-\upnu_n^{0,0}(x)\right|
\leq
\sum_{y\geq0}|\upzeta_y|
\sum_{x\geq0}|U_{n,y}(x)-U_{n,0}(x)|.
\end{equation}
We claim that there exists an absolute constant $C_1\in(0,\infty)$ such that for each $n,y\in\mathbb Z_{\geq0}$,
\begin{equation}\label{eq:diff-nu-small-height}
\sum_{x\geq0}|U_{n,y}(x)-U_{n,0}(x)|
\leq
C_1\frac{(y+1)^2}{n+1}
\qquad\text{ if }(y+2)^2\leq n+1.
\end{equation}
Indeed, if $(y+2)^2\leq n+1$, by Lemma~\ref{lem:comparison kernels} with $(N,x,y)=(n,y,x)$,
\begin{equation*}
    \begin{split}
        \sum_{x\geq0}|U_{n,y}(x)-U_{n,0}(x)|
&\leq
\frac{64(y+1)^2}{(n+1)^{5/2}}
\sum_{x\geq0}(x+1)(x+y+2)
e^{- x^2/(12(n+1)) } .
    \end{split}
\end{equation*}
When  $(y+2)^2\leq n+1$, we also have $(x+1)(x+y+2) \leq2(x^2+n+1)$. Moreover, 
\[
\begin{aligned}
\sum_{x\geq0}(x^2+n+1)e^{- x^2/(12(n+1)) }
&\leq
n+1+\int_0^\infty
\left((t+1)^2+n+1\right)e^{- t^2/(12(n+1)) }\,dt \leq C_2 (n+1)^{3/2},
\end{aligned}
\]
for an absolute constant $C_2\in(0,\infty)$. We conclude the proof of \eqref{eq:diff-nu-small-height}. 
For the range $(y+2)^2 > n+1$, we use the bound $\sum_{x\geq0}|U_{n,y}(x)-U_{n,0}(x)|\leq 2$. Combining with \eqref{eq:diff-nu-small-height}, in view of \eqref{eq:diff-nu-mixture-bound},
\[
\begin{aligned}
\sum_{x\geq0}
\left|\upnu_n^{\a,\c}(x)-\upnu_n^{0,0}(x)\right|
&\leq
\frac{C_1}{n+1}
\sum_{(y+2)^2\leq n+1}|\upzeta_y|(y+1)^2
+
2\sum_{(y+2)^2>n+1}|\upzeta_y| \leq
\frac{C_1+2}{n+1}
\sum_{y\geq0}|\upzeta_y|(y+2)^2.
\end{aligned}
\]
Note that the sum on the right side is finite since $|\upchi_y(\a,\c)|\leq(y+1)\max(\a,|\c|)^y$. We conclude the proof.
\end{proof}

Next, we require an effective bound on the expected density under the measure $\mu_{\mathrm{HL}}^{\a,\c}$ for different choices of $\a$ and $\c$. 

\begin{lemma}\label{lem:asym of dn}
      Assume $\a \in[0,1)$ and $\c \in(-1,0]$. Consider the probability measure $\mu_{\mathrm{HL}}^{\a,\c}$ on $\{0,1\}^{\mathbb N}$ as in Lemma \ref{lem:limit to half space}, and denote its expectation by $\mathbb E_{\mathrm{HL}}^{\a,\c}$. Then 
      \[
      \lim_{n\to\infty}n^{3/2}\left(\mathbb E_{\mathrm{HL}}^{\a,\c}[\tau_n]-\mathbb E_{\mathrm{HL}}^{0,0}[\tau_n]\right)=-\frac{\eta(\a,\c)}{\sqrt\pi},
      \]
      where $\eta(u,v)=u(1-u)^{-2}+v(1-v)^{-2}$ for each $u,v\neq1$.
\end{lemma}
\begin{proof} 
Let $\mathbf D$ and $\mathbf E$ be as in \eqref{eq:DEWV TASEP rep again}, and $\mathbf C=\mathbf D+\mathbf E$. Let $\langle Q|=[Q_0,Q_1,\dots]$, where 
\[
Q_j=\upchi_j(\a,\c)\frac{(1-\a)^2(1-\c)^2}{1-\a\c }-{\bf1}_{\{j=0\}}\qquad\mbox{ for } j\in\mathbb Z_{\geq0}.
\]
Let $|\mathbf 1\rangle=[1,1,1,\ldots]^T$ and $|H\rangle=[1,2,3,\dots]^T$. We have the following simple properties:
\[
(\mathbf C^n)_{x,y}=Z_n(x,y),\qquad
\mathbf D|H\rangle=2|H\rangle+|\mathbf 1\rangle,
\qquad
\mathbf C|H\rangle=4|H\rangle,\qquad
\langle Q|H\rangle=0 . 
\]
Here, the last equality follows from \eqref{eq:sum upchi}. Using Lemma~\ref{lem:limit to half space} at $n+1$ and summing over $\tau_1,\dots,\tau_n$,
\begin{equation}\label{eq:expan taun}
    \mathbb E_{\mathrm{HL}}^{\a,\c}[\tau_{n+1}]
-\mathbb E_{\mathrm{HL}}^{0,0}[\tau_{n+1}]
=4^{-n-1}\langle Q|\mathbf C^n\mathbf D|H\rangle 
=4^{-n-1}\langle Q|\mathbf C^n|\mathbf 1\rangle 
=4^{-n-1}\sum_{x=0}^{\infty}Q_xZ_n(x,\mathrm{free}), 
\end{equation}
where we recall $Z_n(x,\mathrm{free})$ from Definition~\ref{def:bicolored Motzkin path}.
We  claim that 
\begin{equation}\label{eq:converse-short-moments}
\sum_{x\geq0}Q_x(x+1)=0,
\qquad
\sum_{x\geq0}Q_x(x+1)^3=6\eta(\a,\c).
\end{equation}
Indeed, the first identity is equivalent to $\langle Q|H\rangle=0$. Let 
\[
F(z)=\sum_{x\geq0}\upchi_x(\a,\c)z^x=((1-\a z)(1-\c z))^{-1},
\]
which is absolutely convergent in a neighborhood of $|z|\leq 1$. The second identity in \eqref{eq:converse-short-moments} then follows from 
\[
\frac{(1-\a)^2(1-\c)^2}{1-\a\c}
\left.
\left(z\frac{d}{dz}+1\right)^3F(z)
\right|_{z=1}
=
1+6\left(
\frac{\a}{(1-\a)^2}
+\frac{\c}{(1-\c)^2}
\right).
\] 
Recall from Lemma~\ref{lem:bounds ZNxfree} that  for each fixed $L\in(0,\infty)$, uniformly for $n\geq3$ and $x\leq L\log n$,
\begin{equation}\label{eq:converse-short-free-expansion RECALL}
\frac{Z_n(x,\mathrm{free})}{4^n}
=\frac1{\sqrt{\pi n}}
\left(2(x+1)-\frac1n
\left(\frac23(x+1)^3+\frac7{12}(x+1)\right)
+O_L\!\left(\frac{(x+1)^5}{n^2}\right)\right).
\end{equation} 
We estimate the right side of \eqref{eq:expan taun}.  
In view of $|\upchi_x(\a,\c) |\leq(x+1)\max (\a,-\c)^x$, we can choose a large constant $L$ such that
$\sum_{x>L\log n}(x+1)^3|Q_x|=o(n^{-2})$.
Since $Z_n(x,\mathrm{free})\leq4^n$, the contribution to \eqref{eq:expan taun} from $x>L\log n$ is $o(n^{-2})$.
We estimate the contribution from $x\leq L\log n$ using \eqref{eq:converse-short-free-expansion RECALL}.
Note that, for $k=1$ as well as $k=3$,
\[
\left|
\sum_{x\leq L\log n}Q_x(x+1)^k
-\sum_{x\geq0}Q_x(x+1)^k
\right|
\leq
\sum_{x>L\log n}|Q_x|(x+1)^3
=o(n^{-2}).
\]
Also, the  
contribution of the remainder in
\eqref{eq:converse-short-free-expansion} is at most
\[
\frac{C_L}{n^{5/2}}
\sum_{x\leq L\log n}|Q_x|(x+1)^5
\leq
\frac{C_L}{n^{5/2}}
\sum_{x\geq0}|Q_x|(x+1)^5
=
O_{\a,\c,L}(n^{-5/2}),
\]
where  $\sum_{x\geq0}|Q_x|(x+1)^5<\infty$ follows from
$|\upchi_x(\a,\c)|\leq(x+1)\max(\a,-\c)^x$. We conclude that
\[
\begin{aligned}
&\mathbb E_{\mathrm{HL}}^{\a,\c}[\tau_{n+1}]
-\mathbb E_{\mathrm{HL}}^{0,0}[\tau_{n+1}]\\
& =
\frac{1}{4\sqrt{\pi n}}
\left(
2\sum_{x\geq0}Q_x(x+1)
-\frac1n\left(
\frac23\sum_{x\geq0}Q_x(x+1)^3
+\frac7{12}\sum_{x\geq0}Q_x(x+1)
\right)
\right)
+o(n^{-3/2}).
\end{aligned}
\]
In view of \eqref{eq:converse-short-moments}, we conclude the proof of the lemma.
\end{proof}

Next, we express the law of $\mu_{\mathrm{HL}}^{\a,\c}-\mu_{\mathrm{HL}}^{0,0}$ as a suitable mixture of probability measures, which are explicitly determined by the matrix product ansatz.  

\begin{lemma}\label{lem:mathsf P} 
    Let $\mathbf D$ and $\mathbf E$ be as in \eqref{eq:DEWV TASEP rep again}, $\mathbf C=\mathbf D+\mathbf E$, and    $|H\rangle=[1,2,3,\ldots]^T$. 
For \(L\in\mathbb N\), \(x\in\mathbb Z_{\geq0}\),  define
\begin{equation*} 
\mathsf P_{L,x}(\boldsymbol{\upepsilon})
=
\frac{4^{-L}}{x+1}
\left(
\left((1-\upepsilon_1)\mathbf E+\upepsilon_1\mathbf D\right)\cdots
\left((1-\upepsilon_L)\mathbf E+\upepsilon_L\mathbf D\right)|H\rangle
\right)_x
\in\{0,1\}^L 
\end{equation*} for all $\boldsymbol{\upepsilon}=(\upepsilon_1,\ldots,\upepsilon_L)$. 
Then $\mathsf P_{L,x}$ is a probability measure on $\{0,1\}^L$ for every choice of $x$. Furthermore, for each \(\a\in[0,1)\), \(\c\in(-1,0]\) and $n,L\in\mathbb N$,
\begin{equation}\label{eq:decomp mu HL}
\mu_{\mathrm{HL}}^{\a,\c}
 \left((\tau_i)_{i\in\llbracket n,n+L-1\rrbracket}
 \in\,\cdot\,\right)
-
\mu_{\mathrm{HL}}^{0,0}
 \left((\tau_i)_{i\in\llbracket n,n+L-1\rrbracket}
 \in\,\cdot\,\right) 
=
\sum_{x\geq0}\left(\upnu_{n-1}^{\a,\c}(x)-\upnu_{n-1}^{0,0}(x)\right)\mathsf P_{L,x},
\end{equation}
where $\upnu$ is as in Lemma \ref{lem:diff of nu}. 
Moreover, we have that
\begin{equation}\label{eq:second momemt bound} 
\int
\left(\upepsilon_1+\dots+\upepsilon_L-L/2 
\right)^2
\,d\mathsf P_{L,x}(\boldsymbol\upepsilon)\leq L
\end{equation}
for each $L\in\mathbb N$, $x\in\mathbb Z_{\geq0}$.
\end{lemma}
\begin{proof}
  Note that  the entries of $\mathsf P_{L,x}$ are nonnegative. Using \(\mathbf C|H\rangle=4|H\rangle\), we have
\[
\begin{aligned}
\sum_{\boldsymbol{\upepsilon}\in\{0,1\}^L}
\mathsf P_{L,x}(\boldsymbol{\upepsilon})
&=
\frac{4^{-L}}{x+1}
\sum_{\boldsymbol{\upepsilon}\in\{0,1\}^L}
\left(\left((1-\upepsilon_1)\mathbf E+\upepsilon_1\mathbf D\right)\cdots
\left((1-\upepsilon_L)\mathbf E+\upepsilon_L\mathbf D\right)|H\rangle
\right)_x =
\frac{4^{-L}}{x+1}(\mathbf C^L|H\rangle)_x
=1.
\end{aligned}
\]
Hence $\mathsf P_{L,x}$ is a probability measure.
Summing the first displayed formula in Lemma~\ref{lem:limit to half space} over \(\upepsilon_1,\ldots,\upepsilon_{n-1}\),  
\[
\begin{aligned}
\mu_{\mathrm{HL}}^{\a,\c}
 \left((\tau_i)_{i\in\llbracket n,n+L-1\rrbracket}=\boldsymbol{\upepsilon}\right)
 &=
\frac{(1-\a)^2(1-\c)^2}{ 4^{(n-1)+L}(1-\a\c)}
\left\langle W\right|
\mathbf C^{n-1}
\left((1-\upepsilon_1)\mathbf E+\upepsilon_1\mathbf D\right)\cdots
\left((1-\upepsilon_L)\mathbf E+\upepsilon_L\mathbf D\right)
\left|H\right\rangle\\ &
 =
\sum_{x\geq0}
\upnu_{n-1}^{\a,\c}(x)
\mathsf P_{L,x}(\boldsymbol{\upepsilon}),\qquad\qquad\mbox{ for each }\boldsymbol{\upepsilon}\in\{0,1\}^L.
\end{aligned}
\]
Subtracting the corresponding identity at \((\a,\c)=(0,0)\) concludes the proof of \eqref{eq:decomp mu HL}.
We mention that the absolute convergence of these series follows from $\sum_{x\geq0}
 ( |\upnu_{n-1}^{\a,\c}(x) |
+ |\upnu_{n-1}^{0,0}(x) | )<\infty$ in view of Lemma~\ref{lem:diff of nu}.

We next prove \eqref{eq:second momemt bound}. We define a probability measure $\mathtt P_{L,x}$ on \(\mathrm{BCMot}_L^{x,\mathrm{free}}\). Let 
\[
\mathtt P_{L,x}(\bg)
=
4^{-L}\frac{\gamma_L+1}{x+1},\qquad\bg\in\mathrm{BCMot}_L^{x,\mathrm{free}}.
\]
This is a probability measure since $\sum_{\bg\in
\mathrm{BCMot}_L^{x,\mathrm{free}}}
(\gamma_L+1)
=
(\mathbf C^L|H\rangle)_x
=
4^L(x+1)$.  We claim that $\Phi_{\#}\mathtt P_{L,x}=\mathsf P_{L,x}$. Indeed,
for each $\boldsymbol{\upepsilon}\in\{0,1\}^L$, 
\[
\left(
\left((1-\upepsilon_1)\mathbf E+\upepsilon_1\mathbf D\right)\cdots
\left((1-\upepsilon_L)\mathbf E+\upepsilon_L\mathbf D\right)|H\rangle
\right)_x
=
\sum_{\bg\in\mathrm{BCMot}_L^{x,\mathrm{free}},\,
\Phi(\bg)=\boldsymbol{\upepsilon}}
(\gamma_L+1).
\] 
We next claim that  $(\gamma_j)_{j=0}^L$ is a Markov chain, and that for each $y\in\mathbb Z_{\geq0}$ and $j\in\llbracket1,L\rrbracket$ with $|y-x|\leq j-1$,
\begin{equation}\label{eq:transition mathttP}
    \begin{split}
        \mathtt P_{L,x}\left(\gamma_j-\gamma_{j-1}=1\,|\,\gamma_{j-1}=y\right)&=\frac{y+2}{4(y+1)},\\
        \mathtt P_{L,x}\left(\gamma_j-\gamma_{j-1}=0\mbox{ is colored black}\,|\,\gamma_{j-1}=y\right)&=\frac{1}{4},\\
        \mathtt P_{L,x}\left(\gamma_j-\gamma_{j-1}=0\mbox{ is colored white}\,|\,\gamma_{j-1}=y\right)&=\frac{1}{4},\\
        \mathtt P_{L,x}\left(\gamma_j-\gamma_{j-1}=-1\,|\,\gamma_{j-1}=y\right)&=\frac{y}{4(y+1)}, 
    \end{split}
\end{equation} i.e. $(\gamma_j)_{j=0}^L$ is a lazy simple random walk conditioned to stay non-negative.
Indeed, for each fixed admissible path prefix ending at $\gamma_{j-1}=y$, and a fixed $j$th step (including its color when horizontal) ending at $\gamma_j=z$, the ratio of the total $\mathtt P_{L,x}$-weight of the continuations of the extended prefix to the total weight of the continuations of the original prefix is
$ (\mathbf C^{L-j}|H\rangle)_z/(\mathbf C^{L-j+1}|H\rangle)_y
=
(z+1)/(4(y+1))$.

Under $\mathtt P_{L,x}$, write
$(\tau_1,\ldots,\tau_L)=\Phi(\bg)$. Consider the natural filtration given by $\mathcal F_0=\sigma(\gamma_0)$ and $\mathcal F_j
=
\sigma(\gamma_0,\ldots,\gamma_j,\tau_1,\ldots,\tau_j)$ for $j\in\llbracket1,L\rrbracket$. Define
$\upvartheta_j=\tau_j-1/2-(\gamma_j-\gamma_{j-1})/2$.
By \eqref{eq:transition mathttP},  
\[
\mathbb E[\upvartheta_j\mid\mathcal F_{j-1}]=0,
\qquad
\mathbb E[\upvartheta_j^2\mid\mathcal F_{j-1}]= \frac{1}{8}.
\]
Therefore,
\begin{equation}\label{eq:comparison-HL-zeta-moment}
\mathbb E\left[
\left(\upvartheta_1+\dots+\upvartheta_L\right)^2
\right]
=
 \mathbb E[\upvartheta_1^2]+\dots+\mathbb E[\upvartheta_L^2]
=  \frac{L}{8}.
\end{equation}
Using the transition probabilities \eqref{eq:transition mathttP}, we also have  
\[
\mathbb E[
\gamma_j-\gamma_{j-1}\mid\mathcal F_{j-1}]
=
\frac1{2(\gamma_{j-1}+1)}
\geq0,\qquad
\mathbb E[
\gamma_j^2-\gamma_{j-1}^2\mid\mathcal F_{j-1}]
=
\frac{3\gamma_{j-1}+1}{2(\gamma_{j-1}+1)}
\leq\frac32.
\] 
Summing these up, we obtain
$\mathbb E[\gamma_L]\geq x$ 
and
$\mathbb E[\gamma_L^2]\leq x^2+ 3L/2$.
Since \(x\geq0\), we have
\begin{equation}\label{eq:comparison-HL-height-moment} 
\mathbb E[(\gamma_L-x)^2]
 =
\mathbb E[\gamma_L^2]-2x\mathbb E[\gamma_L]+x^2 \leq3L/2. 
\end{equation}
Note that 
$\sum_{j=1}^L\left(\tau_j- 1/2\right)
= (\gamma_L-x)/2+\sum_{j=1}^L\upvartheta_j$. 
Combining this with \eqref{eq:comparison-HL-zeta-moment} and
\eqref{eq:comparison-HL-height-moment}, in view of  the fact that \((u+v)^2\leq2u^2+2v^2\), we conclude the proof of \eqref{eq:second momemt bound}.
\end{proof}

Combining the above insights, we have the following lower bound on the total variation distance between the measures $\mu_{\mathrm{HL}}^{\a,\c}$ for different choices of $\a$ and $\c$.

\begin{proposition}\label{prop:comparison HL measures}
Assume \(\a\in[0,1)\) and \(\c\in(-1,0]\) satisfy \(\eta(\a,\c)=\a(1-\a)^{-2}+\c(1-\c)^{-2}\neq0\). Then there exists a finite closed interval $I_{\a,\c}\subset(0,\infty)$ depending only on $\a$ and $\c$, such that 
\begin{equation}\label{eq:statement prop TV}
 \sup_{\ell\in\mathbb N:\,\ell\geq n}n\,
\TV{
\mu_{\mathrm{HL}}^{\a,\c}
 \left((\tau_i)_{i\in\llbracket n,\ell\rrbracket}\in\,\cdot\,\right)
-
\mu_{\mathrm{HL}}^{0,0}
 \left((\tau_i)_{i\in\llbracket n,\ell\rrbracket}\in\,\cdot\,\right)}\in I_{\a,\c}\qquad\mbox{ for all }n\in\mathbb N.
\end{equation}
Moreover, for every \(\b\in[0,1)\) and \(\d\in(-1,0]\),
\[
\inf_{n\in\mathbb N}
\sup_{\ell\in\mathbb N:\,\ell\geq n}
\liminf_{N\to\infty}
n\,
\TV{
\mu_N^{(0,\a,\b,\c,\d)}
 \left((\tau_i)_{i\in\llbracket n,\ell\rrbracket}\in\,\cdot\,\right)
-
\mathbb P_N^{\mathsf{exc}}
 \left((\tau_i)_{i\in\llbracket n,\ell\rrbracket}\in\,\cdot\,\right)}
>0.
\]
\end{proposition}

\begin{proof} We first prove the upper bound in \eqref{eq:statement prop TV}. Using Lemma \ref{lem:mathsf P},
\[
\begin{aligned}
\TV{
\mu_{\mathrm{HL}}^{\a,\c}
 \left((\tau_i)_{i\in\llbracket n,n+L-1\rrbracket}\in\,\cdot\,\right)
-
\mu_{\mathrm{HL}}^{0,0}
\left((\tau_i)_{i\in\llbracket n,n+L-1\rrbracket}\in\,\cdot\,\right)}
\leq
\frac12\sum_{x\geq0}\left|\upnu_{n-1}^{\a,\c}(x)-\upnu_{n-1}^{0,0}(x)\right| 
\leq\frac{C_{\upnu}}{2n},
\end{aligned}
\]
where
 $C_{\upnu} =
\sup_{r\in\mathbb Z_{\geq0}}(r+1)
\sum_{x\geq0} |\upnu_{r}^{\a,\c}(x)-\upnu_{r}^{0,0}(x) |<\infty$ by Lemma~\ref{lem:diff of nu}. 
Taking the supremum over \(L\in\mathbb N\),  
\begin{equation}\label{eq:comparison-HL-upper}
 \sup_{\ell\in\mathbb N:\,\ell\geq n}n\,
\TV{
\mu_{\mathrm{HL}}^{\a,\c}
 \left((\tau_i)_{i\in\llbracket n,\ell\rrbracket}\in\,\cdot\,\right)
-
\mu_{\mathrm{HL}}^{0,0}
 \left((\tau_i)_{i\in\llbracket n,\ell\rrbracket}\in\,\cdot\,\right)}
\leq C_{\upnu}/2,
\qquad n\in\mathbb N.
\end{equation} 

We next prove the lower bound in \eqref{eq:statement prop TV}. Denote $d_j=\mathbb E_{\mathrm{HL}}^{\a,\c}[\tau_j]-\mathbb E_{\mathrm{HL}}^{0,0}[\tau_j]$ for $j\in\mathbb N$.
By Lemma~\ref{lem:asym of dn}, we have $\lim_{j\to\infty}j^{3/2}d_j=- \eta(\a,\c)/\sqrt\pi$, and hence
\begin{equation*} 
\sqrt n\sum_{j=n}^{2n-1}d_j
 =
\frac1n\sum_{j=n}^{2n-1}
\left(\frac jn\right)^{-3/2}j^{3/2}d_j 
 \longrightarrow
-\frac{\eta(\a,\c)}{\sqrt\pi}
\int_1^2u^{-3/2}\,du 
 =
-\frac{2\eta(\a,\c)}{\sqrt\pi}
\left(1-\frac1{\sqrt2}\right)\qquad\mbox{ as }n\to\infty.
\end{equation*}
 Since \(\eta(\a,\c)\neq0\), we know that there exist $C_1\in(0,\infty)$ and $N_1\in\mathbb N$, depending only on \(\a,\c\), such that
 \begin{equation}\label{eq:comparison-HL-block-mean-lower}
\left|
d_n+\dots+d_{2n-1}
\right|
\geq C_1/\sqrt n 
\qquad \mbox{ for each } n\geq N_1.
\end{equation}

Denote $\Delta_{n,L}= \mu_{\mathrm{HL}}^{\a,\c}
 \left((\tau_i)_{i\in\llbracket n,n+L-1\rrbracket}
 \in\,\cdot\,\right)
-
\mu_{\mathrm{HL}}^{0,0}
 \left((\tau_i)_{i\in\llbracket n,n+L-1\rrbracket}
 \in\,\cdot\,\right) $, which is a signed measure on \(\{0,1\}^L\). Denote by \(|\Delta_{n,L}|\)  its variation measure.
By Lemma \ref{lem:mathsf P} twice, we see that
\[
\left|\Delta_{n,L}(\boldsymbol\upepsilon)\right|
\leq
\sum_{x\geq0}
\left|\upnu_{n-1}^{\a,\c}(x)-\upnu_{n-1}^{0,0}(x)\right|
\mathsf P_{L,x}(\boldsymbol\upepsilon),\qquad\boldsymbol\upepsilon\in\{0,1\}^L.
\]
as well as 
\begin{equation}\label{eq:comparison-HL-signed-moment} 
\begin{split}
\int \left(\tau_n+\dots+\tau_{n+L-1}-L/2\right)^2\,d|\Delta_{n,L}| &\leq
\sum_{x\geq0}
\left|\upnu_{n-1}^{\a,\c}(x)-\upnu_{n-1}^{0,0}(x)\right|
\int\left(-\frac{L}{2}+\sum_{i=0}^{L-1}\tau_{n+i}\right)^2\,d\mathsf P_{L,x}  \\
& \leq
L\sum_{x\geq0}\left|\upnu_{n-1}^{\a,\c}(x)-\upnu_{n-1}^{0,0}(x)\right| 
 \leq
C_{\upnu} L/n. 
\end{split}
\end{equation} 
Let  
$\lambda
=
\min (1,C_1/(C_{\upnu}+1) )\in(0,1]$. 
Then \(C_{\upnu}\lambda\leq C_1\). Denote $X_n=\tau_n+\dots+\tau_{2n-1}-n/2$. 
Note that 
\[
\int X_{n}\,d\Delta_{n,n}
=d_n+\dots+d_{2n-1}.
\] 
Using  
\eqref{eq:comparison-HL-block-mean-lower}  and
\eqref{eq:comparison-HL-signed-moment} with \(L=n\), in view of $|\sin u-u|\leq u^2/2$ for $u\in\mathbb R$, we obtain
\[
\begin{aligned}
\left|
\int
\sin\left(\frac{\lambda X_{n }}{\sqrt n}\right)
\,d\Delta_{n,n}
\right|
&\geq
\frac{\lambda}{\sqrt n}
\left|
\int X_{n }\,d\Delta_{n,n}
\right|
-
\frac{\lambda^2}{2n}
\int X_{n }^2\,d|\Delta_{n,n}| \geq
\frac{\lambda C_1}{n}
-
\frac{C_{\upnu}\lambda^2}{2n} \geq
\frac{\lambda C_1}{2n},
\qquad n\geq N_1.
\end{aligned}
\]
It is a standard fact that for every function \(|f|\leq1\) and probability measures \(P,Q\) on a finite state space, we have
$\left|\int f\,d(P-Q)\right|
\leq 
2\TV{P-Q}$. Taking $f$ to be the sine function, in view of the above display, for $n\geq N_1$,
\[
\sup_{\ell\in\mathbb N:\,\ell\geq n} 
\TV{
\mu_{\mathrm{HL}}^{\a,\c}
 \left((\tau_i)_{i\in\llbracket n,\ell\rrbracket}\in\,\cdot\,\right)
-
\mu_{\mathrm{HL}}^{0,0}
 \left((\tau_i)_{i\in\llbracket n,\ell\rrbracket}\in\,\cdot\,\right)}=\sup_{L\in\mathbb N}\TV{\Delta_{n,L}}
\geq
\TV{\Delta_{n,n}}
\geq \lambda C_1/(4n).
\]
By Lemma~\ref{lem:asym of dn} and \(\eta(\a,\c)\neq0\), there exists $N_2\geq N_1$ such that \(d_{N_2}\neq0\). Hence, for $n < N_1$,
 \[
\begin{aligned}
\sup_{\ell\in\mathbb N:\,\ell\geq n} 
\TV{
\mu_{\mathrm{HL}}^{\a,\c}
 \left((\tau_i)_{i\in\llbracket n,\ell\rrbracket}\in\,\cdot\,\right)
-
\mu_{\mathrm{HL}}^{0,0}
 \left((\tau_i)_{i\in\llbracket n,\ell\rrbracket}\in\,\cdot\,\right)}
 \geq
\TV{
\mu_{\mathrm{HL}}^{\a,\c}(\tau_{N_2}\in\,\cdot\,)
-
\mu_{\mathrm{HL}}^{0,0}(\tau_{N_2}\in\,\cdot\,)} =
|d_{N_2}|.
\end{aligned}
\]
Combining the last two displayed equations and \eqref{eq:comparison-HL-upper}, we conclude the first assertion in the proposition.

It remains to prove the last assertion. By Lemma~\ref{lem:limit to half space}, for each \(\b\in[0,1)\), \(\d\in(-1,0]\), and  \(1\leq n\leq\ell\),
\[
\lim_{N\to\infty}\TV{\mu_N^{(0,\a,\b,\c,\d)}
 \left((\tau_i)_{i\in\llbracket n,\ell\rrbracket}\in\,\cdot\,\right)
-
\mu_{\mathrm{HL}}^{\a,\c}
 \left((\tau_i)_{i\in\llbracket n,\ell\rrbracket}\in\,\cdot\,\right)}=0.
\]
Recall from Proposition~\ref{prop:SpecialExcursion} that $\mathbb P_N^{\mathsf{exc}}=\mu_N^{(0,0,0,0,0)}$. Therefore,
\[
\lim_{N\to\infty}\TV{\mathbb P_N^{\mathsf{exc}}
 \left((\tau_i)_{i\in\llbracket n,\ell\rrbracket}\in\,\cdot\,\right)
-
\mu_{\mathrm{HL}}^{0,0}
 \left((\tau_i)_{i\in\llbracket n,\ell\rrbracket}\in\,\cdot\,\right)}=0.
\]
Combining the last two displayed equations, in view of the triangle inequality,
\begin{equation*}
    \begin{split}
&\lim_{N\to\infty}\TV{
\mu_N^{(0,\a,\b,\c,\d)}
 \left((\tau_i)_{i\in\llbracket n,\ell\rrbracket}\in\,\cdot\,\right)
-
\mathbb P_N^{\mathsf{exc}}
 \left((\tau_i)_{i\in\llbracket n,\ell\rrbracket}\in\,\cdot\,\right)}\\
 &=
 \TV{
\mu_{\mathrm{HL}}^{\a,\c}
\left((\tau_i)_{i\in\llbracket n,\ell\rrbracket}\in\,\cdot\,\right)
-
\mu_{\mathrm{HL}}^{0,0}
\left((\tau_i)_{i\in\llbracket n,\ell\rrbracket}\in\,\cdot\,\right)}.
    \end{split}
\end{equation*} 
Therefore,
\begin{equation*}
    \begin{split}
        &\inf_{n\in\mathbb N}
\sup_{\ell\in\mathbb N:\,\ell\geq n}
\liminf_{N\to\infty}
n\,
\TV{
\mu_N^{(0,\a,\b,\c,\d)}
 \left((\tau_i)_{i\in\llbracket n,\ell\rrbracket}\in\,\cdot\,\right)
-
\mathbb P_N^{\mathsf{exc}}
 \left((\tau_i)_{i\in\llbracket n,\ell\rrbracket}\in\,\cdot\,\right)}\\
 &=\inf_{n\in\mathbb N}\sup_{\ell\in\mathbb N:\,\ell\geq n}n\,
\TV{
\mu_{\mathrm{HL}}^{\a,\c}
 \left((\tau_i)_{i\in\llbracket n,\ell\rrbracket}\in\,\cdot\,\right)
-
\mu_{\mathrm{HL}}^{0,0}
 \left((\tau_i)_{i\in\llbracket n,\ell\rrbracket}\in\,\cdot\,\right)}>0,
    \end{split}
\end{equation*} 
in view of the first assersion. We conclude the proof.
\end{proof}

\begin{proof}[Proof of Theorem \ref{thm:converse TASEP again}]
For each $\a,\b\in[0,1)$, $\c,\d\in(-1,0]$,  and \(N\geq \ell\geq n\), let
\[
D_N^{\a,\b,\c,\d}(n,\ell) =\TV{
\mu_N^{(0,\a,\b,\c,\d)}
\left((\tau_i)_{i\in\llbracket n,\ell\rrbracket}\in\cdot\right)
-
\mathbb P_N^{\mathsf{exc}}
\left((\tau_i)_{i\in\llbracket n,\ell\rrbracket}\in\cdot\right)}.
\]
In view of the last assertion of Proposition~\ref{prop:comparison HL measures}, if $\eta(\a,\c)\neq0$, we have
\[
C_{\a,\b,\c,\d}=
\inf_{n\in\mathbb N}
\sup_{\ell\in\mathbb N:\,\ell\geq n}
\liminf_{N\to\infty} nD_N^{\a,\b,\c,\d}(n,\ell) >0.
\]
Since $\eta(\a,\c)\eta(\b,\d)\neq0$, using the above for both $(\a,\b,\c,\d)$ and $(\b,\a,\d,\c)$, we know that for each fixed $m,k\in\mathbb N$,  there exist integers
\(M_{\mathrm L}\geq\ell_{\mathrm L}\geq m\) and
\(M_{\mathrm R}\geq\ell_{\mathrm R}\geq k\) such that
\[
D_N^{\a,\b,\c,\d}(m,\ell_{\mathrm L})>\frac{C_{\a,\b,\c,\d}}{2m}
\quad\text{for every }N\geq M_{\mathrm L},
\qquad
D_N^{\b,\a,\d,\c}(k,\ell_{\mathrm R})>\frac{C_{\b,\a,\d,\c}}{2k}
\quad\text{for every }N\geq M_{\mathrm R}.
\]
Using the particle-hole duality as in Lemma \ref{lem:particle hole}, when \(N\geq m+k-1\), we have 
\[
D_N^{\a,\b,\c,\d}(m,N-k+1)=D_N^{\b,\a,\d,\c}(k,N-m+1).
\]
Let $N_0=
\max (
m+k-1,\,
M_{\mathrm L},\,
M_{\mathrm R},\,
\ell_{\mathrm L}+k-1,\,
m+\ell_{\mathrm R}-1 )$. Then for \(N\geq N_0\), in view of  the last two displayed equations, $\llbracket m,\ell_{\mathrm L}\rrbracket
\subseteq
\llbracket m,N-k+1\rrbracket$, and $\llbracket k,\ell_{\mathrm R}\rrbracket
\subseteq
\llbracket k,N-m+1\rrbracket$,  
\[
\begin{aligned}
D_N^{\a,\b,\c,\d}(m,N-k+1)
&\geq
\max\left(
D_N^{\a,\b,\c,\d}(m,\ell_{\mathrm L}),
 D_N^{\b,\a,\d,\c}(k,\ell_{\mathrm R})
\right) >  
\frac{\min\left(C_{\a,\b,\c,\d},C_{\b,\a,\d,\c}\right)}
{2\min(m,k)}.
\end{aligned}
\]
We conclude the proof.
\end{proof}

\section{Moderate deviation bounds for the open ASEP in the maximal current phase: Proof of Theorem \ref{thm:moderate deviation MC}}\label{sec:Moderate}

In the following, our goal is to show a moderate deviation bound on the number of particles under the stationary measure of the open ASEP in the maximal current phase. This is the content of Theorem \ref{thm:moderate deviation MC}, which we state again below for clarity.

\begin{theorem}\label{thm:moderate deviation MC again}
Assume $q\in[0,1)$, $\a,\b\in[0,1]$, $\a\b<1$, and $\c,\d\in(-1,0]$. 
For \(N\in\mathbb N\), let \(\mu_N^{(q,\a,\b,\c,\d)}\) be the stationary measure of the open ASEP on the lattice $\llbracket1,N\rrbracket$, and let \(\tau_i\in\{0,1\}\) denote the occupation variable at
site \(i\). There exists a constant \(C\in(0,\infty)\), depending only on \(q,\a,\b,\c,\d\), such that, for every \(N\in\mathbb N\), every nonempty set
\(I\subseteq\llbracket1,N\rrbracket\), and every \(x\geq0\),
\[ 
 \mu_N^{(q,\a,\b,\c,\d)}
 \left(
 \left|\sum_{i\in I}\tau_i-\frac{|I|}{2}\right|
 \geq x\sqrt{|I|}
 \right)
 \leq C e^{-x^2/2}.
\]
\end{theorem}

Let us remark that for $1 \ll x \ll \log(|I|)$, the above result  follows directly from Theorem~\ref{thm:main open ASEP}. However, in order to show the above moderate deviation result for general $x$ and general subsets $I$, we require some new insights. In particular, we will invoke a novel correlation result for a tilted path measure, which is of independent interest.

\subsection{A moderate deviation bound for lattice intervals}

As a first step towards Theorem~\ref{thm:moderate deviation MC again}, we provide a moderate deviation bound for a particle measure closely related to the signed measure for bicolored Motzkin paths from Theorem~\ref{thm:TASEP rep} for the open TASEP, and when the set $I$ is a lattice interval; see also \cite{BrycWang2024MotzkinII} for related work on fluctuations of Motzkin paths with varying endpoints. 

\begin{lemma}\label{lem:endpoint-moderate-deviation}
Let $\a\in[0,1)$, $\b\in[0,1]$.  For $N\in\mathbb N$, let $\mathsf Z_N^{\a,\b}
 =\sum_{\bg\in\mathrm{BCMot}_N}\a^{\gamma_0}\b^{\gamma_N}$. Then $\mathsf Z_N^{\a,\b}\in(0,\infty)$, and 
 \[
  \mathsf Q_N^{\a,\b}(\bg)
 = \a^{\gamma_0}\b^{\gamma_N}/\mathsf Z_N^{\a,\b},
 \qquad
 \bg=(\gamma_0,\ldots,\gamma_N)\in\mathrm{BCMot}_N,
 \]
 defines a probability measure.  If $\bg\sim\mathsf Q_N^{\a,\b}$ and $(\tau_1,\ldots,\tau_N)=\Phi(\bg)$, then there exist constants $C,c\in(0,\infty)$, depending only on $\a$ and $\b$, such that, for every nonempty lattice interval $J\subseteq\llbracket1,N\rrbracket$ and every $x\geq0$,
\begin{equation}\label{eq:endpoint-Motzkin-moderate-deviation}
  \mathsf Q_N^{\a,\b}
 \left(
 \left|\sum_{i\in J}\tau_i-\frac{|J|}{2}\right|
 \geq x\sqrt{|J|}
 \right)
 \leq Ce^{-cx^2}.
\end{equation}
\end{lemma}

\begin{proof}
    For clarity, we split the proof into three steps. In the first step, we recall basic combinatorial identities from Section~\ref{sec:Combinatorial} and apply them to bound $\mathsf Z_N^{\a,\b}$. In the second step, obtain moderate deviations bounds on paths under the measure $\mathsf Q_N^{\a,\b}$. In the third step, we convert these estimates to moderate deviation bounds for the respective particle configuration.\\

\noindent\textbf{Step 1. } In this step, we develop some estimates for the involved partition functions, which in particular imply the first assertion that $\mathsf Z_N^{\a,\b}\in(0,\infty)$. Throughout this step, we write $C,c\in(0,\infty)$ to denote absolute constants which may change from line to line.
We start with the following combinatorial estimates which are immediate from Section~\ref{sec:Combinatorial}. More precisely, for each \(m,n,u,v\in\mathbb Z_{\geq0}\), we have
\begin{align}
 Z_{m+n}(u,v)
 &=\sum_{z\geq0}Z_m(u,z)Z_n(z,v),
 &Z_n(u,v)&=Z_n(v,u),                                      \label{eq:endpoint-K-semigroup-short}\\
  4^n(u+1)&=\sum_{v\geq0}  Z_n(u,v)(v+1),                      \label{eq:endpoint-harmonic-short}\\
 0\leq Z_n(u,v)
 &\leq \binom{2n}{n+v-u} ,                                           \label{eq:endpoint-K-free-short}\\
 Z_n(u,v)
 &\leq
 C 4^n\frac{(u+1)(v+1)}{(n+1)^{3/2}}
 \exp\left(-\frac{(u-v)^2}{n+1}\right),                    \label{eq:endpoint-K-Dirichlet-short}\\
 Z_n(u;\mathrm{free})
 &\leq
 4^n\min\left(1, 2(u+1)(n+1)^{-1/2}\right),            \label{eq:endpoint-H-short}\\
 Z_n(0,0)
 &\geq c 4^n(n+1)^{-3/2},
 &Z_n(0;\mathrm{free})&\geq c 4^n(n+1)^{-1/2}.                              \label{eq:endpoint-lower-short}
\end{align}
Indeed, the path concatenation and reversal give \eqref{eq:endpoint-K-semigroup-short}. One can check that $\sum_{v\geq0}Z_1(u,v)(v+1)
 =4(u+1)$, and then \eqref{eq:endpoint-harmonic-short} follows from \eqref{eq:endpoint-K-semigroup-short} by induction. 
The remaining estimates follow from Lemmas~\ref{lem:count biMot}, \ref{lem:bounds ZNxy}, and
\ref{lem:bounds ZNxfree}, together with  \(Z_n(u;\mathrm{free})\leq4^n\). In addition, we require that for every \(n\in\mathbb N\),
\(u\in\mathbb Z_{\geq0}\), and \(x\geq0\),
\begin{align}
 \sum_{ v\geq0,\,|v-u|\geq x\sqrt n}Z_n(u,v)
 &\leq
 C4^n (n+1)^{-1/2}(u+1 ) e^{-cx^2},                         \label{eq:endpoint-K-tail-0-short}\\
 \sum_{ v\geq0,\, |v-u|\geq x\sqrt n}
 (v+1)Z_n(u,v)
 &\leq C 4^n(u+1)e^{-cx^2}.                                   \label{eq:endpoint-K-tail-1-short}
\end{align}
We  give now a short probabilistic justification of \eqref{eq:endpoint-K-tail-0-short} and \eqref{eq:endpoint-K-tail-1-short}.
Define $Y_n=\frac12\sum_{i=1}^{2n}\upxi_i$, where $\upxi_i$ are i.i.d. with $\mathbb P(\upxi_i=1)=\mathbb P(\upxi_i=-1)=1/2$. Then 
\[
 \mathbb P(Y_n=d)=4^{-n}\binom{2n}{n+d},\qquad
 \mathbb E\left[Y_n^2\right]=\frac n2,\qquad
 \mathbb E\left[e^{\lambda Y_n}\right]
 =\cosh(\lambda/2)^{2n}\leq e^{n\lambda^2/4}.
\]
By Chernoff's inequality and the Cauchy--Schwarz inequality, respectively, 
\begin{equation}\label{eq:endpoint-p-tail-short}
 \sum_{|d|\geq x\sqrt n}4^{-n}\binom{2n}{n+d}\leq2e^{-x^2},
 \qquad
 \sum_{|d|\geq x\sqrt n} 4^{-n}|d|\binom{2n}{n+d}
 \leq\sqrt n\,e^{-x^2/2}.
\end{equation}
If \(u+1\geq\sqrt{n+1}\), then \eqref{eq:endpoint-K-tail-0-short} and \eqref{eq:endpoint-K-tail-1-short}
follow from combining \eqref{eq:endpoint-K-free-short} and \eqref{eq:endpoint-p-tail-short}. 

We next assume \(u+1<\sqrt{n+1}\) and set $w=\sqrt{n+1}$. We claim that for each \(k\in\{1,2\}\), \(a\geq0\), and \(w\geq1\),
\begin{equation}\label{eq:endpoint-discrete-Gaussian-short}
 \sum_{d\in\mathbb Z,\,|d|\geq aw}
 (w+|d|)^k e^{-d^2/w^2}
 \leq
 Cw^{k+1}e^{-a^2/2}.
\end{equation}
Indeed, on the summation range, we have $e^{-d^2/w^2}
 \leq e^{-a^2/2}e^{-d^2/(2w^2)}$.
For each \(j\in\mathbb Z_{\geq0}\), note that the set
\(\{d\in\mathbb Z:jw\leq|d|<(j+1)w\}\) contains at most \(4w\)
integers. Hence
\[
 \sum_{d\in\mathbb Z}(w+|d|)^k e^{-d^2/(2w^2)}
 \leq
 4w^{k+1}\sum_{j=0}^{\infty}(j+2)^k e^{-j^2/2}
 \leq Cw^{k+1},
\]
 and multiplying both sides by \(e^{-a^2/2}\) proves \eqref{eq:endpoint-discrete-Gaussian-short}.
 Using \eqref{eq:endpoint-discrete-Gaussian-short} with  \(a=x\sqrt n/w\)  and  \(k\in\{1,2\}\), in view of \(a^2\geq x^2/2\), and combining with \eqref{eq:endpoint-K-Dirichlet-short}, we conclude the proof of \eqref{eq:endpoint-K-tail-0-short} and \eqref{eq:endpoint-K-tail-1-short}. 

We next estimate 
\[
\mathsf{Z}_N^{\a,\b}
 =\sum_{s,y\geq0}\a^s\b^yZ_N(s,y),\qquad
 A_j(u)=   \sum_{s\geq0}\a^sZ_j(s,u),
 \qquad
 B_j(v)=\sum_{y\geq0}\b^yZ_j(v,y).
\]
When $\b<1$, using \(Z_N(s,y)\leq 4^N\), we see that \(\mathsf{Z}_N^{\a,\b}<\infty\); when $\b=1$, we get that 
 $\mathsf{Z}_N^{\a,1}=\sum_{s\geq0}\a^sZ_N(s;\mathrm{free})
 \leq4^N\sum_{s\geq0}\a^s<\infty$.
The fact that $\mathsf{Z}_N^{\a,\b}>0$ is contained in our next claim, which provides an effective lower bound on $\mathsf{Z}_N^{\a,\b}$:
\begin{align}
 \mbox{when }\b<1:\quad&
  \mathsf{Z}_N^{\a,\b}\geq c4^{N}(N+1)^{-3/2},
 &B_j(v)&\leq
 C4^j(1-\b)^{-2}   (v+1)(j+1)^{-3/2},                            \label{eq:endpoint-pinned-short}\\
  \mbox{when }\b=1:\quad&
  \mathsf{Z}_N^{\a,1}\geq c4^{N}(N+1)^{-1/2},
 &B_j(v)&=Z_j(v;\mathrm{free}).                                          \label{eq:endpoint-free-short}
\end{align}
Indeed, the lower bounds follow by retaining respectively the \((s,y)=(0,0)\) term and the \(s=0\) terms, while the upper bound for \(B_j\) follows from \eqref{eq:endpoint-K-Dirichlet-short}.   Finally, using \eqref{eq:endpoint-harmonic-short} and \eqref{eq:endpoint-H-short}, we have
\begin{align}
 \sum_{u\geq0}(u+1)A_j(u)
 &=4^j\sum_{s\geq0}(s+1)\a^s= 4^j(1-\a)^{-2} ,                  \label{eq:endpoint-A1-short}\\
 \sum_{u\geq0}A_j(u)
 &\leq 2\cdot 4^j(1-\a)^{-2} (j+1)^{-1/2}.                          \label{eq:endpoint-A0-short}
\end{align}

\noindent\textbf{Step 2. } 
Throughout this step, we write $C,c\in(0,\infty)$ to denote constants depending only on $\a$ and $\b$, which may change from line to line.
In this step, we show that  
\begin{equation}\label{eq:endpoint-height-increment}
 \mathsf Q_N^{\a,\b}
 \left(
 |\gamma_{r+\ell}-\gamma_r|\geq x\sqrt\ell
 \right)
 \leq Ce^{-cx^2}
 \qquad\mbox{ for each }0\leq r<r+\ell\leq N\mbox{ and } x\geq0.
\end{equation}
To prove this, we first prove the special case when $r=0$:   
\begin{equation}\label{eq:endpoint-one-time-displacement}
 \mathsf Q_N^{\a,\b}
 \left(|\gamma_j-\gamma_0|\geq x\sqrt j\right)
 \leq  Ce^{-cx^2}
 \qquad\mbox{ for each }j\in\llbracket1,N\rrbracket\mbox{ and } x\geq0.
\end{equation} 
If \(N\geq 2j\), then 
\eqref{eq:endpoint-pinned-short}, or
\eqref{eq:endpoint-H-short} and \eqref{eq:endpoint-free-short} when
\(\b=1\), gives
 \(4^{j}B_{N-j}(z)/\mathsf{Z}_N^{\a,\b}\leq C(z+1)\)
 for every \(z\geq0\).
Therefore, by \eqref{eq:endpoint-K-tail-1-short} and \eqref{eq:endpoint-A1-short},
\begin{equation}\label{eq:ertyby}
    \mathsf Q_N^{\a,\b}
 \left(|\gamma_j-\gamma_0|\geq x\sqrt j\right) =
 \frac1{ \mathsf{Z}_N^{\a,\b}}
 \sum_{s,z\geq0}
 \a^sZ_j(s,z)B_{N-j}(z)
 {\bf1}_{\{|z-s|\geq x\sqrt j\}}
 \leq Ce^{-cx^2}.
\end{equation}
 Suppose now that \(N<2j\).  If \(\b=1\), then we have
\(B_{N-j}(z)=Z_{N-j}(z;\mathrm{free})\leq4^{N-j}\)
for every \(z\geq0\), and
\(4^{N}/\mathsf{Z}_N^{\a,1} \leq C\sqrt{j+1}\).  Using
\eqref{eq:endpoint-K-tail-0-short},
\[
 \mathsf Q_N^{\a,1}
 \left(|\gamma_j-\gamma_0|\geq x\sqrt j\right)
 \leq
 C\sqrt{j+1} \cdot 4^{-j}
 \sum_{s\geq0}\a^s
 \sum_{ z\geq0,\,|z-s|\geq x\sqrt j}Z_j(s,z)
 \leq Ce^{-cx^2}.
\]
It remains to consider the case when \(N<2j\) and \(\b<1\).
By \eqref{eq:endpoint-K-Dirichlet-short} and the lower bound in
\eqref{eq:endpoint-pinned-short}, we have
 $\mathsf Q_N^{\a,\b}(\gamma_0=s,\gamma_N=y)
 \leq C(s+1)(y+1)\a^s\b^y $. Therefore, for sufficiently small $c_0>0$,  
\[
 \mathbb E_{\mathsf Q_N^{\a,\b}}
 \left[e^{c_0|\gamma_N-\gamma_0|}\right]
 \leq
 C\sum_{s,y\geq0}(s+1)(y+1)
 (\a e^{c_0})^s(\b e^{c_0})^y
 \leq C\qquad\mbox{ uniformly in }N\in\mathbb N.
\]
By Markov's inequality, 
 \begin{equation}\label{eq:endpoint-endpoint-tail-short}
 \mathsf Q_N^{\a,\b}(|\gamma_N-\gamma_0|\geq x)
 \leq Ce^{-c x} 
 \qquad \mbox{ for each }x\geq0.
\end{equation}
Note that, under the reversal bijection \((\mathrm{inv}\bg)_i=\gamma_{N-i}\), with horizontal steps retaining their colors, the law becomes \(\mathsf Q_N^{\b,\a}\). 
If $N-j=0$, then $\gamma_j=\gamma_N$.   
If \(N-j\geq1\),  using \eqref{eq:ertyby} at time \(N-j\) to the reversed path, 
\[
 \mathsf Q_N^{\a,\b}
 \left(
 |\gamma_j-\gamma_N|\geq x\sqrt j/2
 \right)
 \leq C\exp\left(-c/4\cdot\, x^2j/(N-j)\right)
 \leq Ce^{-cx^2}.
\]  
If \(x>\sqrt j\), the event in \eqref{eq:endpoint-one-time-displacement} is empty. 
When \(0\leq x\leq\sqrt j\), using the triangle inequality and \eqref{eq:endpoint-endpoint-tail-short}, we complete the proof of \eqref{eq:endpoint-one-time-displacement}.

We return to the proof of the claim \eqref{eq:endpoint-height-increment}. 
We first assume \(r<\ell\). In view of $ |\gamma_{r+\ell}-\gamma_r|
 \leq |\gamma_{r+\ell}-\gamma_0|+|\gamma_r-\gamma_0|$, a union bound using \eqref{eq:endpoint-one-time-displacement} proves \eqref{eq:endpoint-height-increment}.
We next assume $r\geq\ell$. If $\b<1$ and $N-r-\ell<\ell$, reverse the path and apply the preceding case to \(\mathsf Q_N^{\b,\a}\) concludes the proof. When $\b<1$ and $N-r-\ell\geq\ell$, we can also assume $N-r-\ell\geq r$ without loss of generality, since one can reverse the path and swap $\a$ and $\b$.
Then \(N\leq3(N-r-\ell)\), and
\eqref{eq:endpoint-pinned-short} gives
 \(4^{r+\ell}B_{N-r-\ell}(v)/\mathsf{Z}_N^{\a,\b}
 \leq C(v+1)\) for every \(v\geq0\).
Using \eqref{eq:endpoint-K-tail-1-short}  and \eqref{eq:endpoint-A1-short},
\begin{align*}
  \mathsf Q_N^{\a,\b}
 \left(
 |\gamma_{r+\ell}-\gamma_r|\geq x\sqrt\ell
 \right) =
 \frac1{ \mathsf{Z}_N^{\a,\b}}
 \sum_{u,v\geq0}
 A_r(u)Z_\ell(u,v)B_{N-r-\ell}(v)
 {\bf1}_{\{|v-u|\geq x\sqrt\ell\}}
 \leq Ce^{-cx^2}.
\end{align*}
We finally consider the case \(\b=1\) and \(r\geq\ell\).  If \(N-r-\ell>r\), then
\(N<3(N-r-\ell)\), and
\eqref{eq:endpoint-H-short} and \eqref{eq:endpoint-free-short} give
  $4^{r+\ell}  Z_{N-r-\ell}(v;\mathrm{free}) /\mathsf{Z}_N^{\a,\b}\leq C(v+1)$.
The preceding argument then applies verbatim.  If
\(N-r-\ell\leq r\), then \(N\leq3r\), \(Z_{N-r-\ell}(v;\mathrm{free})\leq4^{N-r-\ell}\)
for every \(v\geq0\), and
\(4^{N}/\mathsf{Z}_N^{\a,\b} \leq C\sqrt{r+1}\).  By
\eqref{eq:endpoint-K-free-short}, \eqref{eq:endpoint-p-tail-short},
 \eqref{eq:endpoint-A0-short},
\[
 \mathsf Q_N^{\a,1}
 \left(
 |\gamma_{r+\ell}-\gamma_r|\geq x\sqrt\ell
 \right)
 \leq
 C4^{-\ell-r}\sqrt{r+1}
 \left(\sum_{u\geq0}A_r(u)\right)
 \sum_{|d|\geq x\sqrt\ell} \binom{2\ell}{\ell+d}
 \leq Ce^{-cx^2}.
\]
We conclude the proof of the claim \eqref{eq:endpoint-height-increment}.\\

\noindent\textbf{Step 3. } In this step, we conclude the proof of the lemma.  Write $J=\llbracket r+1,r+\ell\rrbracket$ and set 
\[
 \upvarepsilon_i=
 \begin{cases}
  1,&\text{if step \(i\) is horizontal and black},\\
 -1,&\text{if step \(i\) is horizontal and white},\\
  0,&\text{otherwise}.
 \end{cases}
\]
By the definition of \(\Phi\),
\begin{equation}\label{eq:endpoint-occupation-decomposition}
 \sum_{i\in J}\left(\tau_i-\frac12\right)
 =
 \frac12(\gamma_{r+\ell}-\gamma_r)+
 \frac12\sum_{i=r+1}^{r+\ell}\upvarepsilon_i.
\end{equation} We observe that, conditionally on every height sequence
\((\gamma_i)_{i=0}^N\) with positive
\(\mathsf Q_N^{\a,\b}\)-probability, the colors of the horizontal steps are
independent, with each color being black or white with probability
\(1/2\). This is because \(\mathsf Q_N^{\a,\b}\) depends only on
\(\gamma_0\) and \(\gamma_N\).
If the path has \(h\leq\ell\) horizontal steps in \(J\), then
\[
 \mathbb E\left[
 \left.
 \exp\left(\lambda \left(\upvarepsilon_{r+1}+\dots+\upvarepsilon_{r+\ell}\right)   \right)
 \right|(\gamma_i)_{i=0}^N
 \right]
 =\cosh(\lambda)^h\leq e^{h\lambda^2/2}.
\]
Using Chernoff's inequality with \(\lambda=x/\sqrt\ell\) for both tails,
\begin{equation}\label{eq:endpoint-color-tail}
 \mathsf Q_N^{\a,\b}\left(
 \left.
 \left|\upvarepsilon_{r+1}+\dots+\upvarepsilon_{r+\ell}\right|
 \geq x\sqrt\ell
 \right|(\gamma_i)_{i=0}^N
 \right)
 \leq2e^{-x^2/2}.
\end{equation}
Averaging over the uncolored path $(\gamma_i)_{i=0}^N$ gives the same unconditional bound.
In view of \eqref{eq:endpoint-occupation-decomposition}, the event in \eqref{eq:endpoint-Motzkin-moderate-deviation} is contained in the
union of the event in \eqref{eq:endpoint-height-increment} and the color event in \eqref{eq:endpoint-color-tail}.  By a union bound, we conclude the proof of
\eqref{eq:endpoint-Motzkin-moderate-deviation}.
\end{proof}

\subsection{Extension from intervals to general integer sets via negative correlations}

In the following, we extend the moderate deviation results from Lemma~\ref{lem:endpoint-moderate-deviation} to general subsets $I\subseteq\llbracket1,N\rrbracket$.

\begin{lemma}\label{lem: general I}
Assume the same notation as in Lemma \ref{lem:endpoint-moderate-deviation}. Then there exists a constant \(C\in(0,\infty)\), depending only on \(\a\) and \(\b\), such that for every nonempty set  \(I\subseteq\llbracket1,N\rrbracket\) and every \(x\geq0\),
\begin{equation}\label{eq:endpoint-Motzkin-moderate-deviation general I}
 \mathsf Q_N^{\a,\b}
 \left(
 \left|\sum_{i\in I}\tau_i-\frac{|I|}{2}\right|
 \geq x\sqrt{|I|}
 \right)
 \leq C e^{-x^2/2}.
\end{equation}
\end{lemma}

As a key input, we consider a suitable tilting $\mathsf Q_{N,\boldsymbol t}^{\a,\b}$ of the measure $\mathsf Q_N^{\a,\b}$. We show that under this tilted measure $Q_{N,\boldsymbol t}^{\a,\b}$, the occupation variables of the particle configuration associated to a path chosen according to $\mathsf Q_{N,\boldsymbol t}^{\a,\b}$ will be negatively correlated. 

\begin{lemma}\label{lem:negatively correlated}
     Assume $\a\in[0,1)$ and $\b\in[0,1]$. For each $N\in\mathbb N$ and \(\boldsymbol t=(t_1,\ldots,t_N)\in(0,\infty)^N\), define
    \[
    \mathsf{Z}^{\a,\b}(\boldsymbol t)=\sum_{\bg\in\mathrm{BCMot}_N}
 \a^{\gamma_0}\b^{\gamma_N}
 \prod_{i=1}^N t_i^{\tau_i},
    \]
    where we write $(\tau_1,\dots,\tau_N)=\Phi(\bg)$.
    Then $\mathsf{Z}^{\a,\b}(\boldsymbol t)\in(0,\infty)$. Define the probability measure
    \[
 \mathsf Q_{N,\boldsymbol t}^{\a,\b}(\bg)
 :=
 \frac{
 \a^{\gamma_0}\b^{\gamma_N}
 \prod_{i=1}^N t_i^{\tau_i}
 }{\mathsf{Z}^{\a,\b}(\boldsymbol t)},\qquad\bg=(\gamma_0,\dots,\gamma_N)\in\mathrm{BCMot}_N.
\]
Then, under this probability measure, we have $\operatorname{Cov}_{Q_{N,\boldsymbol t}^{\a,\b}}(\tau_i,\tau_j)\leq0$ for each $1\leq i< j\leq N$. 
\end{lemma} 

Let us now conclude the proof of the moderate deviation result in Lemma \ref{lem: general I} assuming the negative correlations from Lemma \ref{lem:negatively correlated}.

\begin{proof}[Proof of Lemma \ref{lem: general I} assuming Lemma \ref{lem:negatively correlated}]
Denote the expectation of $\mathsf Q_N^{\a,\b}$ by $\mathsf E_N^{\a,\b}$. We first show that
\begin{equation}\label{eq:general-I-density-monotonicity}
 \mathsf E_N^{\a,\b}[\tau_1]\geq\mathsf E_N^{\a,\b}[\tau_2]\geq\cdots\geq \mathsf E_N^{\a,\b}[\tau_N].
\end{equation}
Indeed, let $\mathbf D$ and $\mathbf E$ be as in \eqref{eq:DandE}, \(\mathbf C=\mathbf D+\mathbf E\),  $\langle\a|=[1,\a,\a^2,\dots]$, and $|\b\rangle=[1,\b, \b^2,\dots]^T$.
For 
\(k\in\llbracket1,N-1\rrbracket\),
Theorem~\ref{thm:TASEP rep} ensures that
\[
\begin{aligned}
 \mathsf E_N^{\a,\b}[\tau_k]-\mathsf E_N^{\a,\b}[\tau_{k+1}]
 &=
 \frac{
 \langle\a|\mathbf  C^{k-1}
 (\mathbf D\mathbf  C-\mathbf  C \mathbf D)
 \mathbf  C^{N-k-1}|\b\rangle
 }{
 \langle\a|\mathbf  C^N|\b\rangle
 } =
 \frac{
 \langle\a|\mathbf  C^{k-1}P_0
 \mathbf  C^{N-k-1}|\b\rangle
 }{
 \langle\a|\mathbf  C^N|\b\rangle
 }
 \geq0,
\end{aligned}
\]
where  $(P_0)_i^{j}={\bf1}_{\{i=j=0\}}$ for each $i,j\in\mathbb Z_{\geq0}$.

Fix a nonempty \(I\subseteq\llbracket1,N\rrbracket\), and set $Y_I=\sum_{i\in I}\tau_i$. We next prove the claim that
\begin{equation}\label{eq:general-I-centered-concentration}
 \mathsf Q_N^{\a,\b}
 \left(|Y_I-\mathbb E[Y_I]|\geq u\right)
 \leq
 2e^{-2u^2/|I|}\qquad\mbox{ for every }u\geq0.
\end{equation}
Indeed, it suffices to show that 
\begin{equation}\label{eq:hyrmtn}
 \log\mathsf E_N^{\a,\b}
 \left[e^{s(Y_I-\mathbb E[Y_I])}\right]
 \leq \frac{1}{8}|I|s^2\qquad\mbox{ for }s\in\mathbb R,
\end{equation}
and then \eqref{eq:general-I-centered-concentration} follows from Chernoff's inequality with \(s=4u/|I|\) to the
upper tail and with \(s=-4u/|I|\) to the lower tail. Define
\[ 
 \psi(s)=
 \log\mathsf E_N^{\a,\b}\left[e^{sY_I}\right],
 \qquad s\in\mathbb R.
\]
Choose $\boldsymbol{t}\in(0,\infty)^N$ by $t_i=e^s$ for $i\in I$ and $t_i=1$ for $i\notin I$.
Using Lemma \ref{lem:negatively correlated}, and the standard fact that the second derivative of the log-moment generating function yields the variance,
\[ 
 \psi''(s) =
 \operatorname{Var}_{\mathsf Q_{N,\boldsymbol t}^{\a,\b}}(Y_I) =
 \sum_{i\in I}\operatorname{Var}_{\mathsf Q_{N,\boldsymbol t}^{\a,\b}}(\tau_i)
 +
 2\sum_{ i,j\in I,\,i<j} 
 \operatorname{Cov}_{\mathsf Q_{N,\boldsymbol t}^{\a,\b}}(\tau_i,\tau_j)
 \leq \frac{1}{4}|I|. 
\]
Since \(\psi(0)=0\) and \(\psi'(0)=\mathsf E_N^{\a,\b}[Y_I]\)  by Taylor expansion,  we conclude \eqref{eq:hyrmtn} and hence the claim \eqref{eq:general-I-centered-concentration}.

Using Lemma~\ref{lem:endpoint-moderate-deviation}, there exist constants \(C_0,c_0\in(0,\infty)\), depending only on \(\a\) and \(\b\),
such that, for every \(N\in\mathbb N\), every nonempty lattice interval \(J\subseteq\llbracket1,N\rrbracket\), and every \(s\geq0\),
\begin{equation}\label{eq:general-I-interval-input}
 \mathsf Q_N^{\a,\b}
 \left(
 \left|
 \sum_{k\in J}\left(\tau_k- 1/2\right)
 \right|
 \geq s\sqrt{|J|}
 \right)
 \leq C_0e^{-c_0s^2}.
\end{equation}
Let $C_1=C_0\int_0^\infty e^{-c_0s^2}\,ds<\infty$. We next show that 
\begin{equation}\label{eq:general-I-mean-location}
 \left|\mathbb E[Y_I]- |I|/2\right|
 \leq C_1\sqrt{|I|}.
\end{equation}
Indeed, by \eqref{eq:general-I-density-monotonicity},
\[
 \sum_{k=N-|I|+1}^{N}\left( \mathsf E_N^{\a,\b}[\tau_k]- \frac{1}{2}\right)
 \leq
 \mathbb E[Y_I]- \frac{1}{2}|I|
 \leq
 \sum_{k=1}^{|I|}\left(\mathsf E_N^{\a,\b}[\tau_k]- \frac{1}{2}\right).
\]
For each nonempty lattice interval \(J\subseteq\llbracket1,N\rrbracket\),
denote $X_J=
 \left|
 \sum_{k\in J}\left(\tau_k- 1/2\right)
 \right|$.
Using \eqref{eq:general-I-interval-input}, we have
\[
\begin{aligned}
 \left|
 \sum_{k\in J}\left(\mathsf E_N^{\a,\b}[\tau_k]- \frac{1}{2}\right)
 \right| \leq
 \mathsf E_N^{\a,\b}[X_J]
 =
 \int_0^\infty
 \mathsf Q_N^{\a,\b}(X_J>t)\,dt
 \leq
 C_0\sqrt{|J|}
 \int_0^\infty e^{-c_0s^2}\,ds
 =
 C_1\sqrt{|J|}.
\end{aligned}
\]
Combining the last two displayed equations, we conclude the proof of the claim \eqref{eq:general-I-mean-location}.

We return to the proof of the lemma. We will prove \eqref{eq:endpoint-Motzkin-moderate-deviation general I} for $C=\max(2,e^{2C_1^2})$. 
 If \(0\leq x<2C_1\), then $1\leq e^{2C_1^2}e^{-x^2/2}$, and the bound of probability by $1$ suffices.
We next assume \(x\geq2C_1\). By \eqref{eq:general-I-mean-location}, 
\[
 \left\{\left|Y_I- |I|/2\right|\geq x\sqrt{|I|}\right\}
 \subseteq
 \left\{
 |Y_I-\mathbb E[Y_I]|
 \geq x\sqrt{|I|}/2
 \right\}.
\]
Using \eqref{eq:general-I-centered-concentration} with
\(u=x\sqrt{|I|}/2\), we have
$ \mathsf Q_N^{\a,\b} ( |Y_I- |I|/2 |\geq x\sqrt{|I|} )
 \leq2e^{-x^2/2}$,
 which concludes the proof. 
\end{proof}

\subsection{Negative correlations under the tilted measure $\mathsf Q_{N,\boldsymbol t}^{\a,\b}$: proof of Lemma \ref{lem:negatively correlated}}\label{sec:LogConcave}

In order to show the negative correlation for $\mathsf Q_{N,\boldsymbol t}^{\a,\b}$, we will use the matrices $\mathbf{D}$ and $\mathbf{E}$ as well as the vectors $\langle\a|$ and $|\b\rangle$ from the matrix product ansatz from Theorem \ref{thm:TASEP rep} in order to express the covariance between different sites as a determinant of an explicit $2\times2$ matrix; see \eqref{eq:TargetCov}. In order to determine the sign of this determinant, we recall the notation of log-concave sequences, and total non-negativity of matrices.

\begin{definition}\label{def:log concave}
We call a non-negative sequence \(f=(f_h)_{h\in\mathbb Z}\) log-concave if its positive support is an integer interval and $f_h^2\geq f_{h-1}f_{h+1}$ for each $h\in\mathbb Z$. 
\end{definition}

Note that for all $\a,\b >0$, the vectors $\langle\a|=[1,\a,\a^2,\dots]$, $|\b\rangle=[1,\b, \b^2,\dots]^T$ are log-concave. We will use the following basics facts on log-concave sequences.

\begin{lemma}\label{lem:basic properties log concave}
   Suppose \(f=(f_h)_{h\in\mathbb Z}\) is a log-concave sequence, then $f_r f_s\geq f_{r-1}f_{s+1}$ for each $r\leq s$. 
   Moreover, for each $s\geq0$, the sequences $x_h=f_h+s f_{h-1}$ and $y_h=f_h+s f_{h+1}$ are log-concave.
\end{lemma}
\begin{proof}
    To prove the first assertion, if $f_{r-1}f_{s+1}=0$, this follows from non-negativity. Otherwise, by the interval support property, \(f_{r-1},f_r,\ldots,f_s,f_{s+1}\) are positive. Log-concavity then implies that the ratios \(f_h/f_{h-1}\) are non-increasing, and hence $f_r/f_{r-1}\geq f_{s+1}/f_s$. For the second assertion, let \(I=\{h\in\mathbb Z:f_h>0\}\). The sequences \(x\) and \(y\) are
non-negative. If \(s=0\), both have positive support \(I\), whereas if
\(s>0\), their positive supports are \(I\cup(I+1)\) and
\(I\cup(I-1)\), respectively.   Denote
\(\Delta f_h=f_h^2-f_{h-1}f_{h+1}\); then
\[
\begin{aligned}
 \Delta x_h
 &=
 \Delta f_h
 +s\bigl(f_hf_{h-1}-f_{h-2}f_{h+1}\bigr)
 +s^2\Delta f_{h-1}\geq0,\\
 \Delta y_h
 &=
 \Delta f_h
 +s\bigl(f_hf_{h+1}-f_{h-1}f_{h+2}\bigr)
 +s^2\Delta f_{h+1}\geq0,
\end{aligned}
\]
where we use the first assertion.
\end{proof}

\begin{definition}\label{def:TN}
Suppose $\mathbf A$ is a (possibly infinite and not necessarily square) matrix. Then for each $k\in\mathbb N$, we call $\mathbf A$  totally non-negative of order \(k\), abbreviated simply \(\mathrm{TN}_k\), if every minor of the matrix $\mathbf A$ of
order at most \(k\) with increasing row and column indices is non-negative. 
\end{definition}

\begin{proof}[Proof of Lemma \ref{lem:negatively correlated}]
We will use the matrices  
  \begin{equation}\label{eq:teyadvf}
              \mathbf{D}= \begin{bmatrix}
  1 & 1 & 0 & 0 &\dots\\
  0 & 1 & 1 &0 &\dots\\
  0& 0 & 1 &  1 \\
   \vdots & & \ddots&\ddots \\
\end{bmatrix},\qquad
   \mathbf{E}= \begin{bmatrix}
  1  & 0 & 0 & \dots& \\
  1  & 1  & 0&\dots \\
  0& 1  & 1  &  &\\
   \vdots & & \ddots & \ddots\\
\end{bmatrix},  \end{equation}    
and boundary vectors $\langle\a|=[1,\a,\a^2,\dots]$, $|\b\rangle=[1,\b, \b^2,\dots]^T$. 
Denote $\mathbf{C}_t=\mathbf{E}+t\mathbf{D}$. We have
\[
 \mathsf Z^{\a,\b}(\boldsymbol t)
 =
 \langle\a|\mathbf C_{t_1}\cdots\mathbf C_{t_N}|\b\rangle \qquad\mbox{ for }\boldsymbol t=(t_1,\ldots,t_N)\in(0,\infty)^N,
\]
where all summands are non-negative. We claim that 
\[
0<\mathsf{Z}^{\a,\b}(\boldsymbol t)\leq 2^N(1-\a)^{-1} (1+t_1)\dots(1+t_N)<\infty.
\]
Indeed, the first inequality follows from retaining only the constant zero path colored black at all steps, and the second  follows from 
$\Vert \langle\a|\Vert_{\ell^1}=(1-\a)^{-1}$, $\Vert|\b\rangle\Vert_{\ell^\infty}=1$, and $\sup_{h\geq0}\sum_{k\geq0}(\mathbf C_t)_h^k\leq 2(1+t)$ for $t\in(0,\infty)$.

Fix $1\leq i< j\leq N$. For each $\boldsymbol{t}\in(0,\infty)^N$, define 
\begin{align*}
     \boldsymbol w
 &=
 \langle\a|\mathbf{C}_{t_1}\cdots \mathbf{C}_{t_{i-1}},
 \qquad
 \mathbf B
 =
 \mathbf{C}_{t_{i+1}}\cdots \mathbf{C}_{t_{j-1}},
 \qquad
 \boldsymbol v 
 =
 \mathbf{C}_{t_{j+1}}\cdots \mathbf{C}_{t_N}|\b\rangle,\\
 \boldsymbol{\mathsf X}&=
 \begin{bmatrix}
  \boldsymbol w \mathbf E\\
  \boldsymbol w \mathbf D
 \end{bmatrix},
 \qquad\qquad
 \boldsymbol{\mathsf Y}=
 \begin{bmatrix}
  \mathbf D\boldsymbol v \quad \mathbf E\boldsymbol v 
 \end{bmatrix}.
\end{align*}  These matrices and vectors all have finite non-negative entries. We split the rest of the proof into two steps.\\

\noindent\textbf{Step 1. } In this step, we conclude the proof of the lemma, assuming that   
\begin{equation}\label{eq:claim det}
 \det(\boldsymbol{\mathsf X}\mathbf B\boldsymbol{\mathsf Y})\geq0\qquad\mbox{ for each }\boldsymbol{t}\in(0,\infty)^N.
\end{equation}
Here, we note that $\boldsymbol{\mathsf X}\mathbf B\boldsymbol{\mathsf Y}$ is a finite $2\times 2$ matrix, given in \eqref{def:XBY}, since the rows of \(\boldsymbol{\mathsf X}\) are \(\ell^1\), the columns of \(\boldsymbol{\mathsf Y}\) are \(\ell^\infty\), and $\sup_{r\geq0}\sum_{s\geq0}(\mathbf B)_r^s
 \leq
 \prod_{m=i+1}^{j-1}2(1+t_m)<\infty$.

For each \(\kappa,\tilde{\kappa}\in\{0,1\}\), let \(\mathsf Z_{\kappa\tilde{\kappa}}\) be the contribution to
\(\mathsf Z^{\a,\b}(\boldsymbol t)\) from paths satisfying $\Phi(\bg)_i=\kappa$ and $\Phi(\bg)_j=\tilde{\kappa}$. The rows of \(\boldsymbol{\mathsf X}\) are ordered by
\(\kappa\in\{0,1\}\), whereas the columns of
\(\boldsymbol{\mathsf Y}\) are ordered by \(\tilde{\kappa}\in\{1,0\}\). Hence
\begin{equation}
    \label{def:XBY}
 \boldsymbol{\mathsf X}\mathbf B\boldsymbol{\mathsf Y}
 =
 \begin{bmatrix}
  \mathsf Z_{01}t_j^{-1} & \mathsf Z_{00}\\
  \mathsf Z_{11}(t_it_j)^{-1} & \mathsf Z_{10}t_i^{-1}
 \end{bmatrix},
\end{equation}
and we have
 $\det(\boldsymbol{\mathsf X}\mathbf B\boldsymbol{\mathsf Y})
 = (\mathsf Z_{01}\mathsf Z_{10}-\mathsf Z_{00}\mathsf Z_{11})/(t_i t_j)$. Therefore \eqref{eq:claim det} implies 
\begin{equation}
\begin{aligned}\label{eq:TargetCov}
 \operatorname{Cov}_{\mathsf Q_{N,\boldsymbol t}^{\a,\b}}
 \left(\Phi(\bg)_i,\Phi(\bg)_j\right)
 =
 \frac{\mathsf Z_{00}\mathsf Z_{11}
       -\mathsf Z_{01}\mathsf Z_{10}}
      {\mathsf Z^{\a,\b}(\boldsymbol t)^2} =
 -\frac{t_it_j}
 {\mathsf Z^{\a,\b}(\boldsymbol t)^2}
 \det(\boldsymbol{\mathsf X}\mathbf B\boldsymbol{\mathsf Y})
 \leq0, 
\end{aligned}
\end{equation}
and we conclude the proof of the lemma.\\

\noindent\textbf{Step 2. } In this step, we prove the claim \eqref{eq:claim det}. We first note that it suffices to show that the three matrices $\boldsymbol{\mathsf X}$, $\mathbf B$, and $\boldsymbol{\mathsf Y}$ are $\mathrm{TN}_2$ as in Definition \ref{def:TN}. Indeed, let \(P_K\) be the coordinate projection onto
\(\llbracket0,K\rrbracket\). By the preceding finiteness and
non-negativity, the monotone convergence theorem gives
 $\lim_{K\to\infty}\boldsymbol{\mathsf X}P_K\mathbf BP_K\boldsymbol{\mathsf Y}=
 \boldsymbol{\mathsf X}\mathbf B\boldsymbol{\mathsf Y}$
entrywise.  
By the Cauchy--Binet formula, if
\(\boldsymbol{\mathsf X}\), \(\mathbf B\), and
\(\boldsymbol{\mathsf Y}\) are \(\mathrm{TN}_2\), then
\[
\begin{aligned}
 \det(\boldsymbol{\mathsf X}P_K\mathbf BP_K\boldsymbol{\mathsf Y})
 &=
 \sum_{0\leq r<s\leq K}
 \sum_{0\leq k<\ell\leq K}
 \det \boldsymbol{\mathsf X}_{\cdot,\{r,s\}}\,
 \det\mathbf B_{\{r,s\},\{k,\ell\}}\,
 \det \boldsymbol{\mathsf Y}_{\{k,\ell\},\cdot} \geq0,
\end{aligned}
\] 
which concludes the proof of \eqref{eq:claim det}. In the rest of the proof, we show that $\boldsymbol{\mathsf X}$, $\mathbf B$,  and $\boldsymbol{\mathsf Y}$  are $\mathrm{TN}_2$.

Set $\mathbf D={\bf Id}+\mathbf d$ and $\mathbf E={\bf Id}+\mathbf e$, where ${\bf Id}$ is the identity matrix, 
$\mathbf d_h^k={\bf1}_{\{k=h+1\}}$ and $\mathbf e_h^k={\bf1}_{\{h=k+1\}}$ for each $h,k\in\mathbb Z_{\geq0}$. 
Then $\mathbf C_t=\mathbf E+t\mathbf D=({\bf Id}+t\mathbf d)({\bf Id}+\mathbf e)$. 
We claim that ${\bf Id}+s\mathbf d$ and ${\bf Id}+s\mathbf e$ are  \(\mathrm{TN}_2\) for each $s\geq0$. Indeed, consider \(r_1<r_2\) and \(k_1<k_2\). Then for ${\bf Id}+s\mathbf d$, if two antidiagonal entries would simultaneous vanish, this yields $r_2\leq k_1<k_2\leq r_1+1$, giving a contradiction. Similarly, for ${\bf Id}+s\mathbf e$, if two antidiagonal entries  simultaneous vanish, this would imply $ r_2-1\leq k_1<k_2\leq r_1$,  giving a contradiction. This yields the claim for  $\mathbf C_t$.
Iterating the Cauchy--Binet formula, we also know that  \(\mathbf B\) is
\(\mathrm{TN}_2\). 

Note that the Cauchy--Binet sums contain only finitely many nonzero terms, and that the geometric vectors $\langle\a|$ and $|\b\rangle$ are log-concave. For each row vector $f$ and column vector $g$ which are log-concave, and each $s\geq0$, we have that 
\[
\begin{array}{ll}
 (f({\bf Id}+s\mathbf d))_h=f_h+s f_{h-1},
 &
 (f({\bf Id}+s\mathbf e))_h=f_h+s f_{h+1},\\[1mm]
 (({\bf Id}+s\mathbf e)g)_h=g_h+s g_{h-1},
 &
 (({\bf Id}+s\mathbf d)g)_h=g_h+s g_{h+1}.
\end{array}
\]  Hence all four row and column vectors are
log-concave.
Therefore, by a finite product, the row vector $\boldsymbol w$ and the column vector $\boldsymbol v$ are both log-concave. Using the zero extensions of \(w_h\) and \(v_h\), set
\[
 p_h=w_h+w_{h+1},
 \qquad
 q_h=v_h+v_{h+1},
 \qquad h\in\mathbb Z.
\]
Lemma \ref{lem:basic properties log concave} shows that \(p\) and
\(q\) are log-concave. For each \(h\geq0\), we have
\[
 (\boldsymbol w\mathbf E)_h=p_h,
 \qquad
 (\boldsymbol w\mathbf D)_h=p_{h-1},
 \qquad
 (\mathbf D\boldsymbol v)_h=q_h,
 \qquad
 (\mathbf E\boldsymbol v)_h=q_{h-1}.
\]
Therefore, for \(0\leq r<s\), the first assertion of Lemma
\ref{lem:basic properties log concave} gives
\[
 \det\boldsymbol{\mathsf X}_{\cdot,\{r,s\}}
 =p_rp_{s-1}-p_{r-1}p_s\geq0,
 \qquad
 \det\boldsymbol{\mathsf Y}_{\{r,s\},\cdot}
 =q_rq_{s-1}-q_{r-1}q_s\geq0.
\]
We conclude the proof that \(\boldsymbol{\mathsf X}\) and \(\boldsymbol{\mathsf Y}\) are \(\mathrm{TN}_2\).
\end{proof}

\subsection{Moderate deviations for the open ASEP}

We have now all tools in order to conclude the proof of Theorem \ref{thm:moderate deviation MC again}. 

\begin{proof}[Proof of Theorem \ref{thm:moderate deviation MC again}]
We can assume that \(\a\in[0,1)\) and \(\b\in[0,1]\) as the case $\a=1$ then follows from the particle-hole duality in Lemma \ref{lem:particle hole}. 
Throughout the proof, we write $C_i\in(0,\infty)$, $i\in\mathbb N$ to denote constants depending only on \(q,\a,\b,\c,\d\).
Recall the probability measure $\mathbb P_N^{(q,\a,\b,\c,\d)}$ from Definition \ref{def:measure Motzkin} and $ \mathsf Q_N^{\a,\b}$ from Lemma \ref{lem:endpoint-moderate-deviation}. 
Denote their expectations by $\mathbb E_N^{(q,\a,\b,\c,\d)}$  and $\mathsf E_N^{\a,\b}$.
It suffices to show that  
\begin{equation}\label{eq:suffices moderate d}
    \mathbb P_N^{(q,\a,\b,\c,\d)}(E)
 \leq
 C_1 
 \mathsf Q_N^{\a,\b}(E)\qquad\mbox{ for each } E\subseteq\mathrm{BCMot}_N.
\end{equation}
Indeed, by Theorem \ref{thm:Motzkin representation}, for each nonempty set \(I\subseteq\llbracket1,N\rrbracket\) and $x\geq0$,
\begin{align*}
    & \mu_N^{(q,\a,\b,\c,\d)}
 \left(
 \left|\sum_{i\in I}\tau_i- |I|/2\right|
 \geq x\sqrt{|I|}
 \right) =
 \mathbb P_N^{(q,\a,\b,\c,\d)}(E_{I,x}),\\
 &\mbox{ where }\qquad
 E_{I,x}
 =
 \left\{
 \bg\in\mathrm{BCMot}_N:
 \left|
 \sum_{i\in I}\Phi(\bg)_i-|I|/2
 \right|
 \geq x\sqrt{|I|}
 \right\}.
\end{align*} 
Using Lemma \ref{lem: general I} with a constant $C_2$, combining with \eqref{eq:suffices moderate d}, 
we have $\mathbb P_N^{(q,\a,\b,\c,\d)}(E_{I,x})\leq C_1C_2e^{-x^2/2}$, and we conclude the proof of the theorem.

Recall the endpoint weight $ \uptheta(h)= (\c\d;q)_h/(q;q)_h$ for $h\in\mathbb Z_{\geq0}$, and the bulk weights, for $\bg\in\mathrm{BCMot}_N$, 
\begin{equation}\label{eq:npmom}
\upomega_i(\bg)=
\begin{cases}
1-q^{\gamma_i},
&\gamma_i-\gamma_{i-1}=1,\\
1-\c\d q^{\gamma_i},
&\gamma_i-\gamma_{i-1}=-1,\\
1+\d q^{\gamma_i},
&\gamma_i-\gamma_{i-1}=0
  \text{ and the \(i\)-th step is black},\\
1+\c q^{\gamma_i},
&\gamma_i-\gamma_{i-1}=0
  \text{ and the \(i\)-th step is white}.
\end{cases}
\end{equation}
We write $W_{1,N}(\bg)=\prod_{i=1}^N\upomega_i(\bg)$.
By definition, we have
\begin{equation}\label{eq:MC-MD-path-RN}
 \mathbb P_N^{(q,\a,\b,\c,\d)}( E)
 =
 \frac{
 \mathsf E_N^{\a,\b}
 \left[
 \uptheta(\gamma_N)W_{ 1,N }\mathbf1_{E}
 \right]
 }{
 \mathsf E_N^{\a,\b}
 \left[
 \uptheta(\gamma_N)W_{1,N}
 \right]
 }\qquad\mbox{ for each } E\subseteq\mathrm{BCMot}_N.
\end{equation}  
Recall the probability measure $\mathrm P_N^{\a,\b;\uptheta}$ from Proposition \(\ref{prop:endpt wt compare}\). Denote its expectation by \(\mathrm E_N^{\a,\b;\uptheta}\).
We have 
\begin{equation}\label{eq:RND theta}
 \frac{d\mathrm P_N^{\a,\b;\uptheta}}
 {d\mathsf Q_N^{\a,\b}}(\bg)
 =
 \frac{\uptheta(\gamma_N)}
 {\mathsf E_N^{\a,\b}[\uptheta(\gamma_N)]}, \qquad\bg\in\mathrm{BCMot}_N.
\end{equation} 
We claim that, to prove \eqref{eq:suffices moderate d}, it suffices to show that
\begin{equation}\label{eq:W1Nbig}
    \mathrm E_N^{\a,\b;\uptheta}[W_{1,N}]\geq C_3>0.
\end{equation} 
Indeed, as in \eqref{eq:boundsthetanmi}, we have
\begin{equation}\label{eq:boundsthetanmi final}
0< (\c\d;q)_\infty
\leq \uptheta(h)
\leq   (q;q)_\infty^{-1}<\infty 
\qquad\mbox{ for } h\in\mathbb Z_{\geq0}. 
\end{equation} 
Combining \eqref{eq:RND theta}, \eqref{eq:W1Nbig}, and \eqref{eq:boundsthetanmi final}, 
\begin{align*}
    \mathsf E_N^{\a,\b}
 [\uptheta(\gamma_N)W_{1,N}]
 &=
 \mathsf E_N^{\a,\b}
 [\uptheta(\gamma_N)]
 \,
 \mathrm E_N^{\a,\b;\uptheta}[W_{1,N}] \geq (\c\d;q)_\infty C_3,\\
  \mathsf E_N^{\a,\b}
 \left[
 \uptheta(\gamma_N)W_{1,N}\mathbf1_{E}
 \right]
 &\leq(q;q)_\infty^{-1}\mathsf Q_N^{\a,\b}(E),
\end{align*}
and dividing them, in view of \eqref{eq:MC-MD-path-RN}, we conclude the proof of \eqref{eq:suffices moderate d}.

In the remainder of the proof, we prove the claim \eqref{eq:W1Nbig}. Let 
\[
 C_4
 =
 \max\left(
 \frac{1}{1-q},
 \frac{\c\d}{1-\c\d},
 \frac{-\c}{1+\c},
 \frac{-\d}{1+\d}
 \right)<\infty.
\]
Then, using the definition of \eqref{eq:npmom} of the bulk local weights, and the fact that $ -\log(1-u)\leq u/(1-u)$ for $u\in[0,1)$, one can check that each local weight satisfies
$ -\log\upomega_i(\bg)
 \leq C_4 q^{\gamma_i}$. Summing over \(i\) gives
\begin{equation}\label{eq:MC-MD-global-log-bound}
 W_{1,N}(\bg)
 \geq
 \exp\left(
 -C_4   G(\bg) 
 \right),
\end{equation}
where we denote $G(\bg)=q^{\gamma_1}+\dots+q^{\gamma_N}$.
It then suffices to show that 
\begin{equation}\label{eq:MC-MD-uniform-q-height}
C_5= \sup_{N\in\mathbb N}
 \mathrm E_N^{\a,\b;\uptheta}
 \left[G(\bg)\right]<\infty.
\end{equation}
Indeed, by
Markov's inequality,
$ \mathrm P_N^{\a,\b;\uptheta}
 \left(
G(\bg)>2\max(1,C_5)
 \right)
 \leq \frac{1}{2}$.
Hence, using \eqref{eq:MC-MD-global-log-bound},
\[ 
 \mathrm E_N^{\a,\b;\uptheta}[W_{1,N}]
 \geq
 \frac12e^{-2C_4 \max(1,C_5)} , 
\]
and we conclude the proof of the claim \eqref{eq:W1Nbig}. It remains to establish \eqref{eq:MC-MD-uniform-q-height}. 
Assume that $N\geq2$. When \(\b<1\), by Proposition \(\ref{prop:endpt wt compare}\), for all $1\leq i<N$
\[
 \mathrm E_N^{\a,\b;\uptheta}[q^{\gamma_i}]
 \leq
 C_6
 \frac{(N+1)^{3/2}}
 {(i+1)^{3/2}(N-i+1)^{3/2}}
\]
 Therefore, we see that
 \[
\begin{aligned}
 \sum_{i=1}^{N-1}
 \mathrm E_N^{\a,\b;\uptheta}[q^{\gamma_i}]
 &\leq
 C_7\sum_{i=1}^{\lfloor N/2\rfloor}(i+1)^{-3/2} +
 C_7\sum_{i=\lfloor N/2\rfloor+1}^{N-1}
 (N-i+1)^{-3/2}
 \leq C_8.
\end{aligned}
\]
Since \(q^{\gamma_N}\leq1\), this proves \eqref{eq:MC-MD-uniform-q-height} when \(\b<1\).
When \(\b=1\), Proposition \(\ref{prop:endpt wt compare}\) gives for all $1\leq i<N$
\[
 \mathrm E_N^{\a,1;\uptheta}[q^{\gamma_i}]
 \leq
 \frac{C_9}{(i+1)^{3/2}}
 \left(\frac{N+1}{N-i+1}\right)^{1/2}
 \quad \mbox{ and }  \quad\mathrm E_N^{\a,1;\uptheta}[q^{\gamma_N}]
 \leq\frac{C_9}{N+1}.
\]
By a similar argument as in the previous case, 
\[
\begin{aligned}
 \sum_{i=1}^{N-1}
 \mathrm E_N^{\a,1;\uptheta}[q^{\gamma_i}]
 &\leq
 C_{10}\sum_{i=1}^{\infty}(i+1)^{-3/2}
 + C_{10} (N+1)^{-1}\sum_{r=1}^{N+1}r^{-1/2}
 \leq C_{11}.
\end{aligned}
\]
Adding the bound involving $\gamma_N$ proves
\eqref{eq:MC-MD-uniform-q-height} when \(\b=1\). We conclude \eqref{eq:MC-MD-uniform-q-height} and hence the proof.
\end{proof}

\appendix

\section{Proof of combinatorial identities}\label{app:Combinatorics}
In the following, we give the deferred proofs of the combinatorial identities from Section~\ref{sec:Combinatorial}. We start with the proof of Lemma~\ref{lem:bounds on combinatorial numbers}, which we restate as Lemma~\ref{lem:bounds on combinatorial numbers AGAIN} for clarity.

\begin{lemma}\label{lem:bounds on combinatorial numbers AGAIN}
    For every $n\in\mathbb Z_{\geq0}$ and $m\in\llbracket-n,n\rrbracket$,
    \[
 \frac{4^{n-1}}{e^2\sqrt{n+1}}
\exp\!\left(-(\log4)\frac{m^2}{n+1}\right) \leq \binom{2n}{n+m}\leq  \frac{4^n}{\sqrt{n+1}}
\exp\!\left(- \frac{m^2}{n+1}\right).
    \]
\end{lemma} 
\begin{proof}
When \(n=0\), the inequalities hold trivially. We next assume $n\in\mathbb N$. Stirling's inequality implies that 
    \begin{equation}\label{eq:from Stirling}
    \sqrt r\left(\frac re\right)^r\leq r!\leq e\sqrt{r+1}\left(\frac re\right)^r
    \qquad\mbox{for every }r\in\mathbb Z_{\geq0}.
    \end{equation} Therefore,
    \begin{equation*} 
\begin{aligned}
 \binom{2n}{n+m}
 &\geq \frac1{e^2}
 \sqrt{\frac{2n}{(n+m+1)(n-m+1)}}\,
 \frac{(2n)^{2n}}
 {(n+m)^{n+m}(n-m)^{n-m}}\\
 &=\frac1{e^2}
 \sqrt{\frac{2n}{(n+m+1)(n-m+1)}}\,
 4^n\exp\!\left(-nI\!\left(\frac mn\right)\right)
 \geq \frac{4^{n-1}}{e^2\sqrt{n+1}}
\exp\!\left(-(\log4)\frac{m^2}{n+1}\right) ,
\end{aligned}
\end{equation*}
where in the first step we use \eqref{eq:from Stirling}; in the
second step we set
\[
I(t)=(1+t)\log(1+t)+(1-t)\log(1-t),\qquad t\in[-1,1],
\]
where \(0\log0:=0\);
and in the third step we use 
$\sqrt{\frac{2n}{(n+m+1)(n-m+1)}}
\geq\frac{\sqrt{2n}}{n+1}
\geq\frac1{\sqrt{n+1}}$ together with $I(t)\leq(\log4)t^2$ for $t\in[-1,1]$. The latter follows from $I(t)=I(-t)$, \(I(1)=\log4\), and 
$\frac{d}{dt}(I(t)/t^2)\geq0$ for \(t\in(0,1)\).

Note that
\begin{equation}\label{eq:qwer1}
\binom{2n}{n}\leq\frac{4^n}{\sqrt{n+1}} \qquad\mbox{ for }n\in\mathbb Z_{\geq0}.
\end{equation}
Indeed, for
\(k\in\mathbb Z_{\geq0}\), let
$a_k=\frac{\sqrt{k+1}}{4^k}\binom{2k}{k}$. Then \(a_0=1\) and
$\left(\frac{a_{k+1}}{a_k}\right)^2
=\frac{(2k+1)^2(k+2)}{4(k+1)^3}
=1-\frac{3k+2}{4(k+1)^3}\leq1$.
Hence \(a_k\leq1\) for every \(k\geq0\), proving \eqref{eq:qwer1}. The upper bound follows from the symmetry \(\binom{2n}{n+m}=\binom{2n}{n-m}\) and
\[
 \log\frac{\binom{2n}{n+m}}{\binom{2n}{n}}=\sum_{j=1}^m\log\left(\frac{n-j+1}{n+j}\right)
 \leq-\frac{2}{2n+1}\sum_{j=1}^m(2j-1)=-\frac{2m^2}{2n+1} 
 \leq-\frac{m^2}{n+1}\qquad\mbox{for }m\in\llbracket0,n\rrbracket,
\]
where we have used $ \log\!\left(\frac{1-x}{1+x}\right)\leq-2x$ for $x\in[0,1)$.
\end{proof}

Next, we show effective upper and lower bounds on the partition function $Z_N(x,y)$ from Lemma~\ref{lem:bounds ZNxy}.

\begin{lemma}\label{lem:bounds ZNxy AGAIN}
For $N,x,y\in\mathbb Z_{\geq0}$, recall $Z_N(x,y)$ from Definition~\ref{def:bicolored Motzkin path}. Then for every $N\in\mathbb Z_{\geq0}$,
\begin{align*}
     Z_N(x,y)&\le
  4^{N+1}\frac{(x+1)(y+1)}{(N+1)^{3/2}}
 \exp\!\left(- \frac{(x-y)^2}{N+1}\right)\qquad\mbox{ for every }x,y\in\mathbb Z_{\geq0},\\
 Z_N(0,y)
 &\ge
 e^{-2}4^{N-1}  \, \frac{y+1}{(N+1)^{3/2}}
 \exp\!\left(-(\log 4)\frac{y^2}{N+1}\right)\qquad\mbox{ for every }y\in\llbracket0,N\rrbracket.
\end{align*}
\end{lemma}
\begin{proof}
    Recall from Lemma~\ref{lem:count biMot} that
\begin{equation}\label{eq:formula ZNxy again}
 Z_N(x,y)=
 \binom{2N}{N+x-y}-\binom{2N}{N+x+y+2}.
\end{equation}
By Lemma \ref{lem:bounds on combinatorial numbers}, when $|x-y|\leq N$,  
\begin{equation}\label{eq:binomial-bounds-for-ZNxy}
 \frac{1}{4e^2}\frac{4^N}{\sqrt{N+1}}
 \exp\!\left(-(\log 4)\frac{(x-y)^2}{N+1}\right)
 \le  \binom{2N}{N+x-y}\le
 \frac{4^N}{\sqrt{N+1}}
 \exp\!\left(-\frac{(x-y)^2}{N+1}\right).
\end{equation} 

To prove the first assertion of the lemma, in view of \eqref{eq:binomial-bounds-for-ZNxy}, we only need to show that
\begin{equation}\label{eq:qwer2}
     Z_N(x,y)\le
 \frac{4(x+1)(y+1)}{N+1}\binom{2N}{N+x-y}.
\end{equation}
Indeed, if $x+y+2> N$, we have $\binom{2N}{N+x+y+2}=0$ and
$\frac{4(x+1)(y+1)}{N+1}\ge2$, and we conclude \eqref{eq:qwer2}. If $x+y+2\leq N$,  
\begin{align*}
 1-\frac{\binom{2N}{N+x+y+2}}{ \binom{2N}{N+x-y}}&=1-\prod_{j=|x-y|+1}^{x+y+2}\frac{N-j+1}{N+j}
  \le \sum_{j=|x-y|+1}^{x+y+2}
 \left(1-\frac{N-j+1}{N+j}\right) =\sum_{j=|x-y|+1}^{x+y+2}\frac{2j-1}{N+j}\\
 &\le\frac{1}{N+1}\sum_{j=|x-y|+1}^{x+y+2}(2j-1)
 =\frac{(x+y+2)^2-(x-y)^2}{N+1} =\frac{4(x+1)(y+1)}{N+1},
\end{align*}
where we used the fact that $1-\prod_j u_j\le\sum_j(1-u_j)$ for $u_j\in[0,1]$, and we conclude \eqref{eq:qwer2}.

We next prove the second assertion of the lemma. When $y\in\llbracket0,N\rrbracket$,
\begin{equation} 
 Z_N(0,y)=
 \frac{2(2N+1)(y+1)}{(N+y+1)(N+y+2)}
 \binom{2N}{N-y}\geq\frac{y+1}{N+1}\binom{2N}{N-y},
\end{equation}
where the first step follows from \eqref{eq:formula ZNxy again} with $x=0$, and the second step follows from  $(N+y+1)(N+y+2)\le(2N+1)(2N+2)$.
Combining with \eqref{eq:binomial-bounds-for-ZNxy} with $x=0$ we conclude the proof. 
\end{proof}

Next, we establish Lemma~\ref{lem:bounds ZNxfree} on the function $Z_N(x,\mathrm{free})$. 

\begin{lemma}\label{lem:bounds ZNxfree AGAIN}
For $N,x\in\mathbb Z_{\geq0}$, recall $Z_N(x,\mathrm{free})$ from Definition~\ref{def:bicolored Motzkin path}. Then for every $N,x\in\mathbb Z_{\geq0}$,
\begin{align*}
    & Z_N(x,\mathrm{free})
 \leq
 2\,4^N
 \min\left\{1,\frac{x+1}{\sqrt{N+1}}\right\}
 \qquad\mbox{ and }\qquad  Z_N(0,\mathrm{free}) 
 \geq
 \frac{e^{-2}4^{N-1}}{\sqrt{N+1}}.
\end{align*} 
Moreover, we have that for each fixed $L\in(0,\infty)$, uniformly for $n\geq3$ and $x\leq L\log n$,
\begin{equation}\label{eq:converse-short-free-expansion AGAIN}
\frac{Z_n(x,\mathrm{free})}{4^n}
=\frac1{\sqrt{\pi n}}
\left(2(x+1)-\frac1n
\left(\frac23(x+1)^3+\frac7{12}(x+1)\right)
+O_L\!\left(\frac{(x+1)^5}{n^2}\right)\right).
\end{equation}
\end{lemma}
\begin{proof}
Using Lemma~\ref{lem:count biMot} and binomial symmetry,  
\begin{align}
 Z_N(x,\mathrm{free})
 &=
 \sum_{y\geq0}
 \left[
 \binom{2N}{N+x-y}
 -
 \binom{2N}{N+x+y+2}
 \right] =  
 \sum_{j=N-x-1}^{N+x}\binom{2N}{j}.
 \label{eq:exact-count-ZNxfree}
\end{align}
Note that $Z_N(x,\mathrm{free})\leq4^N$ by the binomial theorem. On the other hand, the sum in \eqref{eq:exact-count-ZNxfree} has $2x+2$ terms. 
We claim that $\binom{2N}{j}\leq\frac{4^N}{\sqrt{N+1}}$ for every $j\in\mathbb Z$. Indeed, if $j\in\llbracket0,2N\rrbracket$ then this follows from Lemma~\ref{lem:bounds on combinatorial numbers} with $n=N$ and $m=j-N$, and if $j\notin\llbracket0,2N\rrbracket$ then $\binom{2N}{j}=0$. Therefore,
\[
 Z_N(x,\mathrm{free})
 \leq
 \min\left\{4^N,\frac{2(x+1)4^N}{\sqrt{N+1}}\right\}
 \leq
 2\,4^N\min\left\{1,\frac{x+1}{\sqrt{N+1}}\right\}.
\]
Finally, setting $x=0$ in \eqref{eq:exact-count-ZNxfree} and using the lower bound in Lemma~\ref{lem:bounds on combinatorial numbers},
\[
 Z_N(0,\mathrm{free})
 =
 \binom{2N}{N-1}+\binom{2N}{N}
 \geq
 \binom{2N}{N}
 \geq
 \frac{e^{-2}4^{N-1}}{\sqrt{N+1}}.
\] 

We now show the expansion \eqref{eq:converse-short-free-expansion AGAIN}. 
As a first step, we claim that for each fixed $L\in(0,\infty)$, 
\begin{equation}\label{eq:asymp of combinatoric}
4^{-n}\binom{2n}{n+j}
=\frac1{\sqrt{\pi n}}
\left(1-\frac{j^2+1/8}{n}
+O_L\!\left(\frac{(1+|j|)^4}{n^2}\right)\right)
\qquad\mbox{uniformly for }|j|\leq L\log n.
\end{equation}
Indeed, Stirling's formula 
implies
\begin{equation}\label{eq:central comb}
4^{-n}\binom{2n}{n}
=\frac1{\sqrt{\pi n}}
\left(1-\frac1{8n}+O(n^{-2})\right).
\end{equation}
When $n$ is sufficiently large depending only on $L$,  Taylor's theorem gives, uniformly in $0\leq j\leq L\log n$,
\[
\begin{aligned}
&\log\frac{\binom{2n}{n+j}}{\binom{2n}{n}}
=\sum_{s=1}^j\left[
\log\left(1-\frac{s-1}{n}\right)
-\log\left(1+\frac{s}{n}\right)
\right]\\
&=-\frac1n\sum_{s=1}^j(2s-1)
+\frac1{2n^2}\sum_{s=1}^j\left(s^2-(s-1)^2\right)
+O_L\!\left(\frac{(1+j)^4}{n^3}\right) =-\frac{j^2}{n}+\frac{j^2}{2n^2}
+O_L\!\left(\frac{(1+j)^4}{n^3}\right).
\end{aligned}
\]
Using binomial symmetry and combining with \eqref{eq:central comb}, we conclude the proof of \eqref{eq:asymp of combinatoric}. 

Now by \eqref{eq:exact-count-ZNxfree},
\[
\frac{Z_n(x,\mathrm{free})}{4^n}
=\sum_{j=-x-1}^{x}4^{-n}\binom{2n}{n+j}.
\]
If $n\geq3$ and $x\leq L\log n$, then $x+1\leq(L+1)\log n$. Hence \eqref{eq:asymp of combinatoric} with $L+1$ instead of $L$, together with 
 $\sum_{j=-h}^{h-1}j^2= (2h^3+h)/3$ and $\sum_{j=-h}^{h-1}(1+|j|)^4=O(h^5)$, gives \eqref{eq:converse-short-free-expansion} for all sufficiently
large $n$. Increasing the constant in the $O$ notation to cover the finitely many remaining pairs $(n,x)$ proves \eqref{eq:converse-short-free-expansion AGAIN}.
\end{proof}

In order to show Lemma \ref{lem:comparison kernels}, we require the following combinatorial estimate.

\begin{lemma}\label{lem:smoothing}
For \(n\in\mathbb Z_{\geq0}\) and \(r\in\mathbb Z\), let
\[
    f_n(r)=\binom{2n}{n+r}-\binom{2n}{n+r+2}
    \qquad\mbox{and}\qquad
    \Delta^2 f_n(r)=f_n(r+2)-2f_n(r+1)+f_n(r).
\]
Then 
\[
    \left|\Delta^2 f_n(r)\right|
    \leq
     4^{n+3}\frac{|r+2|}{(n+1)^{5/2}}
    \exp\!\left(-\frac{(r+2)^2}{6(n+1)}\right).
\]
\end{lemma}
\begin{proof}
    By a computation,
    \[
        \Delta^2 f_n(r)
    =
    \frac{(r+2)\left(2(r+2)^2-3n-5\right)}
    {(n+1)(n+2)(2n+3)}
    \binom{2n+4}{n+r+4}.
    \]
If $|r+2|>n+2$, then 
$\binom{2n+4}{n+r+4}=0$.
and the result holds. We next assume $|r+2|\leq n+2$.
By the upper bound in Lemma~\ref{lem:bounds on combinatorial numbers}
with $n$ there replaced by $n+2$ and $m=r+2$,
\[
   \binom{2n+4}{n+r+4}
    \leq
    \frac{4^{n+2}}{\sqrt{n+3}}
    \exp\!\left(-\frac{(r+2)^2}{n+3}\right).
\]
Note that 
$(n+1)(n+2)(2n+3)\geq2(n+1)^3$,
$ |2(r+2)^2-3(n+1)-2 | \leq2(r+2)^2+5(n+1)$, and
$ n+3\leq3(n+1)$.
Combining with the fact that $ \left(a+ 5/2\right)e^{-a/3}
    \leq4e^{-a/6}$ for every $a\geq0$,
we conclude the proof.
\end{proof}

We can now show Lemma~\ref{lem:comparison kernels}, which we restate  for clarity.

\begin{lemma}\label{lem:comparison kernels AGAIN}
       For every $N,x,y\in\mathbb Z_{\geq0}$ satisfying
  $(x+2)^2\leq  N+1 $, 
    \[
     \left|
 \frac{Z_N(x,y)}{x+1}-Z_N(0,y)
 \right|
 \leq 4^{N+3}\frac{(x+1)^2(x+y+2)}{(N+1)^{5/2}}
 \exp\left(-\frac{y^2}{12(N+1)}\right).
    \]
\end{lemma}
\begin{proof}
    We use the notation from Lemma \ref{lem:smoothing}. By Lemma~\ref{lem:count biMot}, binomial symmetry, and telescoping, 
\begin{equation}\label{eq:one-time-endpoint-telescoping}
    Z_N(0,y)=f_N(y),
    \qquad
    Z_N(x,y)=\sum_{j=0}^{x}f_N(y-x+2j)\qquad\mbox{ for every }y\in\mathbb Z_{\geq0}.
\end{equation}
Note that for every function \(g:\mathbb Z\to\mathbb R\) and every $s\in\llbracket-x,x\rrbracket$,  
\begin{equation}\label{eq:one-time-endpoint-discrete-Taylor}
\begin{split}
& \left|g(y+s)-g(y)-s\left(g(y+1)-g(y)\right)\right|\\
&={\bf1}_{\{s\geq0\}}\left| \sum_{i=0}^{s-1}\sum_{k=0}^{i-1}\Delta^2g(y+k)\right|+
{\bf1}_{\{s<0\}}\left| \sum_{i=1}^{-s}\sum_{k=y-i}^{y-1}\Delta^2g(k)\right|
\leq
 \frac{|s|(|s|+1)}2
 \max_{ r\in\mathbb Z,\,|r-y|\leq x} 
 \left|\Delta^2g(r)\right|.
 \end{split}
\end{equation} 
Since \(\sum_{j=0}^{x}(-x+2j)=0\), averaging \eqref{eq:one-time-endpoint-discrete-Taylor} over \(s\in\{-x,-x+2,\ldots,x\}\) and using
\eqref{eq:one-time-endpoint-telescoping} gives 
\begin{equation}\label{eq:one-time-endpoint-average-error}
 \left|
 \frac{Z_N(x,y)}{x+1}-Z_N(0,y)
 \right|
 \leq
 \frac{x(x+1)}2
 \max_{ r\in\mathbb Z,\,|r-y|\leq x} 
 \left|\Delta^2f_N(r)\right|.
\end{equation}
For \(r\) in the maximum above, let \(d=r+2-y\), then \(|d|\leq x+2\).
By Lemma~\ref{lem:smoothing}, in view of
 $|r+2|\leq y+x+2$, 
 $(r+2)^2=(y+d)^2\geq\frac{y^2}{2}-d^2$,
and, by the hypothesis,
 $\frac{d^2}{N+1}
 \leq
 \frac{(x+2)^2}{N+1}
 \leq 1$ and 
 $\exp\!\left(\frac{d^2}{6(N+1)}\right)\leq2$, we get
\[
 \left|\Delta^2f_N(r)\right|
 \leq
 2\,4^{N+3}\frac{y+x+2}{(N+1)^{5/2}}
 \exp\left(-\frac{y^2}{12(N+1)}\right).
\]
Substituting into \eqref{eq:one-time-endpoint-average-error} concludes the proof of the lemma.
\end{proof}

We have the following proof of Lemma~\ref{lem:TechnicalCombiInput}, which we restate below for clarity. 

\begin{lemma}\label{lem:TechnicalCombiInput AGAIN}
   Set for $N\in\mathbb{Z}_{\geq0}$ and $k\in\mathbb Z$,
\[
y_N(k)=\frac{k}{N+1}\binom{2N+2}{N+k+1} . 
\] Then for each $N\in\mathbb{Z}_{\geq0}$, $j\in\mathbb Z$, and $n\in\llbracket1,\lfloor\varepsilon(N+1)\rfloor\rrbracket$,
\begin{equation}\label{eq:compaAGAIN}
      \frac{|y_N(n+j)-y_N(n)|}{y_N(n)}
\leq
8e^4\,|j| e^{8\varepsilon |j|}\left(\frac1n+\frac{n}{N+1}\right) .
\end{equation}

\end{lemma}

 \begin{proof}
  Assume $N\in\mathbb{Z}_{\geq0}$, $j\in\mathbb Z$, and $n\in\llbracket1,\lfloor\varepsilon(N+1)\rfloor\rrbracket$ throughout this step.
Observe that 
\begin{equation}\label{eq:wer1}
    \frac{|y_N(r+1)-y_N(r)|}{y_N(r)}
=
\frac{|N+1-2r(r+1)|}{r(N+r+2)}
\leq
\frac1r+\frac{4r}{N+1}\qquad\mbox{ for }r\in\llbracket1,N+1\rrbracket.
\end{equation}
Furthermore,
\begin{equation}\label{eq:wer2}
\frac{y_N(r)}{y_N(n)}
=
\frac rn
\prod_{\ell=r}^{n-1}
\frac{N+\ell+2}{N-\ell+1}
\leq
e^{4\varepsilon(n-r)}\qquad\mbox{ for }r\in\llbracket1,n\rrbracket,
\end{equation}
where we use 
\[
\log\frac{N+\ell+2}{N-\ell+1}
=
\log\left(1+\frac{2\ell+1}{N-\ell+1}\right)
\leq
\frac{2\ell+1}{N-\ell+1}
\leq
\frac{4n}{N+1}
\leq
4\varepsilon\qquad\mbox{ for }\ell\leq n-1.
\]
Moreover,
\begin{equation}\label{eq:wer3}
    \frac{y_N(r)}{y_N(n)}
=
\frac rn
\prod_{\ell=n}^{r-1}
\frac{N-\ell+1}{N+\ell+2}
\leq
\frac rn\qquad\mbox{ for }r\geq n. 
\end{equation}
We begin to prove \eqref{eq:compaAGAIN}. When $j=0$ the left side is $0$. We first assume $j>0$. 
When $j<n/2$, by \eqref{eq:wer1} and \eqref{eq:wer3},
\[
\begin{aligned}
\frac{|y_N(n+j)-y_N(n)|}{y_N(n)}
 \leq
\sum_{r=n}^{n+j-1}
\frac{y_N(r)}{y_N(n)}
\frac{|y_N(r+1)-y_N(r)|}{y_N(r)} \leq
\sum_{r=n}^{n+j-1}
\frac rn
\left(\frac1r+\frac{4r}{N+1}\right) \leq
9j\left(\frac1n+\frac{n}{N+1}\right).
\end{aligned}
\]
When $j\geq n/2$,   by  \eqref{eq:wer3}, 
\[
\frac{|y_N(n+j)-y_N(n)|}{y_N(n)}
\leq
2+\frac jn
\leq
5\frac jn.
\]
The above two displayed bounds together imply \eqref{eq:compaAGAIN} when \(j>0\).

We next assume $j<0$. 
When $j>-n/2$, by \eqref{eq:wer1} and \eqref{eq:wer2},
\[
\begin{aligned}
\frac{|y_N(n+j)-y_N(n)|}{y_N(n)} \leq
\sum_{r=n+j}^{n-1}
\frac{y_N(r)}{y_N(n)}
\frac{|y_N(r+1)-y_N(r)|}{y_N(r)} \leq
4|j| e^{4\varepsilon |j|}
\left(\frac1n+\frac{n}{N+1}\right),
\end{aligned}
\]
which implies \eqref{eq:compaAGAIN}. We next assume $j\leq -n/2$. If \(n+j\geq1\), use \eqref{eq:wer2}; if \(n+j=0\), use \(y_N(0)=0\); and if \(n+j<0\),
use the oddness of \(y_N\), together with \eqref{eq:wer2} when \(j\geq -2n\) and \eqref{eq:wer3} when \(j<-2n\). In every case,  
\[
\frac{|y_N(n+j)-y_N(n)|}{y_N(n)}
\leq
1+e^{-4\varepsilon j}-\frac jn
\leq
5\frac{|j|}{n} e^{8\varepsilon |j|}.
\]
The above two displayed bounds together imply \eqref{eq:compaAGAIN} for \(j<0\). We conclude the proof of \eqref{eq:compaAGAIN}.
 \end{proof}

\section{Proof of discrete Gaussian estimates}\label{app:Gauss}

We will now show the deferred discrete Gaussian estimates from Lemma~\ref{lem:discrete Gaussian}, which we restate below for clarity.

\begin{lemma}\label{lem:discrete Gaussian AGAIN}
For each \(s\geq1\), \(u\in[0,1)\), and \(z\in\mathbb Z_{\geq0}\), define
\begin{equation}\label{def:Gaussian AGAIN}
    \mathfrak{X}_{s,u}(z)
=
\sum_{x\geq0}(x+1)u^xe^{-(x-z)^2/s}.
\end{equation}
Then there exist finite constants $C_u$ depending on $u\in[0,1)$; $C_{u,k}$ depending on $u\in[0,1)$, $k\in\mathbb Z_{\geq0}$; and $C_{u,v,k}$ depending on $u,v\in[0,1)$, $k\in\mathbb Z_{\geq0}$, such that, for each fixed $k\in\mathbb Z_{\geq0}$, \(s\geq1\), $x,y\geq2$, \(u,v\in[0,1)\),
\begin{align}
 \label{eq:endpoint-one-gaussian-moment AGAIN}  \sum_{z\geq0}(z+1)^k\mathfrak{X}_{s,u}(z)
&\leq
C_{u,k} s^{(k+1)/2},\\
\label{eq:endpoint-two-gaussian-moment AGAIN} \sum_{z\geq0}(z+1)^k
\mathfrak{X}_{x,u}(z)\mathfrak{X}_{y,v}(z)
&\leq
C_{u,v,k}\left(\frac{xy}{x+y}\right)^{(k+1)/2},\\
\label{eq:endpoint-free-inverse-moment AGAIN} \sum_{z\geq0}\mathfrak X_{x,u}(z)\min\left(1,\frac{z+1}{\sqrt y}\right)
&\leq
C_u\frac{x}{\sqrt{x+y}},\\
\label{eq:endpoint-free-first-moment AGAIN}\sum_{z\geq0}z(z+1)\mathfrak X_{x,u}(z)\min\left(1,\frac{z+1}{\sqrt y}\right)
&\leq
C_u\frac{x^2}{\sqrt{x+y}}.
\end{align} 
\end{lemma}
\begin{proof}
In this proof, for \(j\in\mathbb Z_{\geq0}\) and \(w\in[0,1)\), we denote
\[
A_j
=
4\sum_{\ell=0}^{\infty}(\ell+2)^j e^{-\ell^2} \qquad\mbox{ and }\qquad M_j(w)
=
\sum_{\ell=0}^{\infty}(\ell+1)^j w^\ell,
\]
which are finite real numbers.
For clarity, we split the rest of the proof into four steps. \\

\noindent\textbf{Step 1. } In this step, we prove that 
\begin{equation}\label{eq:shifted-discrete-Gaussian-moment}
\sum_{\ell=0}^{\infty}
(\ell+1)^j e^{-(\ell-\upsigma)^2/t}
\leq
A_j(\upsigma+1)^j t^{(j+1)/2},
\qquad\mbox{ for each }
t\geq1,\,\upsigma\geq0,\,j\in\mathbb Z_{\geq0}.
\end{equation}
Indeed, since $\ell+1 
\leq
(\upsigma+1) (1+|\ell-\upsigma| )$ for each $\ell\geq0$, it suffices to show 
\begin{equation}\label{eq:shifted-discrete-Gaussian}
\sum_{\ell=0}^{\infty}
\left(1+|\ell-\upsigma|\right)^j e^{-(\ell-\upsigma)^2/t}
\leq
A_jt^{(j+1)/2},
\qquad\mbox{ for each }
t\geq1,\,\upsigma\geq0,\,j\in\mathbb Z_{\geq0}.
\end{equation}
Note that, for each $r\in\mathbb Z_{\geq0}$,
 $\#\{
n\in\mathbb Z:
r\sqrt t\leq |n-\upsigma|<(r+1)\sqrt t \}\leq 4\sqrt t$.
For each \(n\) in this set, we have 
 $1+|n-\upsigma|
\leq
1+(r+1)\sqrt t
\leq
(r+2)\sqrt t$
and
$e^{-(n-\upsigma)^2/t}
\leq
e^{-r^2}$.
Therefore
\begin{align*}
\sum_{n\in\mathbb Z}
\left(1+|n-\upsigma|\right)^j e^{-(n-\upsigma)^2/t}
&=
\sum_{r=0}^{\infty}
\sum_{ r\sqrt t\leq |n-\upsigma|<(r+1)\sqrt t}
\left(1+|n-\upsigma|\right)^j e^{-(n-\upsigma)^2/t} \leq
4t^{(j+1)/2}
\sum_{r=0}^{\infty}(r+2)^j e^{-r^2},
\end{align*}
which proves \eqref{eq:shifted-discrete-Gaussian}. We therefore conclude the proof of \eqref{eq:shifted-discrete-Gaussian-moment}.\\

\noindent\textbf{Step 2. } In this step, we prove \eqref{eq:endpoint-one-gaussian-moment AGAIN}.   Using
\eqref{eq:shifted-discrete-Gaussian-moment} with \(t=s\) and \(\upsigma=\ell\),  
\begin{align*}
\sum_{z=0}^{\infty}
(z+1)^k\mathfrak X_{s,u}(z) =
\sum_{\ell=0}^{\infty}(\ell+1)u^\ell
\sum_{z=0}^{\infty}
(z+1)^k e^{-(z-\ell)^2/s} \leq
A_ks^{(k+1)/2}
\sum_{\ell=0}^{\infty}
(\ell+1)^{k+1}u^\ell =
A_kM_{k+1}(u)s^{(k+1)/2},
\end{align*}
hence \eqref{eq:endpoint-one-gaussian-moment AGAIN} holds with
$C_{u,k}
=
A_kM_{k+1}(u)$. \\

\noindent\textbf{Step 3. } In this step, we prove \eqref{eq:endpoint-two-gaussian-moment AGAIN}. Set 
 \[
h
=
\frac{xy}{x+y}\geq1,
\qquad
\xi_{\ell,m}
=
\frac{y\ell+xm}{x+y}\geq0,
\qquad
\ell,m\in\mathbb Z_{\geq0}.
\]
Note that 
\begin{equation}\label{eq:Gaussian-completion-two-endpoints}
\frac{(z-\ell)^2}{x}
+
\frac{(z-m)^2}{y}
=
\frac{(z-\xi_{\ell,m})^2}{h}
+
\frac{(\ell-m)^2}{x+y}.
\end{equation}
Note that 
 $\xi_{\ell,m}+1
\leq
\ell+m+1
\leq
(\ell+1)(m+1)$. Using \eqref{eq:Gaussian-completion-two-endpoints},  \eqref{eq:shifted-discrete-Gaussian-moment}, and
\(e^{-(\ell-m)^2/(x+y)}\leq1\),  
\begin{align*}
&\sum_{z=0}^{\infty}
(z+1)^k
\mathfrak X_{x,u}(z)\mathfrak X_{y,v}(z) =
\sum_{\ell,m=0}^{\infty}
(\ell+1)(m+1)u^\ell v^m
e^{-(\ell-m)^2/(x+y)}
\sum_{z=0}^{\infty}
(z+1)^k e^{-(z-\xi_{\ell,m})^2/h}\\
&\leq
A_kh^{(k+1)/2}
\sum_{\ell,m=0}^{\infty}
(\ell+1)(m+1)u^\ell v^m
(\xi_{\ell,m}+1)^k \leq
A_kh^{(k+1)/2}
\sum_{\ell,m=0}^{\infty}
(\ell+1)^{k+1}(m+1)^{k+1}u^\ell v^m\\
&=
A_kM_{k+1}(u)M_{k+1}(v)
\left(\frac{xy}{x+y}\right)^{(k+1)/2}.
\end{align*}
We conclude the proof of \eqref{eq:endpoint-two-gaussian-moment AGAIN} with
$C_{u,v,k}
=
A_kM_{k+1}(u)M_{k+1}(v)$.\\

\noindent\textbf{Step 4. } In this step, we prove \eqref{eq:endpoint-free-inverse-moment AGAIN} and \eqref{eq:endpoint-free-first-moment AGAIN}.
When \(y\leq x\), using \eqref{eq:endpoint-one-gaussian-moment AGAIN} with \(s=x\) and \(k=0,2\),
\begin{align*}
     \sum_{z=0}^{\infty}
\mathfrak X_{x,u}(z)
\min\left(1,\frac{z+1}{\sqrt y}\right)
 &\leq
C_{u,0}\sqrt x \leq
\sqrt2\,C_{u,0}\frac{x}{\sqrt{x+y}},\\
\sum_{z=0}^{\infty}
z(z+1)\mathfrak X_{x,u}(z)
\min\left(1,\frac{z+1}{\sqrt y}\right) &\leq
\sum_{z=0}^{\infty}
(z+1)^2\mathfrak X_{x,u}(z) \leq
C_{u,2}x^{3/2} \leq
\sqrt2\,C_{u,2}\frac{x^2}{\sqrt{x+y}}.
\end{align*}
When \(y\geq x\), using \eqref{eq:endpoint-one-gaussian-moment AGAIN} with \(s=x\) and \(k=1,3\), 
 \begin{align*}
\sum_{z=0}^{\infty}
\mathfrak X_{x,u}(z)
\min\left(1,\frac{z+1}{\sqrt y}\right)
 &\leq
\frac1{\sqrt y}
\sum_{z=0}^{\infty}
(z+1)\mathfrak X_{x,u}(z) \leq
C_{u,1}\frac{x}{\sqrt y} \leq
\sqrt2\,C_{u,1}\frac{x}{\sqrt{x+y}},\\
\sum_{z=0}^{\infty}
z(z+1)\mathfrak X_{x,u}(z)
\min\left(1,\frac{z+1}{\sqrt y}\right) &\leq
\frac1{\sqrt y}
\sum_{z=0}^{\infty}
(z+1)^3\mathfrak X_{x,u}(z) \leq
C_{u,3}\frac{x^2}{\sqrt y} \leq
\sqrt2\,C_{u,3}\frac{x^2}{\sqrt{x+y}}.
\end{align*} 
We conclude the proof of \eqref{eq:endpoint-free-inverse-moment AGAIN} and \eqref{eq:endpoint-free-first-moment AGAIN} with 
$C_u
=
\sqrt2\max_{0\leq j\leq3}C_{u,j}$.
\end{proof}

\bibliographystyle{goodbibtexstyle.bst}
\bibliography{RWApproximation}

\end{document}